\RequirePackage{fix-cm}
\documentclass[smallextended,envcountsect,envcountsame]{svjour3forpreprint} \mag 1150      
\smartqed  
\usepackage{graphicx}
\usepackage{mathptmx}      
\usepackage{enumerate}
\usepackage{amsmath}
\usepackage{amssymb}
\usepackage{mathrsfs,bm}
\newcommand{\qedhere}{\quad\qed}
\spnewtheorem{strategy}[theorem]{Strategy}{\bf}{\rm}

\numberwithin{equation}{section}

\makeatletter
\newcommand{\rom}[1]{\mbox{\leavevmode\skip@\lastskip\unskip\/%
           \ifdim\skip@=\z@\else\hskip\skip@\fi{\rm{#1}}}}
\makeatother
\newcommand{\Thm}[1]{Theorem~\ref{th:#1}}
\newcommand{\Prop}[1]{Proposition~\ref{prop:#1}}
\newcommand{\Lem}[1]{Lemma~\ref{lem:#1}}
\newcommand{\Cor}[1]{Corollary~\ref{cor:#1}}
\newcommand{\Defn}[1]{Definition~\ref{def:#1}}

\newcommand{\Str}[1]{Strategy~\ref{str:#1}}
\newcommand{\Fig}[1]{Fig.~\ref{fig:#1}}
\newcommand{\Eq}[1]{\rom{\eqref{eq:#1}}}
\renewcommand{\a}{\alpha}\renewcommand{\b}{\beta}
\newcommand{\gm}{\gamma}\newcommand{\dl}{\delta}
\newcommand{\eps}{\epsilon}\newcommand{\zt}{\zeta}
\newcommand{\te}{\theta}
\newcommand{\lm}{\lambda}\newcommand{\kp}{\kappa}
\newcommand{\sg}{\sigma}
\newcommand{\ph}{\varphi}
\newcommand{\om}{\omega}
\newcommand{\Gm}{\Gamma}

\newcommand{\Sg}{\Sigma}
\newcommand{\Ph}{\Phi}

\newcommand{\bff}{\boldsymbol{f}}
\newcommand{\bfF}{\boldsymbol{F}}
\newcommand{\bfg}{\boldsymbol{g}}
\newcommand{\bfh}{\boldsymbol{h}}
\newcommand{\bfr}{\boldsymbol{r}}

\newcommand{\bfone}{\boldsymbol{1}}

\newcommand{\N}{\mathbb{N}}

\newcommand{\R}{\mathbb{R}}
\newcommand{\Z}{\mathbb{Z}}
\newcommand{\cB}{\mathcal{B}}
\newcommand{\cE}{\mathcal{E}}
\newcommand{\cF}{\mathcal{F}}
\newcommand{\cFD}{\mathcal{F}_{\mathrm D}}
\newcommand{\cFH}{\mathcal{F}_{\mathrm H}}
\newcommand{\cG}{\mathcal{G}}
\newcommand{\cH}{\mathcal{H}}
\newcommand{\cK}{\mathcal{K}}

\newcommand{\cM}{\mathcal{M}}
\newcommand{\cN}{\mathcal{N}}
\newcommand{\cP}{\mathcal{P}}
\newcommand{\sC}{\mathscr{C}}

\newcommand{\sG}{\mathscr{G}}
\newcommand{\sH}{\mathscr{H}}
\newcommand{\sL}{\mathscr{L}}
\newcommand{\fM}{\mathfrak{M}}

\newcommand{\la}{\langle}\newcommand{\ra}{\rangle}
\newcommand{\wg}{\wedge}

\newcommand{\maruM}{{\stackrel{\circ}{\smash{\mathcal{M}}\rule{0pt}{1.3ex}}}}
\newcommand{\muessinf}{\mu\!\mathop{\text{\rm -ess\,inf}}}
\newcommand{\muesssup}{\mu\!\mathop{\text{\rm -ess\,sup}}}

\newcommand{\esssup}{\nu\!\mathop{\text{\rm -ess\,sup}}}
\newcommand{\nab}{\nabla}
\newcommand{\Dom}{\mathop{\rm Dom}\nolimits}
\newcommand{\rank}{\mathop{\mathrm{rank }}\nolimits}
\providecommand{\Supp}{\mathrm{Supp}}
\providecommand{\Cp}{\mathop{\rm Cap}\nolimits}
\providecommand{\tr}{\mathop{\rm tr}\nolimits}
\providecommand{\diam}{\mathop{\rm diam}\nolimits}
\providecommand{\PSM}{\mathrm{PSM}}
\providecommand{\Mat}{\mathrm{Mat}}
\providecommand{\Int}{\mathrm{Int}}

\DeclareMathOperator{\mint}{
\mathchoice%
{
\ooalign{%
\ensuremath{%
\begin{picture}(8,8)
\linethickness{0.6pt}
\put(1.25,2.5){\line(1,0){8}}
\end{picture}}%
\crcr
\hss\ensuremath{\displaystyle\int}\hss}
\hspace*{-4.25pt}
}
{
\ooalign{%
\ensuremath{%
\begin{picture}(8,8)
\put(2.0,2.75){\line(1,0){4}}
\end{picture}}%
\crcr
\hss\ensuremath{\textstyle\int}\hss}
\hspace*{-2.63pt}
}
{
\ooalign{%
\ensuremath{%
\begin{picture}(8,8)
\thinlines
\put(2.5,2.0){\line(1,0){3}}
\end{picture}}%
\crcr
\hss\ensuremath{\scriptstyle\int}\hss}
\hspace*{-3.20pt}
}
{
\ooalign{%
\ensuremath{%
\begin{picture}(8,8)
\put(2.98,1.5){\line(1,0){2}}
\end{picture}}%
\crcr
\hss\ensuremath{\scriptscriptstyle\int}\hss}
\hspace*{-3.0pt}
}
}
\journalname{Probability Theory and Related Fields}
\begin{document}\allowdisplaybreaks[3]
\title{Upper estimate of martingale dimension for self-similar fractals\thanks{Research partially supported by KAKENHI (21740094, 24540170).}%
}

\dedication{Dedicated to Professor Leonard Gross on the occasion of his 80th birthday}
\author{Masanori Hino}


\institute{Masanori Hino \at
              Graduate School of Informatics, Kyoto University, Kyoto 606-8501, Japan \\
              Tel.: +81-75-753-3378\\
              Fax: +81-75-753-3381\\
              \email{hino@i.kyoto-u.ac.jp}
}

\date{Received: date / Accepted: date}

\maketitle

\begin{abstract}
We study upper estimates of the martingale dimension $d_\mathrm{m}$ of diffusion processes associated with strong local Dirichlet forms.
By applying a general strategy to self-similar Dirichlet forms on self-similar fractals, we prove that $d_\mathrm{m}=1$ for natural diffusions on post-critically finite self-similar sets and that $d_\mathrm{m}$ is dominated by the spectral dimension for the Brownian motion on Sierpinski carpets.
%
\end{abstract}
\setcounter{tocdepth}{2}\tableofcontents
\section{Introduction}
Studies on the structure of stochastic processes through the space of martingales associated with them date back to the 1960s.
As seen from Meyer's decomposition theorem for example,
martingales are one of the suitable concepts for understanding the randomness of stochastic processes.
In the framework of general Markov processes, Motoo and Watanabe~\cite{MW64} proved that, for a class $\cM$ of martingale additive functionals, there exists a kind of basis $\{x_n\}$ of $\cM$ such that every element in $\cM$ can be represented as a sum of stochastic integrals based on $\{x_n\}$ and a purely discontinuous part.
This is a generalization of the study by Ventcel'~\cite{Ve61}, wherein the Brownian motion on $\R^d$ was considered.
We term the cardinality of the basis as the martingale dimension. (The precise definition is discussed in Section~2.)
Related general theories are found in some articles such as those by Kunita and Watanabe~\cite{KW67} and Cram\'er~\cite{Cr64}.
Later, Davis and Varaiya~\cite{DV74} introduced the concept of {multiplicity of filtration} on filtered probability spaces as an abstract generalization. 
A vast amount of literature is now available on the study of filtrations from various directions by M.~Yor, M.~\'Emery, M.~T.~Barlow, E.~A.~Perkins, B.~Tsirelson, and many others.
In this article, we focus on the quantitative estimate of martingale dimensions associated with symmetric diffusion processes on state spaces that do not necessarily have smooth structures, in particular, on self-similar fractals.

The martingale dimension of typical examples, such as the Brownian motion on a $d$-dimensional complete Riemannian manifold, is $d$.
This number can be informally interpreted as the number of ``independent noises'' included in the process.
When the underlying space does not have a differential structure, it is not easy to determine or even to provide estimates of the martingale dimension.
The first result in this direction is due to Kusuoka~\cite{Ku89}, who considered the martingale dimension $d_\mathrm{m}$ with respect to additive functionals (AF-martingale dimension) and proved that $d_\mathrm{m}=1$ for the Brownian motion on the $d$-dimensional standard Sierpinski gasket SG$_d$ (see \Fig{fig1}) for every $d$. 
\begin{figure}[t]
\includegraphics[width=0.18\hsize]{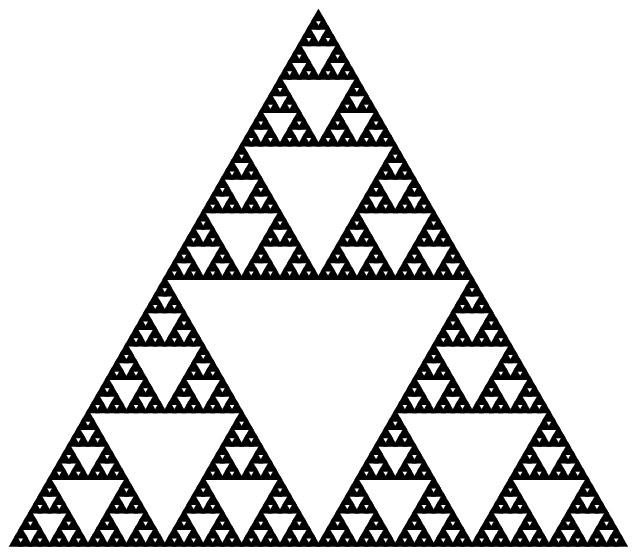}\quad
\includegraphics[width=0.17\hsize]{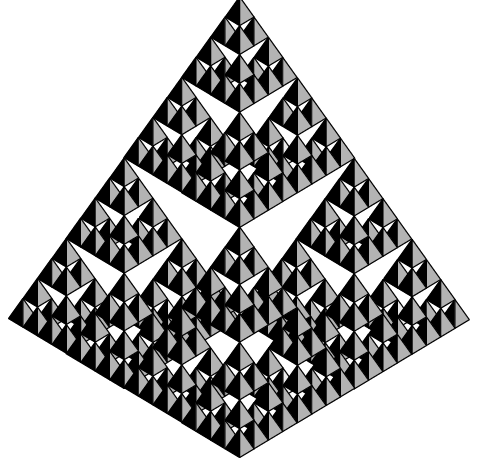}\quad
\includegraphics[width=0.14\hsize]{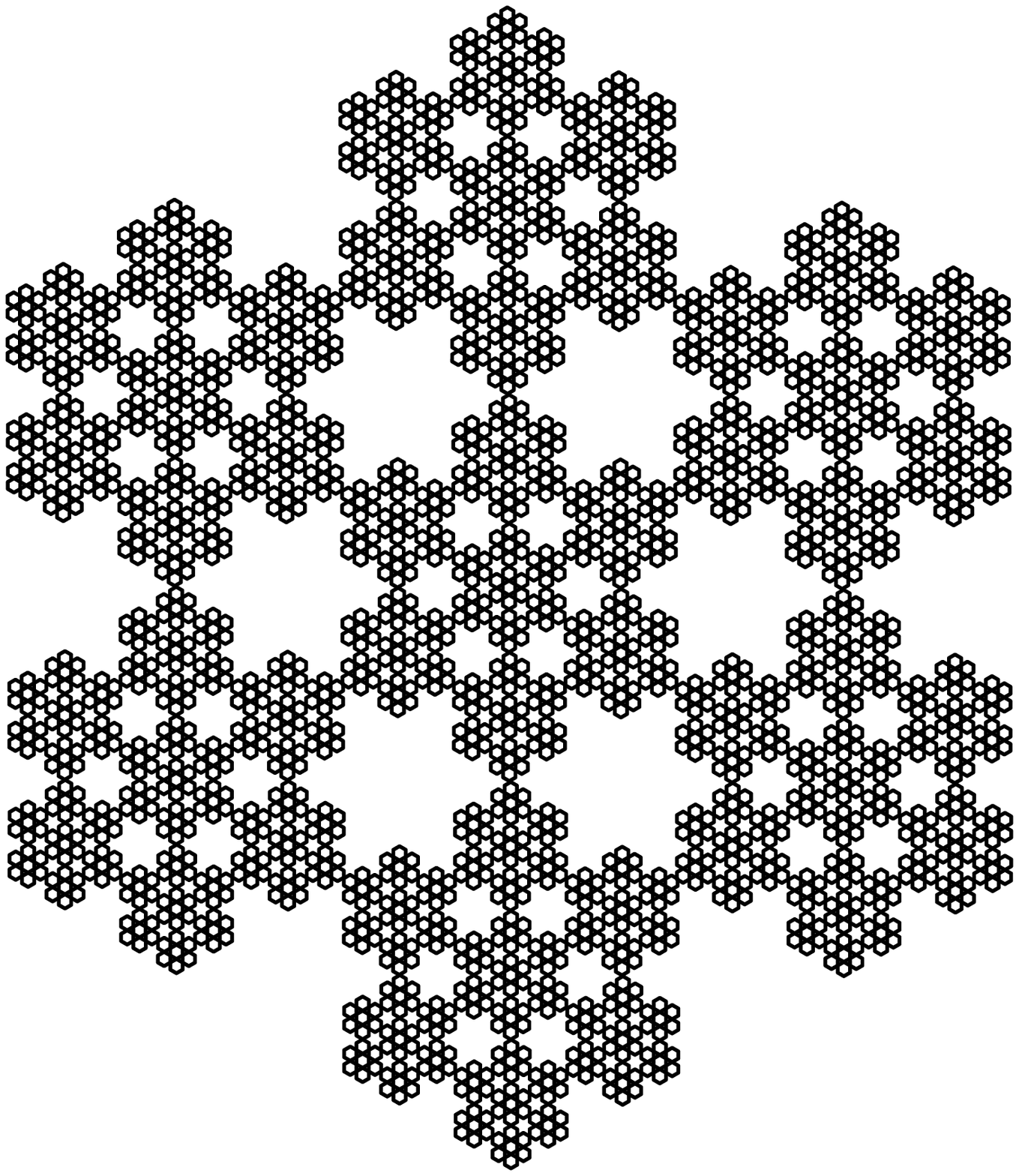}\quad
\includegraphics[width=0.17\hsize]{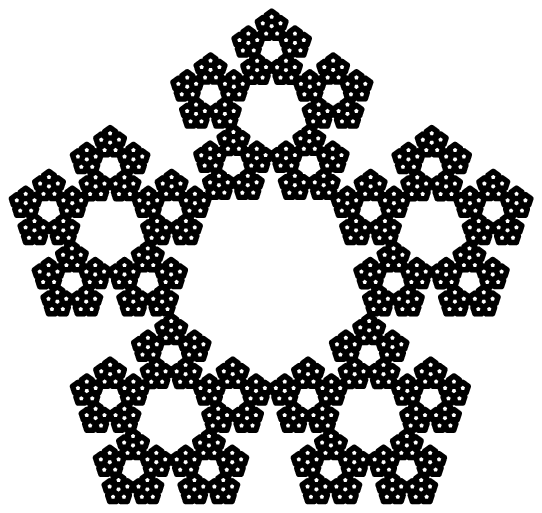}\quad
\includegraphics[width=0.24\hsize]{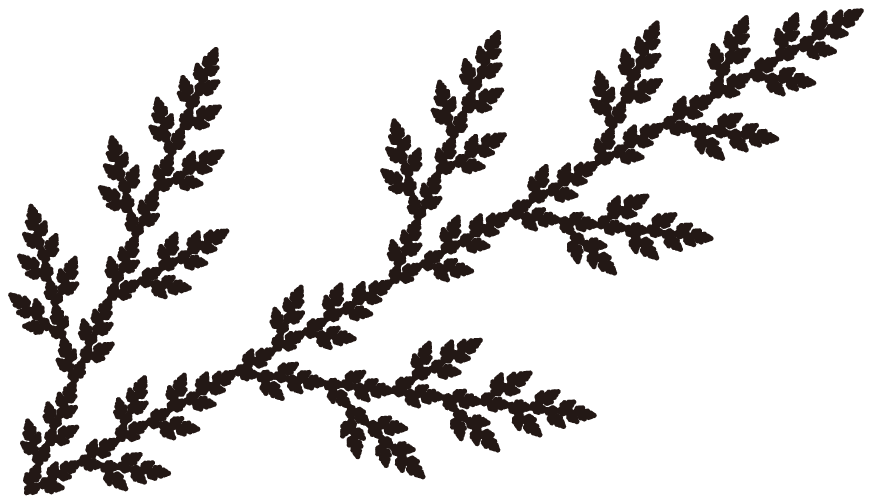}
\caption{SG$_2$, SG$_3$, and some other p.c.f.~self-similar sets}
\label{fig:fig1}
\end{figure}
This was an unexpected result because the Hausdorff dimension of SG$_d$ is $\log (d+1)/\log 2$, which is arbitrarily large when $d$ becomes larger.
This result was generalized in \cite{Hi08,Hi10} to natural self-similar symmetric diffusion processes on post-critically finite (p.\,c.\,f.) self-similar sets (see \Fig{fig1}) satisfying certain technical conditions, with the same conclusion.
The proofs heavily rely on the facts that the fractal sets under consideration are finitely ramified (that is, they can be disconnected by removing finitely many points) and that the Dirichlet form associated with the diffusion is described by infinite random products of a finite number of matrices.
No further results have yet been obtained in this direction.
Thus, the following questions naturally arise:
\begin{itemize}
\item What about the martingale dimensions of infinitely ramified fractals such as\break Sierpinski carpets?
\item In general, are there any relations between $d_\mathrm{m}$ and other kinds of dimensions?
\end{itemize}
In this paper, we provide partial answers to these questions; we prove that the AF-martingale dimension $d_\mathrm{m}$ of the Brownian motion on (generalized) Sierpinski carpets (\Fig{fig2}) are dominated by the spectral dimension $d_\mathrm{s}$. 
\begin{figure}
\includegraphics[width=0.18\hsize]{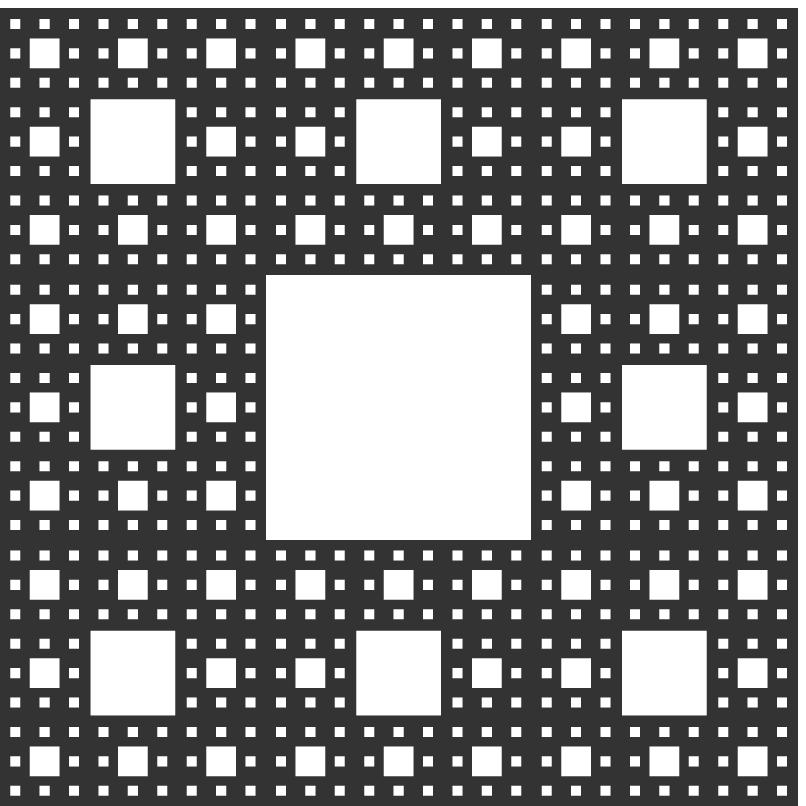}\qquad 
\includegraphics[width=0.18\hsize]{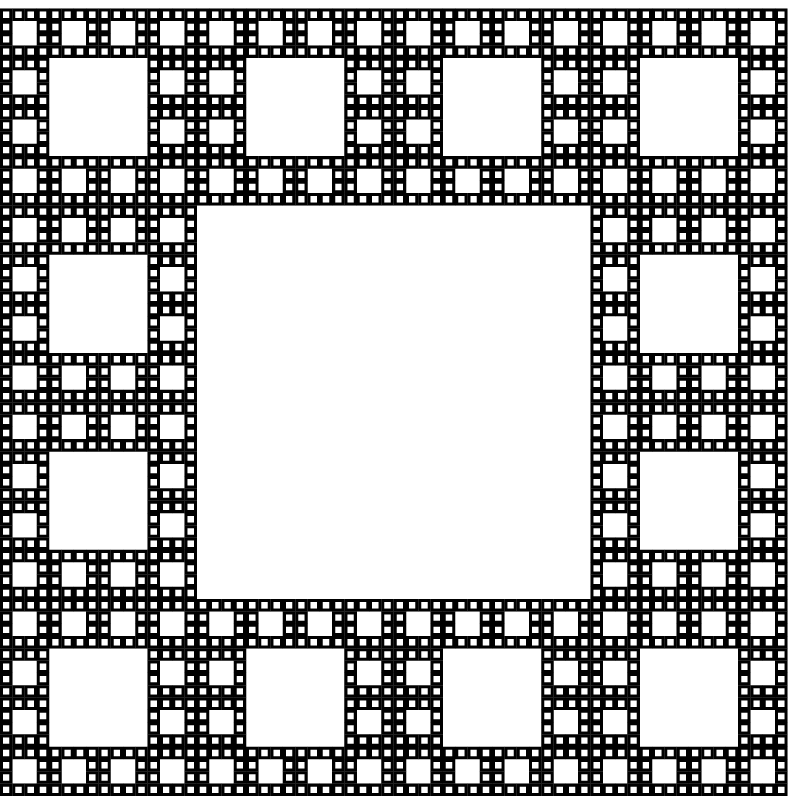}\qquad 
\includegraphics[width=0.17\hsize]{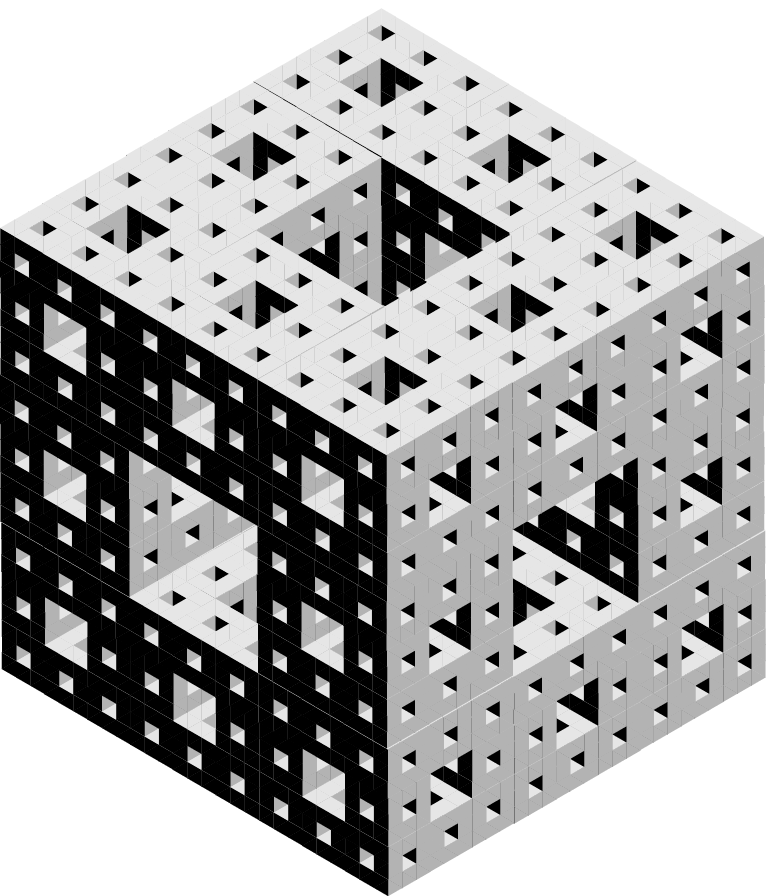}
\caption{Examples of (generalized) Sierpinski carpets}
\label{fig:fig2}
\end{figure}
In particular, if the process is point recurrent (that is, if $d_\mathrm{s}<2$), then $d_\mathrm{m}=1$.
This is the first time that nontrivial estimates of martingale dimensions for infinitely ramified fractals have been obtained.
The proof is based on the analytic characterization of $d_\mathrm{m}$ in terms of the index of the associated Dirichlet form that was developed in \cite{Hi10}, and new arguments for the estimate of the index in general frameworks, in which some harmonic maps play the crucial roles.
This method is also applicable to p.\,c.\,f.\ self-similar sets, which enables us to remove the technical assumptions in \cite{Hi08} and conclude that $d_\mathrm{m}=1$.
In \cite{Hi08}, we had to exclude Hata's tree-like set (the rightmost figure of \Fig{fig1}) because of some technical restrictions such as the condition that every ``boundary point'' had to be a fixed point of one of the maps defining the self-similar set; this example was discussed individually in \cite{Hi10}.

One of the main ingredients of the proof is the construction of a special harmonic map from the fractal to the Euclidean space~$\R^d$, which makes it possible to use certain properties of the classical energy form on $\R^d$.
For this purpose, we use a method analogous to the blowup argument in geometric measure theory.
Although we presently require the self-similar structure of the state space for this argument, we expect the relation $1\le d_\mathrm{m}\le d_\mathrm{s}$ to be true for more general metric measure spaces as well.

This article is organized as follows.
In Section~2, we review the concepts of the index of strong local regular Dirichlet forms and the AF-martingale dimension $d_\mathrm{m}$ of the associated diffusion processes under a general setting.
In Section~3, we develop some tools for the estimation of $d_\mathrm{m}$ in the general framework.
In Section~4.1, we discuss self-similar Dirichlet forms on self-similar fractals and study some properties on the energy measures as a preparation for the proof of the main results.
In Section~4.2, we treat p.\,c.\,f.\ self-similar sets and prove that $d_\mathrm{m}=1$ with respect to natural self-similar diffusions. 
This subsection is also regarded as a warm-up for the analysis on Sierpinski carpets, which is technically more involved.
In Section~4.3, we consider Sierpinski carpets and prove the inequality $1\le d_\mathrm{m}\le d_\mathrm{s}$, putting forth two technical propositions.
These propositions are proved in Section~5.

Hereafter, $c_{i.j}$ denotes a positive constant appearing in Section $i$ that does not play important roles in the arguments.
\section{Martingale dimension of the diffusion processes associated with strong local Dirichlet forms}
In this section, we review a part of the theory of Dirichlet forms and the concept of martingale dimensions, following \cite{FOT,Hi08,Hi10}.
We assume that the state space $K$ is a locally compact, separable, and metrizable space.
We denote the Borel $\sg$-field of $K$ by $\cB(K)$.
Let  $C(K)$ denote the set of all continuous real-valued functions on $K$, and $C_c(K)$, the set of all functions in $C(K)$ with compact support.
Let $\mu$ be a positive Radon measure on $K$ with full support.
For $1\le p\le \infty$, $L^p(K,\mu)$ denotes the real $L^p$-space on the measure space $(K,\cB(K),\mu)$ with norm $\|\cdot\|_{L^p(K,\mu)}$.
The inner product of $L^2(K,\mu)$ is denoted by $(\cdot,\cdot)_{L^2(K,\mu)}$.
Suppose that we are given a symmetric regular Dirichlet form $(\cE,\cF)$ on $L^2(K,\mu)$.
For $\a\in\R$ and $f,g\in\cF$, we define $\cE_\a(f,g)=\cE(f,g)+\a(f,g)_{L^2(K,\mu)}$.
The space $\cF$ becomes a Hilbert space under inner product $(f,g)_\cF:=\cE_1(f,g)$.
Hereafter, the topology of $\cF$ is always considered as that derived from norm $\|\cdot\|_\cF:=(\cdot,\cdot)_\cF^{1/2}$.
We write $\cE(f)$ and $\cE_\a(f)$ instead of $\cE(f,f)$ and $\cE_\a(f,f)$ for simplicity.
The set of all bounded functions in $\cF$ is denoted by $\cF_b$.
The following is a basic fact.
\begin{proposition}[cf.\ {\cite[Theorem~1.4.2]{FOT}}]\label{prop:cEfg}
Let $f,g\in\cF_b$. Then, $fg\in \cF_b$ and
\[
  \cE(fg)^{1/2}\le \cE(f)^{1/2}\|g\|_{L^\infty(K,\mu)}+\cE(g)^{1/2}\|f\|_{L^\infty(K,\mu)}.
\]
\end{proposition}
Let us review the theory of additive functionals associated with $(\cE,\cF)$, following \cite[Chapter~5]{FOT}.
The capacity $\Cp$ associated with $(\cE,\cF)$ is defined as
\[
  \Cp(U)=\inf\{\cE_1(f)\mid f\in\cF\mbox{ and }f\ge1\ \mu\mbox{-a.e.\ on }U\}
\]
if $U$ is an open subset of $K$, and
\[
  \Cp(B)=\inf\{\Cp(U)\mid \mbox{$U$ is open and $U\supset B$}\}
\]
for general subsets $B$ of $K$.
A subset $B$ of $K$ with $\Cp(B)=0$ is called an {\em exceptional set}.
A statement depending on $x\in K$ is said to hold for {\em q.e.} (quasi-every) $x$ if the set of $x$ for which the statement is not true is an exceptional set. 
A real valued function $u$ defined q.e.\ on $K$ is called {\em quasi-continuous} if for any $\eps>0$, there exists an open subset $U$ of $K$ such that $\Cp(U)<\eps$ and $u|_{K\setminus U}$ is continuous. From \cite[Theorem~2.1.3]{FOT}, every $f\in \cF$ has a quasi-continuous modification $\tilde f$ in the sense that $f=\tilde f$ $\mu$-a.e.\ and $\tilde f$ is quasi-continuous.

For a $\mu$-measurable function $u$, the support of the measure $|u|\cdot\mu$ is denoted by $\Supp[u]$.
Hereafter, we consider only the case that $(\cE,\cF)$ is {\em strong local}, that is, the following property holds:
\begin{quote}
 If $u,v\in\cF$, $\Supp[u]$ and $\Supp[v]$ are compact, and $v$ is constant on a neighborhood of $\Supp[u]$, then $\cE(u,v)=0$.
\end{quote}
  From the general theory of regular Dirichlet forms, we can construct a diffusion process $\{X_t\}$ on $K_\Delta$ defined on a filtered probability space $(\Omega,\cF_\infty,P,\{P_x\}_{x\in K_\Delta},\{\cF_t\}_{t\in[0,\infty)})$ associated with $(\cE,\cF)$.
Here, $K_\Delta=K\cup\{\Delta\}$ is a one-point compactification of $K$ and $\{\cF_t\}_{t\in[0,\infty)}$ is the minimum completed admissible filtration.
Any numerical function $f$ on $K$ extends to $K_\Delta$ by letting $f(\Delta)=0$.
We denote by $E_x$ the expectation with respect to $P_x$ for $x\in K$.
The relationship between $\{X_t\}$ and $(\cE,\cF)$ is explained in such a way that the operator $f\mapsto E_\cdot[f(X_t)]$ produces the semigroup associated with $(\cE,\cF)$.
  We may assume that for each $t\in[0,\infty)$, there exists a shift operator $\theta_t\colon \Omega\to\Omega$ that satisfies $X_s\circ\theta_t=X_{s+t}$ for all $s\ge0$.
  We denote the life time of $\{X_t(\om)\}_{t\in[0,\infty)}$ by $\zt(\om)$.
  A $[-\infty,+\infty]$-valued function $A_t(\om)$, $t\in[0,\infty)$, $\om\in\Omega$, is referred to as an {\em additive functional} if the following conditions hold:
  \begin{itemize}
  \item $A_t(\cdot)$ is $\cF_t$-measurable for each $t\ge0$;
  \item There exist a set $\Lambda\in\cF_\infty$ and an exceptional set $N\subset K$ such that $P_x(\Lambda)=1$ for all $x\in K\setminus N$ and $\theta_t\Lambda\subset\Lambda$ for all $t>0$; moreover, for each $\om\in\Lambda$, $A_\cdot(\om)$ is right continuous and has the left limit on $[0,\zt(\om))$, $A_0(\om)=0$, $|A_t(\om)|<\infty$ for all $t<\zt(\om)$, $A_t(\om)=A_{\zt(\om)}(\om)$ for $t\ge\zt(\om)$, and
  \[
    A_{t+s}(\om)=A_s(\om)+A_t(\theta_s \om)
    \quad \mbox{for every }t,s\ge0.
  \]
  \end{itemize}
  The sets $\Lambda$ and $N$ referred to above are called a {\em defining set} and an {\em exceptional set} of the additive functional $A$, respectively.
  A finite (resp.\ continuous) additive functional is defined as an additive functional such that $|A_\cdot(\om)|<\infty$ (resp.\ $A_\cdot(\om)$ is continuous) on $[0,\infty)$ for $\om\in\Lambda$.
  A $[0,+\infty]$-valued continuous additive functional is referred to as a positive continuous additive functional.
   From \cite[Theorems~5.1.3 and 5.1.4]{FOT}, for each positive continuous additive functional $A$, there exists a unique measure~$\mu_A$ on $K$ (termed the Revuz measure of $A$) such that the following identity holds for any $t>0$ and nonnegative Borel functions $f$ and $h$ on $K$:
  \[
  \int_K E_x\left[\int_0^t f(X_s)\,dA_s\right]h(x)\,\mu(dx)
  =\int_0^t \int_K E_x\left[h(X_s)\right]f(x)\,\mu_A(dx)\,ds.
  \]
  Further, if two positive continuous additive functionals $A^{(1)}$ and $A^{(2)}$ have the same Revuz measures, then $A^{(1)}$ and $A^{(2)}$ coincide in the sense that, for any $t>0$, $P_x(A_t^{(1)}=A_t^{(2)})=1$ for q.e.\,$x\in K$.  
  Let $P_\mu$ be a measure on $\Omega$ defined as $P_\mu(\cdot)=\int_K P_x(\cdot)\,\mu(dx)$.
  Let $E_\mu$ denote the integration with respect to $P_\mu$.
  We define the energy $e(A)$ of additive functional $A$ as
  $
  e(A)=\lim_{t\to0}(2t)^{-1}E_\mu[A_t^2]
  $
  if the limit exists.
  
  Let $\cM$ be the space of martingale additive functionals of $\{X_t\}$ that is defined as
  \[
  \cM=\left\{M\;\vrule\;\;\parbox{0.8\textwidth}{%
  $M$ is a finite additive functional such that $M_\cdot(\om)$ is right continuous
  and has a left limit on $[0,\infty)$ for $\om$ in a defining set of $M$, and for
  each $t>0$, $E_x[M_t^2]<\infty$ and $E_x[M_t]=0$ for q.e.\  $x\in K$}\;\right\}.
  \]
  Due to the assumption that $(\cE,\cF)$ is strong local, every $M\in\cM$ is in fact a continuous additive functional (cf.~\cite[Lemma~5.5.1~(ii)]{FOT}).
  Each $M\in\cM$ admits a positive continuous additive functional $\la M\ra$ referred to as the quadratic variation associated with $M$, which satisfies
  $
  E_x[\la M\ra_t]=E_x[M_t^2], \ t>0$ for q.e.\ $x\in K$,
  and the following equation holds:
  $  e(M)=\mu_{\la M\ra}(K)/2$. 
  We set $\maruM=\{M\in\cM\mid e(M)<\infty\}$.
  Then, $\maruM$ is a Hilbert space with inner product $e(M,L):=(e(M+L)-e(M)-e(L))/2$ (see \cite[Theorem~5.2.1]{FOT}). 
  For $M,L\in \maruM$, we set $\mu_{\la M,L\ra}=(\mu_{\la M+L\ra}-\mu_{\la M\ra}-\mu_{\la L\ra})/2$.
  For $M\in\maruM$ and $f\in L^2(K,\mu_{\la M\ra})$, we can define the stochastic integral $f\bullet M$ (cf.~\cite[Theorem~5.6.1]{FOT}), which is a unique element of $\maruM$ such that
  $
  e(f\bullet M,L)=(1/2)\int_Kf(x)\,\mu_{\la M,L\ra}(dx)$
  for all $L\in\maruM$.
  If $f\in C_c(K)$, we may write $\int_0^\cdot f(X_t)\,dM_t$ for $f\bullet M$ since $(f\bullet M)_t=\int_0^t f(X_s)\,dM_s$, $t>0$, $P_x$-a.e.\ for q.e.\,$x\in K$ (cf.~\cite[Lemma~5.6.2]{FOT}).

Let $\Z_+$ denote the set of all nonnegative integers.
\begin{definition}[cf.~\cite{Hi08}]\label{def:AF}
The {\em AF-martingale dimension} of $\{X_t\}$ (or of $(\cE,\cF)$) is defined as the smallest number $p$ in $\Z_+$ satisfying the following: There exists a sequence $\{M^{(k)}\}_{k=1}^p$ in $\maruM$ such that every $M\in\maruM$ has a stochastic integral representation
\[
  M_t=\sum_{k=1}^p(h_k\bullet M^{(k)})_t,\quad t>0,\ P_x\mbox{-a.e.\ for q.e.\,}x,
\]
where $h_k\in L^2(K,\mu_{\la M^{(k)}\ra})$ for each $k=1,\dotsc,p$.
If such $p$ does not exist, the AF-martingale dimension is defined as $+\infty$.
\end{definition}
\begin{remark}\label{rem:martingale}
In the definition above, AF is an abbreviation of ``additive functional.''
We can also consider another version of martingale dimensions for general (not necessarily symmetric) diffusion processes as follows.
Let 
\[
\fM=\left\{M=\{M_t\}_{t\in[0,\infty)}\;\vrule\;\; \parbox{0.6\textwidth}{$M_0=0$ and $M$ is a square-integrable martingale with respect to $P_x$ for all $x\in K$}\right\}.
\]
For $M\in\fM$, denote its quadratic variation process by $\la M\ra$ and define the space $L(\la M\ra)$ as the family of all progressively measurable processes $\ph(t,\om)$ such that $E_x\left[\int_0^t\ph(s)^2\,d\la M\ra_s\right]<\infty$ for all $t>0$ and $x\in K$.
The martingale dimension of $\fM$ is defined as the smallest number $q$ satisfying the following:
There exists $M^{(1)},\dots, M^{(q)}\in\fM$ such that every $M\in\fM$ can be expressed as
$M_t=\sum_{k=1}^q\int_0^t \ph_k(s)\,dM^{(k)}_s$ $P_x$-a.e.\,$x$ for all $x\in K$, where $\ph_k\in L(\la M^{(k)}\ra)$, $k=1,\dots,q$, and the integral above is interpreted as the usual stochastic integral with respect to martingales.
Let us observe the relation between these two concepts.
Suppose that $\{X_t\}$ with $(\Omega,\cF_\infty,P,\{P_x\}_{x\in K_\Delta},\{\cF_t\}_{t\in[0,\infty)})$ is a diffusion process on $K$ with symmetrizing measure $\mu$ and has an associated regular Dirichlet form $(\cE,\cF)$ on $L^2(K,\mu)$.
For $\a>0$ and a bounded Borel measurable function $f$ in $L^2(K,\mu)$, denote the $\a$-order resolvent $E.\left[\int_0^\infty e^{-\a t}f(X_t)\,dt\right]$ by $G_\a f$, and set
\[
  M_t^{f,\a}=(G_\a f)(X_t)-(G_\a f)(X_0)-\int_0^t\left(\a(G_\a f)(X_s)-f(X_s)\right)ds,
  \quad t>0.
\]
Then, $M_\cdot^{f,\a}$ belongs to $\maruM\cap\fM$.
Moreover, concerning the space
\[\hat\fM=\{M_\cdot^{f,\a}\mid \a>0,\ \text{$f$ is a bounded Borel measurable function in $L^2(K,\mu)$}\},\]
the linear span of $\{h\bullet M\mid M\in\hat\fM,\ h\in C_c(K)\}$ is dense in $\maruM$ from \cite[Lemma~5.6.3]{FOT} and the linear span of 
$\left\{\int_0^\cdot \ph(s)\,dM_s\;\vrule\;\;M\in\hat\fM,\ \ph\in L(\la M\ra)\right\}$ is dense in $\fM$ with respect to the natural topology from \cite[Theorem~4.2]{KW67}.
These facts strongly suggest that the two martingale dimensions coincide, although the author does not have a proof.
In this article, we consider only AF-martingale dimensions and often omit ``AF'' from the terminology hereafter.
\end{remark}

We review the analytic representation of the AF-martingale dimension.
First, we introduce the concept of energy measures of functions in $\cF$, which is defined for (not necessarily strong local) regular Dirichlet forms.
For each $f\in\cF$, a positive finite Borel measure $\nu_f$ on $K$ is defined as follows~(cf.~\cite[Section~3.2]{FOT})\footnote{In \cite{FOT}, symbol $\mu_{\la f\ra}$ is used in place of $\nu_f$.}.
When $f$ is bounded, $\nu_f$ is characterized by the identity
\[
  \int_K \ph\,d\nu_f=2\cE(f\ph,f)-\cE(\ph,f^2)\quad
  \mbox{for all }\ph\in\cF\cap C_c(K).
\]
By using the inequality
\begin{equation}\label{eq:energy}
  \left|\sqrt{\nu_f(B)}-\sqrt{\nu_g(B)}\right|^2\le \nu_{f-g}(B)
  \le 2\cE(f-g),
  \quad B\in\cB(K),\ f,g\in\cF_b
\end{equation}
(cf.~\cite[p.~111, and (3.2.13) and (3.2.14) in p.~110]{FOT}), for any $f\in\cF$, we can define a finite Borel measure $\nu_f$ by $\nu_f(B)=\lim_{n\to\infty}\nu_{f_n}(B)$ for $B\in\cB(K)$, where $\{f_n\}_{n=1}^\infty$ is a sequence in $\cF_b$ such that $f_n$ converges to $f$ in $\cF$. 
Then, equation \Eq{energy} still holds true for any $f,g\in\cF$.
The measure $\nu_f$ is called the {\em energy measure} of $f$.
For $f,g\in\cF$, the mutual energy measure $\nu_{f,g}$, which is a signed Borel measure on $K$, is defined as
$
  \nu_{f,g}=(\nu_{f+g}-\nu_{f}-\nu_g)/2$.
Then, $\nu_{f,f}=\nu_f$ and $\nu_{f,g}$ is bilinear in $f$ and $g$ (cf.~\cite[p.~111]{FOT}).
We also have the following inequalities: for $f,g\in\cF$ and $B\in\cB(K)$,
\begin{align}
\label{eq:schwarz}
|\nu_{f,g}(B)|&\le\sqrt{\nu_f(B)}\sqrt{\nu_g(B)},\\
\label{eq:triangular}
\sqrt{\nu_{f+g}(B)}&\le\sqrt{\nu_{f}(B)}+\sqrt{\nu_{g}(B)}.
\end{align}
Moreover, for $f,g\in\cF$ and Borel measurable functions $h_1,h_2$ on $K$,
\begin{equation}\label{eq:KW}
\left|\int_K h_1 h_2\,d\nu_{f,g}\right|\le\left(\int_K h_1^2\,d\nu_f\right)^{1/2}\left(\int_K h_2^2\,d\nu_g\right)^{1/2}
\end{equation}
as long as the integral on the left-hand side makes sense.
This is proved as follows: If $h_1$ and $h_2$ are simple functions, \Eq{KW} follows from \Eq{schwarz} and the Schwarz inequality.
By the limiting argument, \Eq{KW} holds for general $h_1$ and $h_2$.

Under the assumption that $(\cE,\cF)$ is strong local, we have an identity
\begin{equation}\label{eq:323}
\cE(f)=\nu_f(K)/2,\qquad f\in\cF
\end{equation}
(cf.~\cite[Lemma~3.2.3]{FOT}) and the following derivation property.
\begin{theorem}[{cf.\ \cite[Theorem~3.2.2]{FOT}}]\label{th:derivation}
Let $f_1,\dots,f_m$, and $g$ be elements in $\cF$, and $\ph\in C_b^1(\R^m)$ satisfy $\ph(0,\dots,0)=0$.
Then, $u:=\ph(f_1,\dots,f_m)$ belongs to $\cF$ and
\[
  d\nu_{u,g}=\sum_{i=1}^m\frac{\partial\ph}{\partial x_i}(\tilde f_1,\dots,\tilde f_m)\,d\nu_{f_i,g}.
\]
Here, $C^1_b(\R^m)$ denotes the set of all bounded $C^1$-functions on $\R^m$ with bounded derivatives, and $\tilde f_i$ denotes a quasi-continuous modification of $f_i$.
\end{theorem}
We note that the underlying measure $\mu$ does not play an important role with regard to energy measures.

For two $\sg$-finite (or signed) Borel measures $\mu_1$ and $\mu_2$ on $K$, we write $\mu_1\ll\mu_2$ if $\mu_1$ is absolutely continuous with respect to $\mu_2$.
\begin{definition}[cf.~\cite{Hi10}]
A $\sg$-finite Borel measure $\nu$ on $K$ is called a {\em minimal energy-dominant measure} of $(\cE,\cF)$ if the following two conditions are satisfied.
\begin{enumerate}
\item (Domination) For every $f\in \cF$, $\nu_f\ll \nu$.
\item (Minimality) If another $\sg$-finite Borel measure $\nu'$ on $K$ satisfies condition~(i) with $\nu$ replaced by $\nu'$, then $\nu\ll\nu'$.
\end{enumerate}
\end{definition}
By definition, two minimal energy-dominant measures are mutually absolutely continuous.
In fact, a minimal energy-dominant measure is realized by an energy measure as follows.
\begin{proposition}[see {\cite[Proposition~2.7]{Hi10}}]\label{prop:em}
The set of all functions $g\in\cF$ such that $\nu_g$ is a minimal energy-dominant measure of $(\cE,\cF)$ is dense in $\cF$.\end{proposition}
Fix a minimal energy-dominant measure $\nu$ of $(\cE,\cF)$.
From \Eq{schwarz}, $\nu_{f,g}\ll\nu$ for $f,g\in\cF$, so that we can consider the Radon--Nikodym derivative $d\nu_{f,g}/d\nu$.

Let $d\in \N$. We denote $\underbrace{\cF\times\cdots\times\cF}_{d}$ by $\cF^d$ and equip it with the product topology.
\begin{definition}\label{def:basic}
For $\bff=(f_1,\dots,f_d)\in\cF^d$, we define 
\begin{equation}
\cE(\bff)=\frac1d\sum_{i=1}^d\cE(f_i),\quad
\nu_{\bff}=\frac1d\sum_{i=1}^d\nu_{f_i}
\end{equation}
and
\begin{equation}\label{eq:Phif}
\Phi_{\bff}=\begin{cases}\left(\left.\dfrac{d\nu_{f_i,f_j}}{d\nu}\right/\dfrac{d\nu_{\bff}}{d\nu}\right)_{i,j=1}^d& \mbox{on }\left\{\dfrac{d\nu_{\bff}}{d\nu}>0\right\},\\
O& \mbox{on }\left\{\dfrac{d\nu_{\bff}}{d\nu}=0\right\}.
\end{cases}
\end{equation}
\end{definition}
Note that $\Phi_{\bff}$ is a function defined $\nu$-a.e.\ on $K$, taking values in the set of all symmetric and nonnegative-definite matrices of order $d$. 
\begin{lemma}\label{lem:trivial}
For $\bff\in\cF^d$, $\Phi_{\bff}=\left({d\nu_{f_i,f_j}}\big/{d\nu_{\bff}}\right)_{i,j=1}^d$ $\nu_{\bff}$-a.e.
\end{lemma}
\begin{proof}
This is evident from the definition of $\Phi_{\bff}$, by taking into account that $\nu_{f_i,f_j}\ll \nu_{\bff}$ from \Eq{schwarz}.
\qed\end{proof}
\begin{lemma}\label{lem:energyineq}
\begin{enumerate}
\item Let $\{f^{(n)}\}_{n=1}^\infty$ and $\{g^{(n)}\}_{n=1}^\infty$ be sequences in $\cF$ and $f^{(n)}\to f$ and $g^{(n)}\to g$ in $\cF$ as $n\to\infty$.
Then, $d\nu_{f^{(n)},g^{(n)}}/d\nu$ converges to $d\nu_{f,g}/d\nu$ in $L^1(K,\nu)$.
\item Suppose that a sequence $\{\bff^{(n)}\}_{n=1}^\infty$ in $\cF^d$ converges to $\bff$ in $\cF^d$.
Then, there exists a subsequence $\{\bff^{(n')}\}$ such that $\Phi_{\bff^{(n')}}(x)$ converges to $\Phi_{\bff}(x)$ for $\nu_{\bff}$-a.e.\,$x$.
\end{enumerate}
\end{lemma}
\begin{proof}
Assertion (i) is proved in \cite[Lemma~2.5]{Hi10}.
We prove (ii). From (i), we can take a subsequence $\{\bff^{(n')}\}$ such that $\Phi_{\bff^{(n')}}$ converges to $\Phi_{\bff}$  $\nu$-a.e.\ on $\left\{{d\nu_{\bff}}/{d\nu}>0\right\}$. This implies the assertion.
\qed\end{proof}
The following definition is taken from \cite{Hi10}, which is a natural generalization of the concept due to Kusuoka~\cite{Ku89,Ku93}.
\begin{definition}
The {\em index} $p$ of $(\cE,\cF)$ is defined as 
the smallest number satisfying the following:
For any $N\in\N$ and any $f_1,\dots,f_N\in\cF$, 
\[
 \rank \left(\frac{d\nu_{f_i,f_j}}{d\nu}(x)\right)_{i,j=1}^N\le p\quad \mbox{for }\nu\mbox{-a.e.\,}x.
\]
If such a number does not exist, the index is defined as $+\infty$.
\end{definition}
It is evident that this definition is independent of the choice of $\nu$.
\begin{proposition}[{cf.~\cite[Proposition~2.10]{Hi10}}]\label{prop:rank}
Let $\{f_i\}_{i=1}^\infty$ be a sequence of functions in $\cF$ such that the linear span of $\{f_i\}_{i=1}^\infty$ is dense in $\cF$.
Denote the Radon--Nikodym derivative $d\nu_{f_i,f_j}/d\nu$ by $Z^{i,j}$ for $i,j\in\N$.
Then, the index of $(\cE,\cF)$ is described as 
$
 \esssup_{x\in K}\;\sup_{N\in\N}\;\rank \left(Z^{i,j}(x)\right)_{i,j=1}^N$.
\end{proposition}
We remark the following fact.
\begin{proposition}[{cf.~\cite[Proposition~2.11]{Hi10}}]\label{prop:nontrivial}
The index is $0$ if and only if $\cE(f)=0$ for every $f\in\cF$.
\end{proposition}
The following theorem is a natural generalization of  \cite[Theorem~6.12]{Ku93} and underlies the estimate of martingale dimensions from the next section.
\begin{theorem}[see {\cite[Theorem~3.4]{Hi10}}]\label{th:index}
The index of $(\cE,\cF)$ coincides with the AF-martingale dimension of $\{X_t\}$.
\end{theorem}
\section{Strategy for upper estimate of martingale dimension}
In this section, we develop some tools for the estimation of AF-martingale dimensions under a general framework.
We keep the notations in the previous section.

First, we introduce the concept of harmonic functions.
We fix a closed subset $K^\partial$ of $K$. This set is regarded as a boundary of $K$.
We define 
\begin{align}
\cF_0=\{f\in\cF\mid \Supp[f]\cap K^\partial=\emptyset\}
\text{~~and~~}
\cFD=\{f\in\cF\mid \tilde f=0 \mbox{ q.e.\ on }K^\partial\},\label{eq:cF0cFD}
\end{align}
where $\tilde f$ is a quasi-continuous modification of $f$.
We remark the following:
\begin{proposition}[cf.\ {\cite[Corollary~2.3.1]{FOT}}]
The closure of $\cF_0$ in $\cF$ is equal to $\cFD$.
\end{proposition}
An element $h\in\cF$ is called {\em harmonic} if $\cE(h)\le \cE(h+f)$ for all $f\in\cFD$.
The set of all harmonic functions are denoted by $\cH$.
The following is a standard fact and its proof is omitted (cf.~\cite[Lemma~3.6]{Hi05}).
\begin{lemma}\label{lem:harmonic}
For $h\in\cF$, the following are equivalent.
\begin{enumerate}
\item $h\in\cH$.
\item For every $f\in \cFD$, $\cE(h,f)=0$.
\item For every $f\in \cF_0$, $\cE(h,f)=0$.
\end{enumerate}
Moreover, $\cH$ is a closed subspace of $\cF$.
\end{lemma}
Let $d\in \N$. We denote $\underbrace{\cH\times\cdots\times\cH}_{d}$ by $\cH^d$, which is considered as a closed subspace of $\cF^d$.
The Lebesgue measure on $\R^d$ is denoted by $\sL^d$. 
The symbol ``$dx$'' is also used if there is no ambiguity.
For $r\in\N$ and $p\ge1$, $W^{r,p}(\R^d)$ denotes the classical $(r,p)$-Sobolev space on $\R^d$.
Hereafter, for $f\in \cF$, $\tilde f$ denotes a quasi-continuous Borel modification of $f$.
The symbol $\tilde\bff$ corresponding to $\bff\in\cF^d$ is similarly interpreted.
In general, for a measurable map $\bfF\colon X\to Y$ and a measure $m_X$ on $X$, $\bfF_*m_X$ denotes the induced measure of $m_X$ by $\bfF$.

Given $d\in\N$, we consider the following conditions.
\begin{enumerate}[(U')$_d$]
\item[(U)$_d$]
 There exists $\bfh=(h_1,\dots,h_d)\in \cH^d$ such that the following hold:
\begin{enumerate}
\item $\nu_{\bfh}(K)>0$;
\item $\Phi_{\bfh}(x)$ is the identity matrix for $\nu_{\bfh}$-a.e.\,$x\in K$;
\item $\tilde\bfh_*\nu_{\bfh}\ll\sL^d$.
\end{enumerate}
\item[(U')$_d$]
 There exists $\bfh=(h_1,\dots,h_d)\in \cH^d$ such that the following hold:
\begin{enumerate}
\item $\nu_{\bfh}(K)>0$;
\item $\Phi_{\bfh}(x)$ is the identity matrix for $\nu_{\bfh}$-a.e.\,$x\in K$;
\item $\tilde\bfh_*\nu_{\bfh}\ll\sL^d$, and the density $\rho=d(\tilde\bfh_*\nu_{\bfh})/d\sL^d$ is dominated by a certain nonnegative function $\xi$ with $\sqrt\xi\in W^{1,2}(\R^d)$, in that $\rho\le\xi$ $\sL^d$-a.e.
\end{enumerate}
\end{enumerate}
Note that $\tilde\bfh_*\nu_{\bfh}$ does not depend on the choice of $\tilde\bfh$ since $\nu_{\bfh}$ does not charge any sets of zero capacity.

The following three claims are crucial for the estimate of the martingale dimension, the proofs of which are provided later.
\begin{lemma}\label{lem:renormalize}
Let $\bfh=(h_1,\dots,h_d)\in\cH^d$.
Suppose that $\nu_{\bfh}(K)>0$ and $\Phi_{\bfh}(x)=L$ for $\nu_{\bfh}$-a.e.\,$x$ for some symmetric and positive-definite matrix $L$ of order $d$ that is independent of $x$.
Then, there exists $\bfh'=(h'_1,\dots,h_d')\in\cH^d$ such that $\nu_{\bfh'}(K)>0$ and $\Phi_{\bfh'}(x)$ is the identity matrix for $\nu_{\bfh'}$-a.e.\,$x$.
In particular, $\nu_{h'_i}=\nu_{\bfh'}$ for every $i=1,\dots,d$.
\end{lemma}
\begin{proposition}\label{prop:W12}
Assume that $\bfh=(h_1,\dots,h_d)\in\cH^d$ and $\Phi_{\bfh}(x)$ is the identity matrix for $\nu_{\bfh}$-a.e.\,$x$.
Take a bounded function $f$ from $\cFD$.
Then, the induced measure of $\tilde f^2\nu_{\bfh}$ by $\tilde\bfh\colon K\to\R^d$, denoted by $\tilde\bfh_*(\tilde f^2\nu_{\bfh})$, is absolutely continuous with respect to $\sL^d$, and its density $\xi:=d(\tilde\bfh_*(\tilde f^2\nu_{\bfh}))/d\sL^d$ satisfies $\sqrt\xi\in W^{1,2}(\R^d)$.
\end{proposition}
\begin{theorem}\label{th:general}
We assume that $\mu(K)<\infty$ and $1\in\cF$. Then, the following hold for $d\in\N$.
\begin{enumerate}
\item Assume condition~{\rm(U)$_d$}. Moreover, if $\Cp(\{x\})>0$ for every $x\in K$, then $d=1$.
\item Assume condition~{\rm(U')$_d$}. Moreover, suppose that the Sobolev inequality holds for some $d_\mathrm{s}>2$ and $c_{3.1}>0$\textrm{:}
\begin{equation}\label{eq:sobolev}
  \|f\|_{L^{2d_\mathrm{s}/(d_\mathrm{s}-2)}(K,\mu)}^2\le c_{3.1}\cE_1(f),
  \qquad f\in\cF.
\end{equation}
Then, $d\le d_\mathrm{s}$.
\end{enumerate}
\end{theorem}
In virtue of these results, the strategy to provide upper estimates of martingale dimensions is summarized as follows.
\begin{strategy}\label{str:str}
The following is a strategy for upper estimates of the AF-martingale dimensions $d_\mathrm{m}$.
\begin{enumerate}[Step 1:]
\item[Step 0:] Take an arbitrary $d\in\N$ such that $d\le d_\mathrm{m}$.
\item Find $\bfh\in\cH^d$ such that $\nu_{\bfh}(K)>0$ and $\Phi_{\bfh}(x)=L$ for $\nu_{\bfh}$-a.e.\,$x$ for some symmetric positive-definite matrix $L$ of order $d$.
We may assume that $L$ is the identity matrix as seen from \Lem{renormalize}.
\item By using the result of Step~1 and \Prop{W12} if necessary, find (possibly different) $\bfh\in\cH^d$ such that condition~(U)$_d$ or (U')$_d$ holds true in addition.
\item Then, under the assumptions of \Thm{general}, we obtain an estimate of $d_\mathrm{m}$.
\end{enumerate}
\end{strategy}
In Sections~4 and 5, we consider self-similar fractals as $K$ and show that the above procedure can be realized.
In the remainder of this section, we prove \Lem{renormalize}, \Prop{W12}, and \Thm{general}.
\begin{theopargself}\begin{proof}[of \Lem{renormalize}]
There exists an orthogonal matrix $U=(u_{ij})_{i,j=1}^d$ such that
\[
^tULU=\begin{pmatrix}\lm_1&&\raisebox{-1.3ex}[0pt][0pt]{\phantom{$\lm_d$}\llap{\smash{\huge$0$}}}\\&\ddots\\\mbox{\rlap{\smash{\huge$0$}}\quad}&&\lm_d\end{pmatrix}\quad\mbox{with }\lm_i>0,\ i=1,\dots,d.
\]
Define $\hat\bfh=(\hat h_1,\dots,\hat h_d)\in\cH^d$ by
$
  \hat h_i=\sum_{k=1}^d u_{ki}h_k$ for $i=1,\dots,d$.
Then,
$
  \nu_{\hat h_i,\hat h_j}=\sum_{k,l=1}^d u_{ki}u_{lj}\,\nu_{h_k,h_l}$
for $i,j=1,\dots,d$,
which implies that
$
 \Bigl(\frac{d\nu_{\hat h_i,\hat h_j}}{d\nu_{\bfh}}(x)\Bigr)_{i,j=1}^d
 ={}^tULU$
for $\nu_{\bfh}\mbox{-a.e.}\,x$.
In particular, $\nu_{\hat h_i}=\lm_i\nu_{\bfh}$ for $i=1,\dots,d$.
Define $\bfh'=(h'_1,\dots,h'_d)\in\cH^d$ as $h'_i=\lm_i^{-1/2}\hat h_i$ for $i=1,\dots,d$.
Then, $\nu_{h'_i}=\nu_{\bfh}$ for all $i$, which implies that $\nu_{\bfh'}=\nu_{\bfh}$.
Moreover, for $i,j=1,\dots, d$,
\[
\frac{d\nu_{h'_i,h'_j}}{d\nu_{\bfh'}}(x)
=\frac{d\nu_{h'_i,h'_j}}{d\nu_{\bfh}}(x)
=\dl_{ij}
\quad \mbox{for }\nu_{\bfh'}\mbox{-a.e.}\,x,
\]
where $\dl_{ij}$ denotes the Kronecker delta.
Therefore, $\Phi_{\bfh'}(x)$ is the identity matrix for $\nu_{\bfh'}$-a.e.\,$x$.
\qed\end{proof}\end{theopargself}
\begin{remark}
As seen from the proof, two $\cH^d$'s in the statement of \Lem{renormalize} can be replaced by $\cF^d$.
\end{remark}
Before proving \Prop{W12}, we remark the following result.
\begin{proposition}[Energy image density property]\label{prop:edp}
For $f\in\cF$, the measure $\tilde f_*\nu_f$ on $\R$ is absolutely continuous with respect to $\sL^1$.
In particular, $\nu_f$ has no atoms.
\end{proposition}
This proposition is proved in \cite[Theorem~I.7.1.1]{BH} when the strong local Dirichlet form is given by the integration of the carr\'e du champ operator. The proof of \Prop{edp} is provided along the same way, which has been mentioned already, e.g., in \cite{HN06,Hi08}.
See also \cite[Theorem~4.3.8]{CF} for the short proof.

For $\bff\in\cF^d$ with $d\ge2$, the absolute continuity of the measure $\tilde \bff_*\nu_{\bff}$ on $\R^d$ is not expected in general. Some studies on sufficient conditions are found in \cite{BH}.
In \Prop{W12}, we consider a rather special situation that implies a better smoothness.
How to find functions that meet this situation is the main problem that is discussed in the next section. 

\begin{theopargself}\begin{proof}[of \Prop{W12}]
Take an arbitrary $\ph\in C_b^1(\R^d)$ with $\ph(0,\dots,0)=0$ and define $g=\ph\circ\tilde\bfh$.
From \Lem{harmonic} and \Thm{derivation}, for each $i=1,\dots, d$, we have
\begin{align}\label{eq:ibp}
0&= 2\cE(gf^2,h_i)\qquad\mbox{(since $gf^2\in\cFD$ and $h_i\in \cH$)}\notag\\
&=\int_K d\nu_{gf^2,h_i}
=\int_K g\,d\nu_{f^2,h_i}+\int_K\tilde f^2\,d\nu_{g,h_i}\notag\\
&=\int_K g\,d\nu_{f^2,h_i}+\int_K\tilde f^2\frac{\partial \ph}{\partial x_i}(\tilde\bfh)\,d\nu_{h_i}
\qquad\mbox{(since $\nu_{h_i,h_j}=0$ if $i\ne j$)}\notag\\
&=\int_{\R^d}\ph\,d(\tilde\bfh_*\nu_{f^2,h_i})+\int_{\R^d}\frac{\partial \ph}{\partial x_i}\,d(\tilde\bfh_*(\tilde f^2\nu_{\bfh})).
\end{align}
The last equality follows from the change of variable formula and 
$\nu_{h_i}=\nu_{\bfh}$.

Let $\kappa=\tilde\bfh_*(\tilde f^2\nu_{\bfh})$ and $\kappa_i=\tilde\bfh_*\nu_{f^2,h_i}$ for $i=1,\dots, d$.
From \Eq{ibp},
\begin{equation}\label{eq:ibp2}
  \int_{\R^d}\frac{\partial \ph}{\partial x_i}\,d\kappa=-\int_{\R^d}\ph\,d\kappa_i,\quad
  i=1,\dots,d,
\end{equation}
for $\ph\in C_b^1(\R^d)$ with $\ph(0,\dots,0)=0$.
Since 
$
\kp_i(\R^d)=\nu_{f^2,h_i}(K)=2\cE(f^2,h_i)=0$,
identity~\Eq{ibp2} holds for all $\ph\in C_b^1(\R^d)$.
Therefore, in the distribution sense, $\frac{\partial}{\partial x_i}\kp=\kp_i$ for $i=1,\dots,d$.
This implies that $\kp\ll \sL^d$, e.g., from \cite[pp.~196--197]{Ma} or \cite[Lemma~I.7.2.2.1]{BH}.
Denote the Radon--Nikodym derivative $d\kp/d\sL^d$ by $\xi$.
Then, \Eq{ibp2} can be interpreted as $\partial \xi/\partial x_i=\kappa_i$ in the distribution sense for $i=1,\dots,d$.

Now, for any $B\in\cB(\R^d)$, from \Thm{derivation} and \Eq{KW},
\begin{align}\label{eq:Kschwarz}
|\kappa_i(B)|&=\left|\int_{\tilde\bfh^{-1}(B)} 2\tilde f\,d\nu_{f,h_i}\right|
\le 2\left(\int_{\tilde\bfh^{-1}(B)} \tilde f^2\,d\nu_{h_i}\right)^{1/2}\left(\int_{\tilde\bfh^{-1}(B)} d\nu_{f}\right)^{1/2}\notag\\*
&=2\kappa(B)^{1/2}(\tilde\bfh_*\nu_f)(B)^{1/2}.
\end{align}
Therefore, $\kappa_i\ll\kappa$, in particular, $\kappa_i\ll\sL^d$.
This implies that $\xi$ belongs to the Sobolev space $W^{1,1}(\R^d)$.
Let $d\kappa_i/d\sL^d$ be denoted by $\xi_i$.
From \Eq{Kschwarz} and Theorems~1 and 3 in \cite[Section~1.6]{EG}, we have
\[
  \left|\xi_i\right|\le2\xi^{1/2}\biggl(\frac{d(\tilde\bfh_*\nu_f)_{\mathrm{ac}}}{d\sL^d}\biggr)^{1/2} \quad \sL^d\mbox{-a.e.},
\]
where $(\tilde\bfh_*\nu_f)_{\mathrm{ac}}$ denotes the absolutely continuous part in the Lebesgue decomposition of $\tilde\bfh_*\nu_f$.
For $\eps>0$, let $\gm_\eps(t)=\sqrt{t+\eps}-\sqrt{\eps}$, $t\ge0$.
Then,
\begin{align*}
\left(\frac{\partial (\gm_\eps(\xi))}{\partial x_i}\right)^2
=\biggl(\frac1{2\sqrt{\xi+\eps}}\frac{\partial\xi}{\partial x_i}\biggr)^2
=\frac{\xi_i^2}{4(\xi+\eps)}
\le\frac{d(\tilde\bfh_*\nu_f)_{\mathrm{ac}}}{d\sL^d},
\end{align*}
which implies that
\[
\int_{\R^d}\left(\frac{\partial (\gm_\eps(\xi))}{\partial x_i}\right)^2\,dx\le (\tilde\bfh_*\nu_f)_{\mathrm{ac}}(\R^d)
\le \nu_f(K)=2\cE(f)<\infty.
\]
Since $\gm_\eps(t)\nearrow\sqrt{t}$ as $\eps\searrow0$, $\partial\sqrt\xi\big/\partial x_i$ belongs to $L^2(\R^d,dx)$.
This implies that $\sqrt\xi\in W^{1,2}(\R^d)$.
\qed\end{proof}\end{theopargself}
For the proof of \Thm{general}, we need several claims.
Let $\rho$ and $\xi$ be Lebesgue measurable functions on $\R^d$ such that $0\le\rho\le \xi$ $\sL^d$-a.e.\ and $\sqrt\xi\in W^{1,2}(\R^d)$.
Define a bilinear form $Q^\rho$ on $L^2(\R^d,(\rho+1)\,dx)$ by
\[
  Q^\rho(u,v)=\int_{\R^d}(\nab u,\nab v)_{\R^d}(\rho+1)\,dx,\quad
  u,v\in C_c^1(\R^d),
\]
where $(\cdot,\cdot)_{\R^d}$ denotes the standard inner product on $\R^d$ and $C_c^1(\R^d)=C^1(\R^d)\cap C_c(\R^d)$.
It is easy to see that $(Q^\rho,C_c^1(\R^d))$ is closable in $L^2(\R^d,(\rho+1)\,dx)$, and its closure, denoted by $(Q^\rho,\Dom(Q^\rho))$, is a regular Dirichlet form on $L^2(\R^d,(\rho+1)\,dx)$.
We also define the standard regular Dirichlet form $(Q,W^{1,2}(\R^d))$ on $L^2(\R^d,dx)$ as
\[
Q(u,v)=\int_{\R^d}(\nab u,\nab v)_{\R^d}\,dx,\quad
  u,v\in W^{1,2}(\R^d).
\]
The capacities associated with $Q^\rho$ and $Q$ are denoted by $\Cp^\rho$ and $\Cp^{1,2}$, respectively.
For $x\in\R^d$ and $r>0$, we define
\[
B(x,r)=\{y\in\R^d\mid |x-y|_{\R^d}<r\}
\quad\mbox{and}\quad
\bar B(x,r)=\{y\in\R^d\mid |x-y|_{\R^d}\le r\},
\]
where $|\cdot|_{\R^d}$ denotes the Euclidean norm on $\R^d$.
In general, for a measure space $(X,\lm)$ and a subset $E$ with $\lm(E)<\infty$, the normalized integral $\lm(E)^{-1}\int_E\cdots d\lm$ is denoted by $\mint_E\cdots d\lm$.
\begin{lemma}\label{lem:bound}
For $\Cp^{1,2}$-q.e.\,$x\in\R^d$,
$\sup_{r>0}\mint_{B(x,r)}\rho(y)\,dy<\infty$.
\end{lemma}
\begin{proof}
Take a quasi-continuous modification of $\sqrt \xi$ with respect to $\Cp^{1,2}$, which is denoted by the same symbol. 
We may assume that $0\le\xi(x)<\infty$ for every $x\in \R^d$.

From \cite[Theorem~6.2.1]{AH}, there exists a $\Cp^{1,2}$-null set $B$ of $\R^d$ such that, for every $x\in \R^d\setminus B$,
\[
\lim_{r\to0}\mint_{B(x,r)}\left|\sqrt{\xi(y)}-\sqrt{\xi(x)}\right|^2dy=0.
\]
For $x\in \R^d\setminus B$, take $r_0>0$ such that $\sup_{0<r<r_0}\mint_{B(x,r)}\bigl|\sqrt{\xi(y)}-\sqrt{\xi(x)}\bigr|^2dy\le1$. 
Then, for $r\in(0,r_0)$,
\begin{align*}
\left(\mint_{B(x,r)}\xi(y)\,dy\right)^{1/2}
&\le\left(\mint_{B(x,r)}\left|\sqrt{\xi(y)}-\sqrt{\xi(x)}\right|^2dy\right)^{1/2}
+\left(\mint_{B(x,r)}\xi(x)\,dy\right)^{1/2}\\
&\le 1+\sqrt{\xi(x)}.
\end{align*}
Since $0\le \rho\le \xi$ $\sL^d$-a.e., we obtain that $\sup_{0<r<r_0}\mint_{B(x,r)}\rho(y)\,dy<\infty$.
We also have 
\[
\sup_{r\ge r_0}\mint_{B(x,r)}\rho(y)\,dy\le \sL^d(B(x,r_0))^{-1}\int_{\R^d}\rho(y)\,dy<\infty.\qedhere
\]
\end{proof}
\begin{proposition}\label{prop:equivalence}
Let $B\subset\R^d$.
Then, $\Cp^\rho(B)=0$ if and only if $\Cp^{1,2}(B)=0$.
\end{proposition}
\begin{proof}
We define a Dirichlet form $(Q^\xi,\Dom(Q^\xi))$ on $L^2(\R^d,(\xi+1)dx)$ and its capacity $\Cp^\xi$, just as $(Q^\rho,\Dom(Q^\rho))$ and $\Cp^\rho$, with $\rho$ replaced by $\xi$.
Then, from the result in \cite{RZ94} (see also Theorem~3.3, Theorem~3.6, and the subsequent Remark~(iv) in \cite{Eb}), $\Dom(Q^\xi)$ is characterized as follows:
\[
\Dom(Q^\xi)
=\left\{ u\;\vrule\;\;
\parbox{0.75\textwidth}{$u\in L^2(\R^d,(\xi+1)\,dx)$, and for every $i=1,\dots,d$, ${\partial u}/{\partial x_i}$ exists in the distribution sense and 
${\partial u}/{\partial x_i}\in L^2(\R^d,(\xi+1)\,dx)$}\;
\right\}.
\]
Since $\Cp^{1,2}(B)\le\Cp^\rho(B)\le\Cp^\xi(B)$ for $B\subset \R^d$, it suffices to show that any $\Cp^{1,2}$-null set $B$ satisfies $\Cp^\xi(B)=0$.
Let $g(x)=\log(\sqrt{\xi(x)}+1)$. 
Since $\sqrt\xi\in W^{1,2}(\R^d)$, $\xi\in L^{1}(\R^d,dx)\cap L^{1+\dl}(\R^d,dx)$ for some $\dl\in(0,2]$ from the Sobolev imbedding theorem.
There exists $c_{3.2}>0$ such that $\log(t+1)\le c_{3.2}t^{\dl/2}\wg t$ for $t\ge0$.
Then,
\[
\int_{\R^d}g^2(\xi+1)\,dx
\le \int_{\R^d}(c_{3.2}^2\xi^\dl\wg\xi)(\xi+1)\,dx
\le\int_{\R^d}(c_{3.2}^2\xi^{1+\dl}+\xi)\,dx <\infty
\]
and
\[
\int_{\R^d}|\nab g|_{\R^d}^2(\xi+1)\,dx
\le \int_{\R^d}\frac{\bigl|\nab \sqrt{\xi}\bigr|_{\R^d}^2}{(\sqrt{\xi}+1)^2}\cdot(\xi+1)\,dx
\le\int_{\R^d}\bigl|\nab \sqrt{\xi}\bigr|_{\R^d}^2\,dx
<\infty.
\]
Thus, $g$ belongs to $\Dom(Q^\xi)$.
We denote by $\tilde g$ the quasi-continuous modification of $g$ with respect to $(Q^\xi,\Dom(Q^\xi))$.
Let $\eps>0$. There exist some $b>0$ and an open set $U_1$ of $\R^d$ such that $U_1\supset\{\tilde g>b\}$ and $\Cp^\xi(U_1)<\eps$.
Note that 
$
\{\tilde g>b\}=\{\xi>(e^b-1)^2\}$
up to $\sL^d$-null set.
Take an open set $U_2$ of $\R^d$ such that $U_2\supset B$ and $\Cp^{1,2}(U_2)<e^{-2b}\eps$.
We denote the $1$-equilibrium potential of $U_1$ with respect to $(Q^\xi,\Dom(Q^\xi))$ by $e_1$, and that of $U_2$ with respect to $(Q,W^{1,2}(\R^d))$ by $e_2$. 

Define $f(x)=e_1(x)\vee e_2(x)$ for $x\in\R^d$. 
Then, $f\in W^{1,2}(\R^d)$ and $f=1$ on $U_2\,(\supset B)$.
Since $\{e_1\ge e_2\}\supset U_1\supset \{\xi> (e^b-1)^2\}$
up to $\sL^d$-null set, we have
\begin{align*}
&\int_{\R^d}(|\nab f|_{\R^d}^2+f^2)(\xi+1)\,dx\\
&= \int_{\{e_1\ge e_2\}}(|\nab e_1|_{\R^d}^2+e_1^2)(\xi+1)\,dx
+\int_{\{e_1< e_2\}}(|\nab e_2|_{\R^d}^2+e_2^2)(\xi+1)\,dx
\\
&\qquad\mbox{(e.g., from \cite[Proposition~I.7.1.4]{BH})}\\
&\le \eps
+\int_{\{\xi\le (e^b-1)^2\}}(|\nab e_2|_{\R^d}^2+e_2^2)(\xi+1)\,dx\\
&\le \eps+\{(e^b-1)^2+1\}e^{-2b}\eps
\le 2\eps.
\end{align*}
Therefore, $f\in \Dom(Q^\xi)$ and $\Cp^\xi(B)\le 2\eps$.
Since $\eps>0$ is arbitrary, we obtain that $\Cp^\xi(B)=0$.
\qed\end{proof}
\begin{lemma}\label{lem:RN}
Let $\kp$ be a positive Radon measure on $\R^d$.
Then, $\kp(A)=0$ with
\[
A=\left\{x\in\R^d\Bigm| \liminf_{r\searrow0}\frac{\kp(\bar B(x,r))}{r^d}=0\right\}.
\]
\end{lemma}
We remark that the set $A$ above is Borel measurable.
Indeed, 
\[
2^{-d}\cdot\frac{\kp(\bar B(x,2^{-k}))}{(2^{-k})^d}\le\frac{\kp(\bar B(x,r))}{r^d}\le 2^d\cdot\frac{\kp(\bar B(x,2^{-k+1}))}{(2^{-k+1})^d}
\quad \text{if }2^{-k}\le r<2^{-k+1},
\]
which implies
$
A=\{x\in\R^d\mid \liminf_{k\to\infty,\,k\in\N}{\kp(\bar B(x,2^{-k}))}/{(2^{-k})^d}=0\}$.
It is easy to see that the right-hand side is a Borel set.
\begin{theopargself}
\begin{proof}[of \Lem{RN}]
For $n\in\N$, let
$
A_{n}=A\cap B(0,n)\in\cB(\R^d)
$.
From Lemma~1 in \cite[Section~1.6]{EG}, for any $\a>0$, $\kp(A_{n})\le \a\sL^d(A_{n})$.
By letting $\a\to0$, we have $\kp(A_{n})=0$.
This implies the assertion.
\qed\end{proof}
\end{theopargself}
For the proof of the next lemma, let us recall the definition of the Hausdorff (outer) measure on $\R^d$.
Let $A\subset \R^d$ and $s>0$.
For $\dl>0$, define
\[
  \sH^s_\dl(A)=\inf\left\{\sum_{j=1}^\infty v_s\left(\frac{\diam C_j}2\right)^s\;\vrule\;
  A\subset \bigcup_{j=1}^\infty C_j,\ \diam C_j\le\dl\right\},
\]
where $v_s=\pi^{s/2}/\Gm(s/2+1)$.
Then, the $s$-dimensional Hausdorff measure of $A$, denoted by $\sH^s(A)$, is defined as $\sH^s(A)=\lim_{\dl\to0}\sH^s_\dl(A)$.
\begin{lemma}\label{lem:RN2}
Suppose that $d\ge2$ and $\kp$ is a positive Radon measure on $\R^d$.
Then,
\begin{equation}\label{eq:RN2}
\lim_{r\searrow0}\frac{\kp(\bar B(x,r))}{r^{d-2}}=0
\quad\text{for $\Cp^{1,2}$-q.e.\,$x\in\R^d$.}
\end{equation}
\end{lemma}
\begin{proof}
When $d=2$, \Eq{RN2} is equivalent to the statement that the set $\{x\in\R^2\mid \kp(\{x\})>0\}$ is $\Cp^{1,2}$-null, which is true because the cardinality of this set is at most countable.

We suppose $d\ge3$.
Let $n\in\N$ and set
\[
A_n=\biggl\{x\in\R^d\biggm|\limsup_{r\searrow0}\frac{\kp(\bar B(x,r))}{r^{d-2}}>\frac1n\biggr\}\cap B(0,n).
\]
For $\dl>0$, set
\[
  \sG_\dl=\biggl\{B\biggm| B=\bar B(x,r),\ x\in A_n,\ 0<r<\dl,\ B\subset B(0,n),\ \frac{\kp(B)}{r^{d-2}}>\frac1n\biggr\}.
\]
Then, for each $x\in A_n$, $\inf\{r\mid \bar B(x,r)\in\sG_\dl\}=0$.
From Vitali's covering lemma, there exists an at most countable family $\{\bar B(x_j,r_j)\}_j$ of disjoint balls in $\sG_\dl$ such that $(A_n\subset\nobreak)\bigcup_{B\in\sG_\dl}B\subset \bigcup_j \bar B(x_j,5r_j)$.
Then,
\begin{align*}
\sH^{d-2}_{10\dl}(A_n)
\le\sum_j v_{d-2}(5r_j)^{d-2}
\le v_{d-2}5^{d-2}\sum_j n\kp(\bar B(x_j,r_j))
\le v_{d-2}5^{d-2}n\kp(B(0,n)).
\end{align*}
Letting $\dl\to0$, we obtain that $\sH^{d-2}(A_n)\le v_{d-2}5^{d-2}n\kp(B(0,n))<\infty$.
From Theorem~3 in \cite[Section~4.7]{EG}, $\Cp^{1,2}(A_n)=0$. (Here, we used the relation $d>2$.)
Therefore, $\Cp^{1,2}\left(\bigcup_{n=1}^\infty A_n\right)=0$, which implies \Eq{RN2}.
\qed\end{proof}

\begin{proposition}\label{prop:charge}
Suppose that condition~{\rm(U')$_d$} holds for some $d\in\N$. 
Then, for $\bfh\in\cH^d$ in {\rm(U')$_d$}, the measure $\bfh_*\mu$ does not concentrate on $\Cp^{1,2}$-null set.
More precisely stated, if $B\in\cB(\R^d)$ satisfies $\Cp^{1,2}(B)=0$, then $(\bfh_*\mu)(\R^d\setminus B)>0$.
\end{proposition}
\begin{proof}
Although the claim might be deduced from the results of \cite{FST91}, we provide a direct proof.
Take $\zt\in L^1(K,\mu)$ such that $0<\zt\le1$ on $K$.
We denote the measure $\zt\cdot\mu$ by $\mu_\zt$.
Then, $\mu_\zt$ is a finite measure on $K$ and $(\cE,\cF)$ is closable in $L^2(K,\mu_\zt)$ (cf.~\cite[Corollary~4.6.1]{FOT}, Eq.~(6.2.22) in \cite{FOT} and the description around there.)
From \Prop{equivalence}, it is sufficient to prove that $(\bfh_*\mu_\zt)(\R^d\setminus B)>0$ for any Borel subset $B$ of $\R^d$ with $\Cp^\rho(B)=0$.
Assume that this claim is false. Then, there exists $B\in\cB(\R^d)$ such that $\Cp^\rho(B)=0$ and $(\bfh_*\mu_\zt)(\R^d\setminus B)=0$.
Since $(\bfh_*\mu_\zt)(\R^d)<\infty$, we can take a sequence of compact sets $\{B_k\}_{k=1}^\infty$ such that $B_1\subset B_2\subset\dots\subset B$ and $(\bfh_*\mu_\zt)(\R^d\setminus B_k)\searrow0$ as $k\to\infty$.
Note that $\Cp^\rho(B_k)=0$ for all $k$.
From \cite[Lemma~2.2.7]{FOT}, there exists $f_k\in C_c^1(\R^d)$, $k=1,2,\dots$, such that 
\[
\begin{array}{l}
\mbox{$1\le f_k\le 1+1/k$ on $B_k$, $0\le f_k\le 1+1/k$ on $\R^d$}\\
\mbox{(in particular, $\lim_{k\to\infty}f_k(x)=1$ for $\bfh_*\mu_\zt$-a.e.\,$x$ on $B$),}
\end{array}
\]
and
\[
  \lim_{k\to\infty}\int_{\R^d}\bigl(|\nab f_k|_{\R^d}^2+f_k^2\bigr)(\rho+1)\,dx= 0.
\]
By taking a subsequence if necessary, we may also assume that $\lim_{k\to\infty}f_k(x)=0$ for $\sL^d$-a.e.\,$x$ on $\R^d$.
Define $g_k=1-f_k\in C^1(\R^d)$ for $k\in\N$.

Now, fix $u\in C_c^1(\R^d)$ such that $u(0,\dots,0)=0$ and $\int_{\R^d}|\nab u|_{\R^d}^2\rho\,dx\ne0$.
Then, for each $k\in\N$, $u g_k\in C_c^1(\R^d)$, $u(\bfh)g_k(\bfh)\in\cF$, and the following estimates hold:
\begin{align*}
&\cE(u(\bfh)g_k(\bfh)-u(\bfh)g_l(\bfh))\\
&=\frac12\sum_{i,j=1}^d
\int_K\biggl(\frac\partial{\partial x_i}\{u(g_k-g_l)\}\biggr)(\bfh)
\biggl(\frac\partial{\partial x_i}\{u(g_k-g_l)\}\biggr)(\bfh)
\,d\nu_{h_i,h_j}\\
&=\frac12\int_{\R^d}|\nab(u(g_k-g_l))|_{\R^d}^2\rho\,dx
\quad\text{(from (U')$_d$ (b) and (c))}\\
&\le \int_{\R^d}|\nab u|_{\R^d}^2(g_k-g_l)^2\rho\,dx
+\int_{\R^d}u^2|\nab(g_k-g_l)|_{\R^d}^2\rho\,dx\\
&\le \|\nab u\|_{L^\infty(\R^d,dx)}^2\int_{\R^d}(f_k-f_l)^2\rho\,dx
+\|u\|_{L^\infty(\R^d,dx)}^2\int_{\R^d}|\nab(f_k-f_l)|_{\R^d}^2\rho\,dx\\
&\to 0 \qquad(k,l\to\infty)
\end{align*}
and
\begin{align*}
\int_K \{u(\bfh)g_k(\bfh)\}^2\,d\mu_\zt
&=\int_{\R^d}(ug_k)^2\,d(\bfh_*\mu_\zt)
\le\|u\|_{L^\infty(\R^d,\bfh_*\mu_\zt)}^2\int_B g_k^2\,d(\bfh_*\mu_\zt)\\*
&\qquad\mbox{(since $(\bfh_*\mu_\zt)(\R^d\setminus B)=0$)}\\
&\to 0 \qquad (k\to\infty),
\end{align*}
while 
\begin{align*}
&\cE(u(\bfh)g_k(\bfh))\\
&=\frac12\int_{\R^d}|\nab(ug_k)|_{\R^d}^2\rho\,dx\\
&=\frac12\int_{\R^d}|\nab u|_{\R^d}^2g_k^2\rho\,dx
+\int_{\R^d}ug_k(\nab u,\nab g_k)_{\R^d}\rho\,dx
+\frac12\int_{\R^d}u^2|\nab g_k|_{\R^d}^2\rho\,dx\\
&=\frac12\int_{\R^d}|\nab u|_{\R^d}^2g_k^2\rho\,dx
+\int_{\R^d}ug_k(\nab u,\nab f_k)_{\R^d}\rho\,dx
+\frac12\int_{\R^d}u^2|\nab f_k|_{\R^d}^2\rho\,dx\\
&\to \frac12\int_{\R^d}|\nab u|_{\R^d}^2\rho\,dx+0+0\ne0\qquad(k\to\infty).
\end{align*}
These estimates contradict the closability of $(\cE,\cF)$ on $L^2(K,\mu_\zt)$.
\qed\end{proof}
\begin{theopargself}\begin{proof}[of \Thm{general}]
(i) Since condition~(U)$_{d'}$ implies (U)$_{d}$ if $d'>d$, it suffices to deduce a contradiction by assuming $d=2$.
Take $\bfh=(h_1,h_2)\in \cH^2$ in condition~(U)$_2$. 
Then, there exists $x_0\in \tilde\bfh(K)\subset \R^2$ such that 
\begin{enumerate}[(a)]
\item $\sup_{r>0}\mint_{B(x_0,r)}\rho(x)\,dx=:b<\infty$;
\item $(\bfh_*\mu)(\{x_0\})=0$.
\end{enumerate}
This is because $\sL^2(\tilde\bfh(K))>0$, the set of $x_0\in\tilde\bfh(K)$ that does not satisfy (a) is an $\sL^2$-null set from the Hardy--Littlewood maximal inequality, and the points in $\tilde\bfh(K)$ that do not satisfy (b) are at most countable.
By considering $\bfh(\cdot)-x_0$ instead of $\bfh$, we may assume $x_0=0$ without loss of generality.

Let $\eps>0$.
Take a smooth function $g$ on $[0,\infty)$ such that
\[
g(t)=\begin{cases}
 1& t\in[0,e^{-2/\eps}],\\
 {-3\eps\log t}-4 & t\in[e^{-14/(9\eps)},e^{-13/(9\eps)}],\\
 0 & t\in[e^{-1/\eps},\infty),
\end{cases}
\]
and $-3\eps/t\le g'(t)\le0$ for all $t>0$.
We write $|\tilde \bfh|(x)=\sqrt{\tilde h_1(x)^2+\tilde h_2(x)^2}$ and define $f(x):=g\bigl(|\tilde\bfh|(x)\bigr)$.
Then, $f$ is quasi-continuous and $f=1$ on $\tilde\bfh^{-1}(\{0\})$.
We have
\begin{align*}
2\cE(f)&=\nu_f(K)\\
&=\int_K g'(|\tilde\bfh|)^2\left(\frac{\tilde h_1}{|\tilde\bfh|}\right)^2d\nu_{h_1}
+\int_K g'(|\tilde\bfh|)^2\left(\frac{\tilde h_2}{|\tilde\bfh|}\right)^2d\nu_{h_2}
\quad\mbox{(since $\nu_{h_1,h_2}=0$)}\\
&=\int_K g'(|\tilde\bfh|)^2d\nu_{\bfh}
\qquad\mbox{(since $\nu_{h_1}=\nu_{h_2}=\nu_{\bfh}$)}\\
&=\int_0^\infty g'(r)^2\,(|\tilde\bfh|_*\nu_{\bfh})(dr)
\le \int_{e^{-2/\eps}}^{e^{-1/\eps}}9\eps^2 r^{-2}\,(|\tilde\bfh|_*\nu_{\bfh})(dr).
\end{align*}
Define
$
\Theta(r)=(|\tilde\bfh|_*\nu_{\bfh})\bigl([0,r]\bigr)
$
for $r>0$. Then, 
\begin{align*}
0\le \Theta(r)=\nu_{\bfh}(\{|\tilde\bfh|\le r\})
=\int_{\bar B(0,r)}\rho(x)\,dx
\le b\sL^2\left( B(0,r)\right)=b\pi r^2
\end{align*}
and
\begin{align*}
\frac{2}{9}\cE(f)
&\le \int_{e^{-2/\eps}}^{e^{-1/\eps}} \eps^2 r^{-2}\,d\Theta(r)
=\eps^2\biggl(\left[r^{-2}\Theta(r)\right]_{e^{-2/\eps}}^{e^{-1/\eps}}+\int_{e^{-2/\eps}}^{e^{-1/\eps}} 2r^{-3}\Theta(r)\,dr\biggr)\\
&\le \eps^2\biggl(b\pi +2b\pi\int_{e^{-2/\eps}}^{e^{-1/\eps}} r^{-1}dr\biggr)
= b\pi(\eps^2+2\eps)
\to 0\quad (\eps\to0).
\end{align*}
Also, we have
\[
\int_K f^2\,d\mu
\le \mu(\{|\bfh|<e^{-1/\eps}\})
=(\bfh_*\mu)\bigl(B(0,e^{-1/\eps})\bigr)\to
(\bfh_*\mu)(\{0\})=0 \quad(\eps\to0).
\]
Therefore, $\Cp\bigl(\tilde\bfh^{-1}(\{0\})\bigr)=0$ from \cite[Theorem~2.1.5]{FOT}.
This contradicts the assumption.

(ii)
Since inequality $d\le d_\mathrm{s}$ is evident if $d\le2$, we may assume $d\ge3$.
Take $\bfh=(h_1,\dots,h_d)\in\cH^d$ in condition~(U')$_d$.
First, we will prove that there exists $x_0\in \tilde\bfh(K)$ such that 
\begin{enumerate}[(a)]
\item $\sup_{r>0}\mint_{B(x_0,r)}\rho(y)\,dy=:b<\infty$;
\item $(\bfh_*\mu)(\bar B(x_0,r))=o(r^{d-2})$ as $r\to0$;
\item there exist $a>0$ and $r_0>0$ such that $(\bfh_*\mu)(\bar B(x_0,r))\ge ar^d$ for every $r\in(0,r_0]$.
\end{enumerate}
Indeed, the set of $x_0\in\tilde\bfh(K)$ that fails to satisfy both (a) and (b) is $\Cp^{1,2}$-null from Lemmas~\ref{lem:bound} and \ref{lem:RN2}.
The set of $x_0\in\tilde\bfh(K)$ that does not satisfy (c) is $\bfh_*\mu$-null from \Lem{RN}.
Therefore, \Prop{charge} assures the existence of $x_0$ that satisfies (a), (b), and (c).
By considering $\bfh(\cdot)-x_0$ instead of $\bfh$, we may assume $x_0=0$ without loss of generality.

It is sufficient to deduce the contradiction by assuming $d> d_\mathrm{s}$.
We write $|\tilde \bfh(x)|=\sqrt{\tilde h_1(x)^2+\cdots+ \tilde h_d(x)^2}$ for $x\in K$.
Take a smooth function $g$ on $[0,\infty)$ such that
\[
g(t)=\begin{cases}
 1& t\in[0,1],\\
 t^{2-d} & t\in[2,3],\\
 0 & t\in[4,\infty),
\end{cases}
\]
and $-c_{3.3}\le g'(t)\le0$ for all $t>0$, where $c_{3.3}$ is a positive constant.
For $\dl\in(0,r_0]$, define $g_\dl(t)=\dl^{1-(d/2)}g(t/\dl)$ for $t\ge0$, and $f_\dl(x)=g_\dl(|\tilde\bfh(x)|)$ for $x\in K$.
Then, as in the calculation in the proof of (i), we have
\begin{align*}
2\cE(f_\dl)
&=\int_K g_\dl'(|\tilde\bfh|)^2\,d\nu_{\bfh}
=\int_{\dl}^{4\dl}g_\dl'(r)^2\,\bigl(|\tilde\bfh|_*\nu_{\bfh}\bigr)(dr)\\
&\le c_{3.3}^2 \dl^{-d}\bigl(|\tilde\bfh|_*\nu_{\bfh}\bigr)\bigl([\dl,4\dl]\bigr)
\le c_{3.3}^2\dl^{-d}\bigl(\tilde\bfh_*\nu_{\bfh}\bigr)(\bar B(0,4\dl))\\
&\le c_{3.3}^2 \dl^{-d}b v_d (4\dl)^d
=O(1) \quad (\dl\to0),
\end{align*}
where $v_d=\sL^d(\bar B(0,1))$, and
\begin{align*}
\|f_\dl\|_{L^2(K,\mu)}^2
&=\int_0^{4\dl}g_\dl(r)^2\,\bigl(|\bfh|_*\mu\bigr)(dr)\\
&\le \dl^{2-d}(\bfh_*\mu)(\bar B(0,4\dl))
= \dl^{2-d}o((4\dl)^{d-2})
=o(1) \quad (\dl\to0).
\end{align*}
On the other hand, we have
\begin{align*}
\|f_\dl\|_{L^{2d_\mathrm{s}/(d_\mathrm{s}-2)}(K,\mu)}^2
&\ge\biggl(\int_0^{\dl}g_\dl(r)^{2d_\mathrm{s}/(d_\mathrm{s}-2)}\,(|\bfh|_*\mu)(dr)\biggr)^{(d_\mathrm{s}-2)/{d_\mathrm{s}}}\\
&\ge \dl^{2-d}a\dl^{d(d_\mathrm{s}-2)/d_\mathrm{s}}
=a\dl^{-2(d-d_\mathrm{s})/d_\mathrm{s}}
\to+\infty\quad (\dl\to0).\end{align*}
Therefore, the Sobolev inequality~\Eq{sobolev} does not hold, which is a contradiction.
\qed\end{proof}\end{theopargself}
\section{Estimation in the case of self-similar sets}
In this section, we consider self-similar Dirichlet forms on self-similar sets such as p.\,c.\,f.~fractals and Sierpinski carpets and show that \Str{str} can be realized to deduce the estimates of the martingale dimensions. 
\subsection{Self-similar Dirichlet forms on self-similar sets}
We follow \cite{Ki,Hi05} to set up a framework.
Let $K$ be a compact and metrizable topological space, and $S$, a finite set with $\# S\ge2$.
We suppose that we are given continuous injective maps $\psi_i\colon K\to K$ for $i\in S$.
Set $\Sg=S^\N$. For $i\in S$, we define a shift operator $\sg_i\colon \Sg\to\Sg$ by $\sg_i(\om_1\om_2\cdots)=i\om_1\om_2\cdots$.
Suppose that there exists a continuous surjective map $\pi\colon \Sg\to K$ such that $\psi_i\circ \pi=\pi\circ\sg_i$ for every $i\in S$.
We term $(K,S,\{\psi_i\}_{i\in S})$ a self-similar structure.

We also define $W_0=\{\emptyset\}$, $W_m=S^m$ for $m\in \N$, and
denote $\bigcup_{m\ge0}W_m$ by $W_*$.
For $w=w_1w_2\cdots w_m\in W_m$, we define $\psi_w=\psi_{w_1}\circ\psi_{w_2}\circ\cdots\circ\psi_{w_m}$
and $K_w=\psi_w(K)$.
By convention, $\psi_{\emptyset}$ is the identity map from $K$ to $K$.
For $w\in W_*$ and a function $f$ on $K$, $\psi_w^*f$ denotes the pullback of $f$ by $\psi_w$, that is, $\psi_w^*f=f\circ \psi_w$.
\begin{definition}\label{def:AA'}
For $w=w_1w_2\cdots w_m\in W_m$ and $w'=w'_1w'_2\cdots w'_{m'}\in W_{m'}$, $ww'$ (or $w\cdot w'$) denotes $w_1w_2\cdots w_mw'_1w'_2\cdots w'_{m'}\in W_{m+m'}$.
For $A\subset W_m$ and $A'\subset W_{m'}$, $A\cdot A'$ denotes $\{ww'\in W_{m+m'}\mid w\in A,\ w'\in A'\}$.
If $A=\{w\}$, we denote $A\cdot A'$ by $w\cdot A'$.
\end{definition}
  Take $\te=\{\te_i\}_{i\in S}\in \R^S$ such that $\te_i>0$ for every $i\in S$ and $\sum_{i\in S}\te_i=1$.
  We set $\te_w=\te_{w_1}\te_{w_2}\cdots \te_{w_m}$ for $w=w_1w_2\cdots w_m\in W_m$, and $\te_\emptyset=1$.  
  Let $\lm_\te$ denote the Bernoulli measure on $\Sg$ with weight $\te$.
  That is, $\lm_\te$ is a unique Borel probability measure such that $\lm_\te(\Sg_w)=\te_w$ for every $w\in W_*$.
Define a Borel measure $\mu_\te$ on $K$ by $\mu_\te=\pi_*\lm_\te$, that is, $\mu_\te(B)=\lm_\te(\pi^{-1}(B))$ for $B\in\cB(K)$.
  It is called the self-similar measure on $K$ with weight $\te$.
  
We impose the following assumption.
\begin{enumerate}[(A0)]
\item[(A1)] For every $x\in K$, $\pi^{-1}(\{x\})$ is a finite set.
\end{enumerate}
Then, according to Theorem~1.4.5 and Lemma~1.4.7 in \cite{Ki}, $\mu_\te(K^b)=0$ with $K^b=\{x\in K\mid \#(\pi^{-1}(\{x\}))>1\}$, and 
  $\mu_\te(K_w)=\te_w$ for all $w\in W_*$.
For any $x\in K\setminus K^b$, there exists a unique element $\om=\om_1\om_2\cdots\in \Sg$ such that $\pi(\om)=x$.
We denote $\om_1\om_2\cdots\om_m\in W_m$ by $[x]_m$ for each $m\in\N$, and define $[x]_0=\emptyset$.
The sequence $\{K_{[x]_m}\}_{m=0}^\infty$ is a fundamental system of neighborhoods of $x$ from \cite[Proposition~1.3.6]{Ki}.

Fix a self-similar measure $\mu$ on $K$.
\begin{definition}\label{def:Psiw}
For $w\in W_*$ and $f\in L^2(K,\mu)$, we define $\Psi_w f\in L^2(K,\mu)$ by 
  \[
    \Psi_w f(x)=\begin{cases}
    f(\psi_w^{-1}(x)) & \mbox{if }x\in K_w, \\
    0 & \mbox{otherwise.}
    \end{cases}
  \]
\end{definition}
  Since $\mu(K^b)=0$, $\psi_{w'}^* \Psi_w f:=(\Psi_w f)\circ\psi_{w'} =0 $ $\mu$-a.e.\ if $w$ and $w'$ are different elements of some $W_m$.

We set 
$
\cP=\bigcup_{m=1}^\infty \sg^m\left(\pi^{-1}\left(\bigcup_{i,j\in S,\,i\ne j}(K_i\cap K_j)\right)\right)$ and $V_0=\pi(\cP)$,
where $\sg^m\colon\Sigma\to\Sigma$ is a shift operator that is defined by $\sg^m(\om_1\om_2\cdots)=\om_{m+1}\om_{m+2}\cdots$.
The set $\cP$ is referred to as the post-critical set.

We consider a regular Dirichlet form $(\cE,\cF)$ defined on $L^2(K,\mu)$.
Take a closed subset $K^\partial$ of $K$ such that $V_0\subset K^\partial\subsetneq K$. 
In concrete examples discussed later, we always take $V_0$ as $K^\partial$. 
Recall $\cF_0$ and $\cFD$ that were introduced in \Eq{cF0cFD}.
We assume the following.
  \begin{enumerate}[(A0)]
  \item[(A2)] $1\in\cF$ and $\cE(1)=0$.
  \item[(A3)] (Self-similarity) $\psi_i^* f\in\cF$ for every $f\in \cF$ and $i\in S$, and there exists $\bfr=\{r_i\}_{i\in S}$ with $r_i>0$ for all $i\in S$ such that
  \[
    \cE(f)=\sum_{i\in S}\frac1{r_i} \cE(\psi_i^* f),
    \quad f\in \cF.
 \]
  \item[(A4)] (Spectral gap) There exists a constant $c_{4.1}>0$ such that
 \begin{equation}\label{eq:Npoincare}
   \left\|f-\int_K f\,d\mu \right\|_{L^2(K,\mu)}^2
   \le c_{4.1} \cE(f)
   \quad \mbox{for all }f\in\cF.
 \end{equation}
  \item[(A5)] $\Psi_i f\in \cF_0$ for any $f\in\cF_0$ and $i\in S\subset W_*$.
  \item[(A6)] For any $f\in\cF$ and $w\in W_*$, there exists $\hat f\in\cF$ such that $\psi_w^*\hat f=f$.
\end{enumerate}
We remark that, for any $f,g\in \cF$ and $m\in\N$, it holds that 
\begin{equation}\label{eq:selfsimilar2}
  \cE(f,g)=\sum_{w\in W_m}\frac1{r_w} \cE(\psi_w^* f,\psi_w^* g)
\end{equation}
from the polarization argument and repeated use of (A3), where $r_w$ denotes $r_{w_1}r_{w_2}\cdots r_{w_m}$ for $w=w_1w_2\cdots w_m$ and $r_{\emptyset}=1$.
The Dirichlet form $(\cE,\cF)$ is inevitably strong local, e.g., from \cite[Lemma~3.12]{Hi05} and (A2).
Typical examples are self-similar Dirichlet forms on post-critically finite self-similar sets and Sierpinski carpets, which we discuss in Sections~4.2 and 4.3.
Readers who are not familiar with these objects are recommended to read the definitions described in these subsections before proceeding to the subsequent arguments.   

The following is a basic property of harmonic functions.
\begin{lemma}\label{lem:basicharmonic}
For any $h\in \cH$ and $w\in W_*$, $\psi_w^* h$ belongs to $\cH$.
\end{lemma}
\begin{proof}
Take any $g\in \cF_0$. From condition~(A5), $\Psi_w g\in\cF_0$.
Then, by \Lem{harmonic} and \Eq{selfsimilar2},
\begin{align*}
0=\cE(h,\Psi_w g)
=\sum_{w'\in W_m}r^{-m}\cE(\psi_{w'}^*h,\psi_{w'}^*\Psi_w g)
=r^{-m}\cE(\psi_w^*h,g).
\end{align*}
Therefore, $\cE(\psi_w^*h,g)=0$.
This implies that $\psi_w^* h\in\cH$.
\qed\end{proof}
The energy measures associated with $(\cE,\cF)$ have the following properties.
\begin{lemma}[cf.~{\cite[Lemma~3.11]{Hi05}}]\label{lem:energymeas}
Let $f\in\cF$. Then, the following hold.
\begin{enumerate}
\item Let $w\in W_*$. For any exceptional set $N$ of $K$, $\psi_w^{-1}(N)$ is also an exceptional set.
  In particular, if we denote a quasi-continuous modification of $f\in \cF$ by $\tilde f$, then $\psi_w^* \tilde f$ is a quasi-continuous modification of $\psi_w^* f$.
\item For $m\in\Z_+$ and a Borel subset $B$ of $K$,
\[
  \nu_f(B)=\sum_{w\in W_m}\frac1{r_w}\nu_{\psi_w^*f}(\psi_w^{-1}(B)).
\]
\end{enumerate}
\end{lemma}
For $d\in\N$, $\bff=(f_1,\dots,f_d)\in \cF^d$, and a map $\psi\colon K\to K$, we denote the $\R^d$-valued function $(\psi^*f_1,\dots,\psi^*f_d)$ on $K$ by $\psi^*\bff$.
We also recall the terminology in \Defn{basic}.
\begin{lemma}\label{lem:scale}
Let $d\in\N$, $\bff=(f_1,\dots,f_d)\in\cF^d$, and $w\in W_*$. Take a quasi-continuous nonnegative function $g$ on $K$ such that $g\ge1$ q.e.\ on $K_w$. 
Then, 
\begin{equation}\label{eq:scale}
(\psi_w^*\tilde\bff)_*\nu_{\psi_w^*\bff}\le r_w\tilde\bff_*(g\cdot\nu_{\bff})
\end{equation}
as measures on $\R^d$, that is,
$\nu_{\psi_w^*\bff}((\psi_w^*\tilde\bff)^{-1}(B))\le r_w\int_{\tilde\bff^{-1}(B)}g\,d\nu_{\bff}$ for any $B\in\cB(\R^d)$.
\end{lemma}
\begin{proof}
Let $B\in\cB(\R^d)$ and denote $\tilde\bff^{-1}(B)$ by $B'$.
From \Lem{energymeas}~(ii), for $i=1,\dots,d$,
\begin{align*}
r_w^{-1}\nu_{\psi_w^* f_i}((\psi_w^*\tilde\bff)^{-1}(B))
&=r_w^{-1}\nu_{\psi_w^* f_i}(\psi_w^{-1}(B'))
=r_w^{-1}\nu_{\psi_w^* f_i}(\psi_w^{-1}(B'\cap K_w))\\
&\le \nu_{f_i}(B'\cap K_w)
\le (g\cdot\nu_{f_i})(\tilde\bff^{-1}(B)).
\end{align*}
Therefore, $\nu_{\psi_w^*f_i}((\psi_w^*\tilde\bff)^{-1}(B))\le r_w\int_{\tilde\bff^{-1}(B)}g\,d\nu_{f_i}$.
This implies \Eq{scale}.
\qed\end{proof}
We note that condition (A7) mentioned below is not required for Lemmas~\ref{lem:energymeas} and \ref{lem:scale}.

We fix a minimal energy-dominant measure $\nu$ with $\nu(K)<\infty$, and further assume the following.
\begin{enumerate}[(A0)]
  \item[(A7)] $\nu(K^\partial)=0$.
\end{enumerate}

Let $K^\partial_*=\bigcup_{w\in W_*}\psi_w(K^\partial)$ and $V_*=\bigcup_{w\in W_*}\psi_w(V_0)$. 
Clearly, $K^\partial_*\supset V_*$.
\begin{lemma}\label{lem:energymeas'}
Let $f\in\cF$. Then, the following hold.
\begin{enumerate}
\item $\nu_f(K^\partial_*)=0$.
\item For $w\in W_*$ and a Borel subset $B$ of $K_w$,
\[
  \nu_f(B)=\frac1{r_w}\nu_{\psi_w^*f}(\psi_w^{-1}(B)).
\]
\end{enumerate}
\end{lemma}
\begin{proof}
(i): For $m\in\N$ and $w'\in W_m$, from \Lem{energymeas}~(ii) and (A7),
\begin{align*}
\nu_f(\psi_{w'}(K^\partial))
&=\sum_{w\in W_m}\frac1{r_w}\nu_{\psi^*_w f}(\psi_w^{-1}(\psi_{w'}(K^\partial)))
\le \sum_{w\in W_m}\frac1{r_w}\nu_{\psi^*_w f}(K^\partial)
=0,
\end{align*}
where in the second line, we used the relation 
\[
\psi_w^{-1}(\psi_{w'}(K^\partial))\begin{cases}=K^\partial&\text{if } w=w',\\ \subset V_0\subset K^\partial&\text{otherwise.}\end{cases}
\]
Therefore, $\nu_f(\psi_{w'}(K^\partial))=0$.
This implies (i).
Item~(ii) follows from (i), \Lem{energymeas}~(ii), and the fact $K^\partial_*\supset V_*$.
\qed\end{proof}

For the proof of the next proposition, let $\cB_m$ be a $\sg$-field on $K$ generated by $\{K_w\mid w\in W_m\}$ for $m\ge0$.
Then, $\{\cB_m\}_{m=0}^\infty$ is a filtration on $K$ and the $\sg$-field generated by $\{\cB_m\mid m\ge0\}$ is equal to $\cB(K)$ (from the result of \cite[Proposition~1.3.6]{Ki}, for example). 
\begin{proposition}\label{prop:energyequivalence}
Let $m\in\Z_+$.
Define $\nu'_m=\sum_{w\in W_m}r_w^{-1}(\psi_w)_*\nu$.
That is, 
\[
\nu'_m(B):=\sum_{w\in W_m}\frac{1}{r_w}\nu(\psi_w^{-1}(B)),
\quad B\in\cB(K).
\]
Then, $\nu$ and $\nu'_m$ are mutually absolutely continuous.
Moreover, for any $f,g\in \cF$ and $w\in W_m$,
\begin{equation}\label{eq:abso}
  \frac{d\nu_{f,g}}{d\nu'_m}(x)=\frac{d\nu_{\psi_w^* f,\psi_w^* g}}{d\nu}(\psi_w^{-1}(x))
  \quad \mbox{for }\nu\mbox{-a.e.\,}x\in K_w.
\end{equation}
\end{proposition}
\begin{proof}
This is proved as in \cite[Proposition~4.3]{Hi10}.
From \Prop{em}, there exists $f\in\cF$ such that $\nu_f$ and $\nu$ are mutually absolutely continuous. 
Let $B$ be a Borel set of $K$.
Suppose $\nu'_m(B)=0$.
Then, for $w\in W_m$,
$
0=((\psi_w)_*\nu)(B)=\nu(\psi_w^{-1}(B\cap K_w))$.
Since $\nu_{\psi_w^* f}\ll\nu$, we have
$0=\nu_{\psi_w^* f}(\psi_w^{-1}(B\cap K_w))=r_w\nu_f(B\cap K_w)$
from \Lem{energymeas'}~(ii).
Since $w\in W_m$ is arbitrary, $\nu_f(B)=0$, that is, $\nu(B)=0$.
Therefore, $\nu\ll\nu'_m$.

Next, suppose $\nu(B)=0$.
Let $w\in W_m$.
From (A6), there exists $\hat f\in\cF$ such that $\psi_w^* \hat f=f$.
From \Lem{energymeas}~(ii), 
$0=\nu_{\hat f}(B)\ge r_w^{-1}\nu_{f}(\psi_w^{-1}(B))$.
Thus, $0=\nu(\psi_w^{-1}(B))=((\psi_w)_*\nu)(B)$.
Therefore, $\nu'_m(B)=0$.
This implies $\nu'_m\ll\nu$.

For the proof of \Eq{abso}, let $n\ge m$.
From \Lem{energymeas'}, for $x\in K_w\setminus V_*$,
\begin{equation}\label{eq:ratio}
  \frac{\nu_{f,g}(K_{[x]_n})}{\nu'_m(K_{[x]_n})}
  = \frac{r_w^{-1}\nu_{\psi_w^* f, \psi_w^* g}(\psi_w^{-1}(K_{[x]_n}))}
  {r_w^{-1}\nu(\psi_w^{-1}(K_{[x]_n}))}
  = \frac{\nu_{\psi_w^* f, \psi_w^* g}(K_{[\psi_w^{-1}(x)]_{n-m}})}
  {\nu(K_{[\psi_w^{-1}(x)]_{n-m}})}.
\end{equation}
If $\nu'_m$ is a probability measure, the first term is given by the conditional expectation $E^{\nu'_m}[{d\nu_{f,g}}/{d\nu'_m}\mid\cB_n](x)$.
From the martingale convergence theorem, this term converges to $({d\nu_{f,g}}/{d\nu'_m})(x)$ for $\nu'_m$-a.e.\,$x$ as $n\to\infty$.
It is evident that this convergence holds true for general $\nu'_m$.
By the same reasoning, the last term of \Eq{ratio} converges $({d\nu_{\psi_w^* f, \psi_w^* g}}/{d\nu})(\psi_w^{-1}(x))$ for $(\psi_w)_*\nu$-a.e.\,$x$ as $n\to\infty$.
Since $\nu$, $\nu'_m$, and $(\psi_w)_*\nu$ are mutually absolutely continuous on $K_w$ from the first claim, we obtain \Eq{abso}.
\qed\end{proof}
\begin{corollary}\label{cor:Phi}
For $d\in \N$, $\bff=(f_1,\dots,f_d)\in\cF^d$ and $w\in W_*$,
\[
  \Ph_{\bff}(\psi_w(y))=\Ph_{\psi_w^*\bff}(y)
  \quad \mbox{for $\nu$-a.e.\,}y\in K,
\]
where $\Ph.$ is defined in \Eq{Phif}.
\end{corollary}
\begin{proof}
Let $m=|w|$. From \Prop{energyequivalence}, for $i,j=1,\dots,d$,
\[
\frac{d\nu}{d\nu'_m}(\psi_w(y))\frac{d\nu_{f_i,f_j}}{d\nu}(\psi_w(y))
=\frac{d\nu_{\psi_w^* f_i,\psi_w^* f_j}}{d\nu}(y)
  \quad \mbox{for $\nu$-a.e.\,}y\in K.
\]
This implies the assertion.
\qed\end{proof}

For $m\ge0$, let $\cH_m$ denote the set of all functions $f$ in $\cF$ such that $\psi_w^*f\in\cH$ for all $w\in W_m$.
Let $\cH_*=\bigcup_{m\ge0}\cH_m$.
Functions in $\cH_*$ are referred to as piecewise harmonic functions.
From \cite[Lemma~3.10]{Hi05}, $\cH_*$ is dense in $\cF$. 
The AF-martingale dimension of $(\cE,\cF)$ is denoted by $d_\mathrm{m}$ as before.
\begin{proposition}\label{prop:invertible}
Let $d\in\N$ satisfy $d\le d_\mathrm{m}$. 
Then, there exists $\bfg=(g_1,\dots,g_d)\in\cH^d$ such that
\begin{equation}\label{eq:invertible}
\nu_{\bfg}\left(\{x\in K\mid \Ph_{\bfg}(x)\mbox{ is invertible}\}\right)>0.
\end{equation}
\end{proposition}
\begin{proof}
Take a countable set $\{f_i\mid i\in \N\}$ from $\cH_*$ such that it is dense in $\cF$.
For $i,j\in\N$, define $\hat Z^{i,j}={d\nu_{f_i,f_j}}/{d\nu}$.
From \Prop{rank} and \Thm{index}, we have
$
\esssup_{x\in K}\;\sup_{N\in\N}\; \rank\left(\hat Z^{i,j}(x)\right)_{i,j=1}^N\ge d$.
Then, there exists $N\in\N$ such that 
$
\nu\bigl(\bigl\{x\in K\bigm| \rank\bigl(\hat Z^{i,j}(x)\bigr)_{i,j=1}^N\ge d\bigr\}\bigr)>0
$.
Therefore, there exists $1\le \a_1<\a_2<\dots<\a_d\le N$ such that $\nu(\hat B)>0$ with
\[
\hat B=\bigl\{x\in K\bigm| \mbox{the matrix}\left(\hat Z^{\a_i,\a_j}(x)\right)_{i,j=1}^d\mbox{is invertible}\bigr\}.
\]
We can take a sufficiently large $m\in\Z_+$ such that every $f_{\a_i}$, $i=1,\dots,d$, belongs to $\cH_m$.
Take $w\in W_m$ such that $\nu(\hat B\cap K_w)>0$.
Define $\bfg=(g_1,\dots,g_d)\in\cH^d$ by $g_i=\psi_w^* f_{\a_i}$, $i=1,\dots,d$, and let $Z^{i,j}={d\nu_{g_i,g_j}}/{d\nu}$
for $i,j\in\{1,\dots,d\}$.
From \Prop{energyequivalence}, we have $\nu(B)>0$, where
\[
B=\bigl\{x\in K\bigm|  \mbox{the matrix }\left(Z^{i,j}(x)\right)_{i,j=1}^d\mbox{ is invertible}\bigr\}.
\]
Since the trace of any invertible and nonnegative definite symmetric matrix is positive, and $({d\nu_{\bfg}}/{d\nu})(x)=(1/d)\tr\left(Z^{i,j}(x)\right)_{i,j=1}^d$, we have $B\subset \{d\nu_{\bfg}/d\nu>0\}$ up to $\nu$-null set, which implies $\nu_{\bfg}(B)>0$.
Then, \Eq{invertible} holds since
$
  \Phi_{\bfg}(x)=\bigl(Z^{i,j}(x)\big/\frac{d\nu_{\bfg}}{d\nu}(x)\bigr)_{i,j=1}^d
$ on $\{{d\nu_{\bfg}}/{d\nu}>0\}$.
\qed\end{proof}
For later use, we introduce the following sets for given $d\in\N$ and $a>0$:
\begin{align*}
\Mat(d)&=\{\mbox{All real square matrices of order $d$}\},\\
\PSM(d;a)&=\{Q\in \Mat(d)\mid \mbox{$Q$ is a positive definite symmetric matrix and }\det Q\ge a\}.
\end{align*} 
The set $\Mat(d)$ is identified with $\R^{d\times d}$ as a topological vector space, and 
$\PSM(d;a)$ is regarded as a closed subset of $\Mat(d)$.
\subsection{Case of post-critically finite self-similar sets}
In this subsection, we follow \cite{Ki} and consider the case that $K$ is connected and the self-similar structure $(K,S,\{\psi_i\}_{i\in S})$ that was introduced in the previous subsection is {\em post-critically finite} (p.c.f.), that is, $\cP$ is a finite set.
See \Fig{fig1} as some of the typical examples.
Let $V_m=\bigcup_{w\in W_m}\psi_w(V_0)$ for $m\in\N$ and $V_*=\bigcup_{m=0}^\infty V_m$.

In general, given a finite set $V$, $l(V)$ denotes the space of all real-valued functions on $V$.
We equip $l(V)$ with an inner product $(\cdot,\cdot)_{l(V)}$ that is defined by $(u,v)_{l(V)}=\sum_{q\in V}u(q)v(q)$.
Let $D=(D_{qq'})_{q,q'\in V_0}$ be a symmetric linear operator on $l(V_0)$, which is also regarded as a square matrix of size $\#V_0$, such that the following conditions hold:
\begin{enumerate}[(D1)]
\item $D$ is nonpositive-definite;
\item $Du=0$ if and only if $u$ is constant on $V_0$;
\item $D_{qq'}\ge0$ for all $q,q'\in V_0$ with $q\ne q'$.
\end{enumerate}
We define $\cE^{(0)}(u,v)=(-Du,v)_{l(V_0)}$ for $u,v\in l(V_0)$.
This is a Dirichlet form on $l(V_0)$, where $l(V_0)$ is identified with the $L^2$ space on $V_0$ with the counting measure (see \cite[Proposition~2.1.3]{Ki}).
For $\bfr=\{r_i\}_{i\in S}$ with $r_i>0$, we define a bilinear form $\cE^{(m)}$ on $l(V_m)$ as
\begin{equation*}
  \cE^{(m)}(u,v)=\sum_{w\in W_m}\frac{1}{r_w}\cE^{(0)}(u\circ\psi_w|_{V_0},v\circ\psi_w|_{V_0}),\quad
  u,v\in l(V_m).
\end{equation*}
We refer to $(D,\bfr)$ as a {\em harmonic structure} if for every $v\in l(V_0)$,
\[
\cE^{(0)}(v,v)=\inf\{\cE^{(1)}(u,u)\mid u\in l(V_1)\mbox{ and }u|_{V_0}=v\}.
\]
Then, for $m\in\Z_+$ and $v\in l(V_m)$,
\[
\cE^{(m)}(v,v)=\inf\{\cE^{(m+1)}(u,u)\mid u\in l(V_{m+1})\mbox{ and }u|_{V_m}=v\}.
\]
In particular, $\cE^{(m)}(u|_{V_m},u|_{V_m})\le \cE^{(m+1)}(u,u)$ for $u\in l(V_{m+1})$.

We consider only {\em regular} harmonic structures, that is, $0<r_i<1$ for all $i\in S$.
Demonstrating the existence of regular harmonic structures is a nontrivial problem.
Several studies have been conducted, such as in \cite{Li,HMT06,Pe07}. 
We only remark here that all nested fractals have canonical regular harmonic structures.
Nested fractals are self-similar sets that are realized in Euclidean spaces and have good symmetry; for the precise definition, see \cite{Li,Ki}. 
All the fractals shown in \Fig{fig1} except the rightmost one are nested fractals.

We assume that a regular harmonic structure $(D,\bfr)$ is given.
Let $\mu$ be a self-similar probability measure on $K$, and take $V_0$ as $K^\partial$.
We can then define a regular Dirichlet form $(\cE,\cF)$ on $L^2(K,\mu)$ associated with $(D,\bfr)$, satisfying conditions (A1)--(A7), by
\begin{align*}
\cF&=\left\{u\in C(K)\subset L^2(K,\mu)\left|\,
\lim_{m\to\infty}\cE^{(m)}(u|_{V_m},u|_{V_m})<\infty\right.\right\},\\
\cE(u,v)&= \lim_{m\to\infty}\cE^{(m)}(u|_{V_m},v|_{V_m}),\quad u,v\in\cF.
\end{align*}
(See the beginning of \cite[Section~3.4]{Ki}.)
Note that (A7) follows from the fact that $\#K^\partial\,({}=\#V_0)<\infty$, \Prop{em}, and \Prop{edp}.
From \cite[Theorem~3.3.4]{Ki}, a property stronger than (A4) follows: There exists a constant $c_{4.2}>0$ such that
\begin{equation}\label{eq:poincare}
  \sup_{x\in K}f(x)-\inf_{x\in K}f(x)\le c_{4.2}\sqrt{\cE(f)},
  \quad f\in\cF\subset C(K).
\end{equation}
From this inequality, it is easy to prove that the capacity associated with $(\cE,\cF)$ of any nonempty subset of $K$ is uniformly positive (see, e.g., \cite[Proposition~4.2]{Hi08}).

Let us recall that the space of all harmonic functions is denoted by $\cH$.
For each $u\in l(V_0)$, there exists a unique $h\in\cH$ such that $h|_{V_0}=u$.
For any $w\in W_*$ and $h\in\cH$, $\psi_w^* h$ belongs to $\cH$.
By using the linear map $l(V_0)\ni u\mapsto h\in \cH$, we can identify $\cH$ with $l(V_0)$.
In particular, $\cH$ is a finite dimensional subspace of $\cF$.

The following is the main theorem of this subsection, which is an improvement of \cite[Theorem~4.4]{Hi08}.
\begin{theorem}\label{th:main1}
The index of $(\cE,\cF)$ is $1$. In other words, the AF-martingale dimension $d_\mathrm{m}$ of the diffusion process associated with $(\cE,\cF)$ is $1$.
\end{theorem}
Unlike \cite[Theorem~4.4]{Hi08}, we do not need technical extra assumptions.  
The main ideas of the proofs of \cite[Theorem~4.4]{Hi08} and \Thm{main1} are quite different from each other.
\begin{theopargself}
\begin{proof}[of \Thm{main1}]
Since $(\cE,\cF)$ is nontrivial, $d_\mathrm{m}\ge1$ from \Prop{nontrivial}.
We will derive a contradiction by assuming $d_\mathrm{m}\ge2$.
We proceed to Step~1 of \Str{str} with $d=2$.
From \Prop{invertible}, there exists $\bfg=(g_1,g_2)\in\cH^2$ such that \Eq{invertible} holds.
Take $a>0$ such that
\begin{equation}\label{eq:defa}
\nu_{\bfg}\left(\{x\in K\mid \det\Ph_{\bfg}(x)\ge a\}\right)=:\dl>0.
\end{equation}
Let $B=\left\{x\in K\mid \det\Ph_{\bfg}(x)\ge a\right\}\setminus V_*$. From \Eq{defa} and \Lem{energymeas'}~(i), $\nu_{\bfg}(B)=\dl>0$.
Let us recall $\Mat(2)$ and $\PSM(2;a)$ that were introduced in the end of the previous subsection.
A map that is obtained by restricting the domain of $\Ph_{\bfg}$ to $B$ is denoted by $\Ph_{\bfg}|_B$.
This is regarded as a map from $B$ to $\PSM(2;a)$.
Fix an element $L$ in the support of the induced measure $(\Phi_{\bfg}|_B)_*(\nu_{\bfg}|_B)$ on $\PSM(2;a)$.

We will perform a kind of blowup argument.
Let $k\in\N$. We denote by $U_k$ the intersection of $\PSM(2;a)$ and the open ball with center $L$ and radius $1/k$ in $\Mat(2)\simeq\R^{2\times2}$ with respect to the Euclidean norm. Let $B_k=(\Phi_{\bfg}|_B)^{-1}(U_k)\subset B$.
Then, $\nu_{\bfg}(B_k)>0$.
For $n\in \N$, we set
\[
  Y_n^{(k)}(x)=\begin{cases}{\nu_{\bfg}(K_{[x]_n}\cap B_k)}\big/{\nu_{\bfg}(K_{[x]_n})}& \mbox{if $x\in K\setminus V_*$ and $\nu_{\bfg}(K_{[x]_n})>0$},\\0&\mbox{otherwise.}\end{cases}
\]
Then, from the martingale convergence theorem as in the proof of \Prop{energyequivalence},
$\lim_{n\to\infty}Y_n^{(k)}=1$ $\nu_{\bfg}$-a.e.\ on $B_k$.
In particular, there exist $x_k\in B_k$ and $N_k\in\N$ such that $Y_n^{(k)}(x_k)\ge1-2^{-k}$ for any $n\ge N_k$.

Take increasing natural numbers $n_1<n_2<n_3<\cdots$ such that $Y_{n_k}^{(k)}(x_k)\ge1-2^{-k}$ for all $k$. 
We set $\hat B_k=\psi_{[x_k]_{n_k}}^{-1}(B_k)$. 
We define $\bfg^{(k)}=(g^{(k)}_1,g^{(k)}_2)\in\cH^2$ as
\[
\bfg^{(k)}:=\psi_{[x_k]_{n_k}}^*\bfg=(\psi_{[x_k]_{n_k}}^*g_1,\psi_{[x_k]_{n_k}}^*g_2),
\]
and $\bfh^{(k)}=(h^{(k)}_1,h^{(k)}_2)\in\cH^2$ as
\[
  h_i^{(k)}=\Bigl({g_i^{(k)}-\int_K g_i^{(k)}\,d\mu}\Bigr)\Big/{\sqrt{2\cE(\bfg^{(k)})}},
  \quad i=1,2.
\]
Here, we note that $2\cE(\bfg^{(k)})=\nu_{\bfg^{(k)}}(K)=r_{[x_k]_{n_k}}\nu_{\bfg}(K_{[x_k]_{n_k}})>0$ from \Lem{energymeas'} and $Y_{n_{k}}^{(k)}(x_{k})>0$.
Then,
\begin{equation}\label{eq:average}
\int_K h_i^{(k)}\,d\mu=0,
\quad i=1,2,
\end{equation}
\begin{equation}\label{eq:nuh}
\nu_{\bfh^{(k)}}(K)
=2\cE(\bfh^{(k)})
=1,
\end{equation}
and
\begin{align*}
\nu_{\bfh^{(k)}}(\hat B_k)
=\frac{\nu_{\bfg^{(k)}}(\hat B_k)}{2\cE(\bfg^{(k)})}
=\frac{r_{[x_k]_{n_k}}\nu_{\bfg}(K_{[x_k]_{n_k}}\cap B_k)}{r_{[x_k]_{n_k}}\nu_{\bfg}(K_{[x_k]_{n_k}})}
=Y_{n_{k}}^{(k)}(x_{k})\ge 1-2^{-k}.
\end{align*}
From \Eq{poincare}, \Eq{average}, and \Eq{nuh}, $\{\bfh^{(k)}\}_{k\in\N}$ is bounded in $\cF^2$.
Since $\cH^2$ is a {\em finite-dimensional} subspace of $\cF^2$, we can take a subsequence $\{\bfh^{(k(j))}\}$ of $\{\bfh^{(k)}\}$ converging to some $\bfh\in\cH^2$ in $\cF^2$.
We may assume that $2\cE(\bfh-\bfh^{(k(j))})\le 2^{-j}$ for all $j$ and 
\begin{equation}\label{eq:convergence1}
\lim_{j\to\infty}\Ph_{\nu_{\bfh^{(k(j))}}}(x)=\Ph_{\nu_{\bfh}}(x)\text{ for $\nu_{\bfh}$-a.e.\,$x$}
\end{equation}
from \Lem{energyineq}, by taking a further subsequence if necessary.

Since $2\cE(\bfh^{(k(j))})=1$ for all $j$, $\nu_{\bfh}(K)=2\cE(\bfh)=1$.
We also have
\begin{align*}
\sqrt{\nu_{\bfh}(K\setminus \hat B_{k(j)})}
&\le\left|\sqrt{\nu_{\bfh}(K\setminus \hat B_{k(j)})}-\sqrt{\nu_{\bfh^{(k(j))}}(K\setminus \hat B_{k(j)})}\right|+\sqrt{\nu_{\bfh^{(k(j))}}(K\setminus \hat B_{k(j)})}\\
&\le\sqrt{2\cE(\bfh-\bfh^{(k(j))})}+\sqrt{\nu_{\bfh^{(k(j))}}(K\setminus \hat B_{k(j)})}\notag\\
&\le 2^{-j/2}+2^{-k(j)/2}
\le 2^{-j/2}+2^{-j/2},\notag
\end{align*}
that is, $\nu_{\bfh}(K\setminus \hat B_{k(j)})\le 2^{-j+2}$.

From Borel--Cantelli's lemma, for $\nu_{\bfh}$-a.e.\,$x\in K$, $x$ belongs to $\hat B_{k(j)}$ for sufficiently large $j$.
Note that $x\in \hat B_{k(j)}$ implies $\Ph_{\bfg}(\psi_{[x_{k(j)}]_{n_{k(j)}}}(x))\in U_{k(j)}$.
From \Cor{Phi}, for $\nu_{\bfh}$-a.e.\,$x\in K$, $\Ph_{\bfh^{(k(j))}}(x)\in U_{k(j)}$ for sufficiently large $j$.
Therefore, $\Ph_{\bfh}(x)=L$ for $\nu_{\bfh}$-a.e.\,$x\in K$ from \Eq{convergence1}.
From \Lem{renormalize}, we may assume that $L$ is the identity matrix. This completes Step~1 of \Str{str}.

Take $f\in \cFD$ such that $f>0$ on $K\setminus V_0$.
From \Prop{W12}, $\bfh_*(f^2\nu_{\bfh})\ll\sL^2$.
Since $\nu_{\bfh}(V_0)=0$ by (A7), $\bfh_*\nu_{\bfh}\ll\sL^2$.
This meets condition~(U)$_2$, which conflicts with \Thm{general}~(i) since the capacity of any nonempty set is positive. Therefore, the assumption $d_\mathrm{m}\ge2$ is invalid, which completes the proof of \Thm{main1}.
\qed\end{proof}
\end{theopargself}
\subsection{Case of Sierpinski carpets}
Let $D$ and $l$ be integers with $D\ge2$ and $l\ge3$.
We assume that the cardinality of the index set $S$, denoted by $M$, is less than $l^D$.
Let $Q_0=[0,1]^D$, the $D$-dimensional unit cube.
Let $\sC$ be the collection of all cubes that are described as $\prod_{j=1}^D[k_j/l,(k_j+1)/l]$ for $k_j\in\{0,1,\dots,l-1\}$.
Assume that we are given a family $\{\psi_i\}_{i\in S}$ of contractive affine transformations on $\R^D$ of type $\psi_i(x)=l^{-1} x+b_i$ for $b_i\in \R^D$ such that each $\psi_i$ maps $Q_0$ onto some cube in $\sC$, and $\psi_i\ne \psi_{i'}$ if $i\ne i'$.
Let $Q_m=\bigcup_{w\in W_m}\psi_w(Q_0)$ for $m\in\N$ and $K=\bigcap_{m\in\N}Q_m$.
Then, $(K,S,\{\psi_i\}_{i\in S})$ is a self-similar structure and $K$ is called a (generalized) Sierpinski carpet, which satisfies condition~(A1) in Section~4.1.
See \Fig{fig2} in Section~1 for typical examples.
We take the normalized Hausdorff measure on $K$ as the underlying measure $\mu$.
In order to define a self-similar Dirichlet form on $L^2(K,\mu)$, we further assume the following properties, which are due to M.~T.~Barlow and R.~F.~Bass:
\begin{itemize}
  \item (Symmetry) $Q_1$ is preserved by all the isometries of the unit cube $Q_0$.
  \item (Connectedness) $\Int(Q_1)$ is connected and contains a path connecting the hyperplanes $\{x_1=0\}$ and $\{x_1=1\}$.
  \item (ND: Nondiagonality) Let $m\ge1$ and $B$ be a cube in $Q_0$ of side length $2/l^m$ that is described as $\prod_{j=1}^D [k_j/l^m,(k_j+2)/l^m]$ for $k_j\in\{0,1,\dots, l^m-2\}$.
  Then, $\Int(Q_1\cap B)$ is either an empty set or a connected set.
  \item (BI: Borders included) $Q_1$ contains the line segment $\{(x_1,0,\dots,0)\in\R^D \mid 0 \le x_1\le1\}$.
\end{itemize}
In the above description, $\Int(B)$ denotes the interior of $B$ in $\R^D$.
After several studies such as \cite{BB89,KZ92,BB99}, the unique existence of the ``Brownian motion'' on $K$ up to the constant time change was proved in \cite{BBKT10}.
It has an associated nontrivial regular Dirichlet form $(\cE,\cF)$ on $L^2(K,\mu)$ that satisfies conditions~(A2)--(A4), where $r_i$ in (A3) is independent of $i$. We denote $r_i$ by $r$ and take $K\setminus \Int(Q_0)$ as $K^\partial$, which coincides with $V_0$.
Moreover, $(\cE,\cF)$ has the following property:
\begin{equation}\label{eq:reflection}
\text{For any isometries $\psi$ on $Q_0$ and $f\in\cF$, $\psi^*f$ belongs to $\cF$ and $\cE(\psi^*f)=\cE(f)$.}
\end{equation}
From this property, we can easily prove the following:
\begin{lemma}\label{lem:symmetry}
For any isometries $\psi$ on $Q_0$, $f\in\cF$, and $B\in\cB(K)$, we have $\nu_{\psi^*f}(B)=\nu_f(\psi(B))$.
\end{lemma}
We will confirm that conditions~(A5)--(A7) are also satisfied.
We remark that we do not use the uniqueness of $(\cE,\cF)$ in the subsequent argument.
\begin{remark}
In \cite{BB99}, the nondiagonality condition was assumed only for $m=1$, but it was not sufficient; it was corrected to the above form in \cite{BBKT10}. In some articles such as \cite{Hi05,HK06}, the conditions described in \cite{BB99} were inherited, which should also be corrected.
\end{remark}
Concerning the nondiagonality, we remark the following fact.
See \cite{Ka10} for the proof.
\begin{proposition}
The following are mutually equivalent.
\begin{itemize}
\item Nondiagonality condition {\rm(ND)} holds.
\item {\rm(ND)} with only $m=2$ holds.
\item {\rm (ND)${}_\text{H}$:} Let $B$ be a $D$-dimensional rectangle in $Q_0$ such that each side length of $B$ is either $1/l$ or $2/l$ and $B$ is a union of some elements of $\sC$. Then, $\Int(B\cap Q_1)$ is either an empty set or a connected set.
\end{itemize}
\end{proposition}
We list some properties of this Dirichlet form and the associated objects.
  The Brownian motion has the heat kernel density $p(t,x,y)$ that is continuous in $(t,x,y)\in(0,\infty)\times K\times K$ such that, for some positive constants $c_{4.i}$ $(i=3,4,5,6)$,
\begin{align}\label{eq:Aronson}
    &c_{4.3} t^{-d_\mathrm{s}/2}\exp\bigl(-c_{4.4} (|x-y|_{\R^D}^{d_\mathrm{w}}/t)^{1/(d_\mathrm{w}-1)}\bigr)\notag\\
    &\quad\le 
    p(t,x,y)
   \le c_{4.5} t^{-d_\mathrm{s}/2}\exp\bigl(-c_{4.6}(|x-y|_{\R^D}^{d_\mathrm{w}}/t)^{1/(d_\mathrm{w}-1)}\bigr),
   \quad t\in(0,1],\ x,y\in K.
\end{align}
Here, $d_\mathrm{s}=(2\log M)/\log (M/r)>1$ and $d_\mathrm{w}=\log (M/r)/\log l\ge2$ (cf.\ \cite{BB99,BBK06,BBKT10}).
The constants $d_\mathrm{s}$ and $d_\mathrm{w}$ are called the spectral dimension and the walk dimension, respectively.
The resolvent operators are compact ones on $L^2(K,\mu)$.
 The Sobolev inequality~\Eq{sobolev} holds if $d_\mathrm{s}>2$.
Indeed, from \cite{Va85}, \Eq{sobolev} is equivalent to the on-diagonal upper heat kernel estimate
\begin{equation}\label{eq:upper}
p(t,x,x)\le c_{4.7}t^{-d_\mathrm{s}/2},\quad t\in(0,1],\ x\in K
\end{equation}
for some positive constant $c_{4.7}$.
  The domain $\cF$ is characterized as a Besov space. More precisely stated, the Besov spaces on $(K,\mu)$ are defined as follows: For $1\le p<\infty$, $\b\ge 0$ and $m\in \Z_+$, we set
\[
a_m (\beta, f):=\gamma^{m\beta}\biggl(\gamma^{md_\mathrm{H}}
\iint_{\{(x,y)\in K\times K\mid |x-y|_{\R^D}<c\gamma^{-m}\}}|f(x)-f(y)|^p\,\mu (dx)\,\mu(dy)\biggr)^{1/p}
\]
for $f\in L^p(K,\mu)$,
where $\gamma\in(1,\infty)$ and $c\in(0,\infty)$ are fixed constants, and $d_\mathrm{H}$ is the Hausdorff dimension of $K$, which is equal to $\log M/\log l $. 
Note that the relation 
\begin{equation}\label{eq:einstein}
d_\mathrm{H}={d_\mathrm{w} d_\mathrm{s}}/2\ge d_\mathrm{s}
\end{equation}
holds.
Then, for $1\le q\le \infty$, the Besov space
$\Lambda^\beta_{p,q}(K)$ is defined as the set of all $f\in L^p(K,\mu)$ such that
${\bar a}(\b, f):=\{a_m(\b, f)\}_{m=0}^{\infty}\in l^q$. 
$\Lambda^\b_{p,q}(K)$ is a Banach space with  
norm $\| f\|_{\Lambda^\beta_{p,q}(K)}
:=\|f\|_{L^p(K,\mu)}+\|{\bar a}(\beta, f)\|_{l^q}$. 
Different selections of $c>0$ and $\gamma>1$ provide the same space $\Lambda^\beta_{p,q}(K)$ with equivalent norms.
For $f\in L^2(K,\mu)$ and $\dl>0$, we define
\begin{equation}\label{eq:Edl}
E_\dl(f):=\dl^{-d_\mathrm{w}-d_\mathrm{H}}\iint_{\{(x,y)\in K\times K\mid |x-y|_{\R^D}<\dl\}}|f(x)-f(y)|^2\,\mu(dx)\,\mu(dy).
\end{equation}
\begin{theorem}[{cf.\ \cite[Theorem~5.1]{Gr03}, \cite{Kum00}}]\label{th:E}
The domain $\cF$ is equal to $\Lambda_{2,\infty}^{d_\mathrm{w}/2}(K)$, and the norm $\|\cdot\|_\cF$ is equivalent to $\|\cdot\|_{\Lambda_{2,\infty}^{d_\mathrm{w}/2}(K)}$.
Moreover, $f\in\cF$ if and only if $f\in L^2(K,\mu)$ and $\limsup_{\dl\to0}E_\dl(f)<\infty$. 
Further, for $f\in \cF$,
\[
\cE(f)\asymp \sup_{\dl>0}E_\dl(f)\asymp \limsup_{\dl\to0}E_\dl(f).
\]
\end{theorem}
Here, $a_1\asymp a_2$ represents that there exists a constant $c\ge1$ depending only on $K$ and $(\cE,\cF)$ such that $c^{-1}a_1\le a_2\le ca_1$ holds.

  From this characterization, condition~(A5) is verified.
  Condition (A6) is confirmed, for example, by (A3), \Thm{E}, and a property of the unfolding operator introduced in \cite[p.~665]{BBKT10}. This is also assured by \Lem{reflection} in the next section.
  Condition~(A7) is proved in \cite[Remark~5.3]{BBKT10} under some extra assumptions, e.g., the set $\{(x_2,\dots,x_D)\in \R^{D-1}\mid (0,x_2,\dots,x_D)\in K\}$ also satisfies the conditions corresponding to (H1)--(H4). The proof is based on \cite[Proposition~3.8]{HK06}, and these extra assumptions were introduced for the main topic of the paper~\cite{HK06}, i.e., the characterization of the trace space of $\cF$ on subsets such as surfaces of Sierpinski carpets.
  However, in order to prove condition~(A7) only, such assumptions are in fact not necessary, as seen from the careful modification of the arguments in \cite{HK06}.  
  Since the setup of \cite{HK06} is quite complicated and it is not easy to extract and modify the necessary parts for this purpose, this will be discussed in Section~5 and the following proposition is proved there. 
\begin{proposition}\label{prop:A7}
Condition {\rm(A7)} holds true.
In particular, $\nu_f(K^\partial)=0$ for any $f\in\cF$.
\end{proposition}
For the time being, we admit this proposition and continue arguments.
  The main theorem of this subsection is as follows.
\begin{theorem}\label{th:main2}
$1\le d_\mathrm{m}\le d_\mathrm{s}$. In particular, if $d_\mathrm{s}<2$, then $d_\mathrm{m}=1$.
\end{theorem}
We note that $d_\mathrm{s}<2$ if and only if the diffusion process associated with $(\cE,\cF)$ is point recurrent.
In view of \Eq{einstein}, $d_\mathrm{s}<2$ holds in particular for 2-dimensional Sierpinski carpets (that is, when $D=2$).
For the 3-dimensional standard Sierpinski carpet (shown in the rightmost figure of \Fig{fig2}), $2<d_\mathrm{s}<3$ holds from \cite[Corollary~5.3]{BB99}, which implies that $d_\mathrm{m}$ is either $1$ or $2$. It has not been determined which is true.

Compared with the case of p.c.f.\ fractals in Section~4.2, the proof of \Thm{main2} is more complicated in that the space $\cH$ of all harmonic functions is infinite-\hspace{0pt}dimensional, so that much work is required to select a converging sequence from a bounded set in $\cH$.

For the proof of \Thm{main2}, we introduce one more notation.
\begin{definition}\label{def:cell}
For $A\subset W_m$ for $m\in\Z_+$, we set $K_A=\bigcup_{w\in A}K_w$.
For $w\in W_m$ with $m\in\Z_+$, we define $\cN_0(w)=\{w\}$ and
\[
  \cN_n(w)=\{v\in W_m\mid K_v\cap K_{\cN_{n-1}(w)}\ne\emptyset\},\quad
  n=1,2,3,\dots,
\]
inductively.
\end{definition}
We remark the following:
  Let $f\in\cF$, $m\in\N$, and $A,A'\subset W_m$ with $A\cap A'=\emptyset$.
  From \Lem{energymeas'}~(i), we have
\begin{equation}\label{eq:additive}
\nu_f(K_{A\cup A'})=\nu_f(K_{A})+\nu_f(K_{A'}).
\end{equation}
We also note that for any $n\in \Z_+$, $\#\cN_n(w)\le (2n+1)^D$ for $w\in W_*$ and 
\[
\sup_{m\in\N}\max_{v\in W_m}\#\{w\in W_m\mid v\in \cN_n(w)\}\le (2n+1)^D.
\]
\begin{theopargself}
\begin{proof}[of \Thm{main2}]
  Since $(\cE,\cF)$ is nontrivial, it is sufficient to prove that $d_\mathrm{m}\le d_\mathrm{s}$ from \Prop{nontrivial}.
Take $d\in\N$ arbitrarily such that $d\le d_\mathrm{m}\,(\,\le+\infty)$.
From \Prop{invertible}, there exists $\bfg=(g_1,\dots,g_d)\in\cH^d$ 
that satisfies \Eq{invertible}.
We may assume $\nu_{\bfg}(K)=1$ by multiplying $\bfg$ by a normalizing constant.
There exists $a>0$ such that
\begin{equation}\label{eq:A1}
\nu_{\bfg}(B_0)=:\dl>0,
\mbox{ where }B_0=\{x\in K\mid \det\Ph_{\bfg}(x)\ge a\}.
\end{equation}
Since $\nu_{\bfg}(K^\partial)=0$ by (A7), there exists $n_0\in\N$ such that for any $n\ge n_0$,
\begin{equation}\label{eq:A2}
\nu_{\bfg}(K_{A(n)})\le\dl/3,
\mbox{ where }A(n)=\{w\in W_n\mid K^\partial\cap K_{\cN_3(w)}\ne\emptyset\}.
\end{equation}
Let 
$b=\sup_{n\in\N}\max_{v\in W_n}\#\{w\in W_n\mid v\in \cN_3(w)\}(\le 7^D)$
and $\eps=\dl /(3b)$.
For $n\ge n_0$, define
$G_n=\{w\in W_n\mid \nu_{\bfg}(K_w)\le\eps\nu_{\bfg}(K_{\cN_3(w)})\}$.
Then, from \Eq{additive},
\[
\nu_{\bfg}(K_{G_n})
=\sum_{w\in G_n}\nu_{\bfg}(K_w)
\le\eps\sum_{w\in G_n}\nu_{\bfg}(K_{\cN_3(w)})
\le \eps b\nu_{\bfg}(K)
=\dl/3.
\]
We define $K_\infty=\liminf_{n\to\infty}K_{G_n}$. From Fatou's lemma,
\begin{equation}\label{eq:A3}
\nu_{\bfg}(K_\infty)\le\liminf_{n\to\infty}\nu_{\bfg}(K_{G_n})\le \dl/3.
\end{equation}
We set $B=B_0\setminus(K_{A(n_0)}\cup K_\infty\cup K^\partial_*)$. 
Then, 
$
\nu_{\bfg}(B)\ge \dl-\dl/3-\dl/3=\dl/3
$
from \Eq{A1}, \Eq{A2}, \Eq{A3}, and \Lem{energymeas'}~(i).

Let $\Ph_{\bfg}|_B$ denote the map $\Ph_{\bfg}$ whose defining set is restricted to $B$. This is a map from $B$ to $\PSM(d;a)$.
Fix an element $L$ in the support of the measure $(\Phi_{\bfg}|_B)_*(\nu_{\bfg}|_B)$ on $\PSM(d;a)$.
We will perform a blowup argument.

Let $k\in\N$. We denote by $U_k$ the intersection of $\PSM(d;a)$ and the open ball with center $L$ and radius $1/k$ in $\Mat(d)\simeq\R^{d\times d}$ with respect to the Euclidean norm. Let $B_k=(\Phi_{\bfg}|_B)^{-1}(U_k)\subset B$.
Then, $\nu_{\bfg}(B_k)>0$.
For $n\in \N$, we set
\[
  Y_{n}^{(k)}(x)=\begin{cases}{\nu_{\bfg}(K_{[x]_n}\cap B_k)}\big/{\nu_{\bfg}(K_{[x]_n})}& \mbox{if $x\in K\setminus K^\partial_*$ and $\nu_{\bfg}(K_{[x]_n})>0$},\\0&\mbox{otherwise.}\end{cases}
\]
Then, from the martingale convergence theorem as in the proof of \Prop{energyequivalence},
$\lim_{n\to\infty}Y_{n}^{(k)}=1$ $\nu_{\bfg}$-a.e.\ on $B_k$.
In particular, there exist $x_k\in B_k$ and $N_k\in\N$ such that $Y_{n}^{(k)}(x_k)\ge1-2^{-k}$ for any $n\ge N_k$.
Since $x_k\notin K_\infty$, for infinitely many $n$, $\nu_{\bfg}(K_{[x_k]_n})>\eps\nu_{\bfg}(K_{\cN_3([x_k]_n)})$.
Therefore, there exists a sequence of increasing natural numbers \mbox{$(n_0\le\,)\,n_1<n_2<n_3<\cdots$} such that 
\begin{equation}\label{eq:bound}
Y_{n_k}^{(k)}(x_k)\ge1-2^{-k}
\quad
\mbox{and}
\quad
\nu_{\bfg}(K_{[x_k]_{n_k}})>\eps\nu_{\bfg}(K_{\cN_3([x_k]_{n_k})})
\end{equation}
for all $k\in\N$. 
For each $k\in\N$, define $\bfg^{(k)}=(g_1^{(k)},\dots,g_d^{(k)})\in\cH^d$ as
\begin{equation}\label{eq:gik}
  g_i^{(k)}=\biggl({g_i-\mint_{K_{[x_k]_{n_k}}}g_i\,d\mu}\biggr)\Big/{\sqrt{r^{n_k}\nu_{\bfg}(K_{[x_k]_{n_k}})}},
  \quad i=1,\dots,d.
\end{equation}
Then, 
$\int_{K_{[x_k]_{n_k}}}g_i^{(k)}\,d\mu=0$ $(i=1,\dots,d)$, 
\begin{equation}\label{eq:nubfg}
  \nu_{\bfg^{(k)}}(K_{[x_k]_{n_k}})=\frac1d\sum_{i=1}^d\nu_{g_i^{(k)}}(K_{[x_k]_{n_k}})
  =r^{-n_k},
\end{equation}
and
\begin{equation}\label{eq:rnk}
  r^{n_k}\nu_{\bfg^{(k)}}(K_{\cN_3([x_k]_{n_k})})={\nu_{\bfg}(K_{\cN_3([x_k]_{n_k})})}\big/{\nu_{\bfg}(K_{[x_k]_{n_k}})}<{1}/\eps
\end{equation}
for all $k\in\N$, from \Eq{bound} and \Eq{gik}.
We denote $\psi_{[x_k]_{n_k}}^*\bfg^{(k)}$ by $\bfh^{(k)}$.
Then, from \Lem{energymeas'}~(ii) and \Eq{nubfg},
\begin{align*}
  \nu_{\bfh^{(k)}}(K)
  =\frac1d\sum_{i=1}^d \nu_{\psi_{[x_k]_{n_k}}^* g_i^{(k)}}(K)
  =\frac1d\sum_{i=1}^d {r^{n_k}}\nu_{g_i^{(k)}}(K_{[x_k]_{n_k}})
  =1.
\end{align*}
Denoting $\psi_{[x_k]_{n_k}}^{-1}(B_k)$ by $\hat B_k$, we have
\begin{align*}
\nu_{\bfh^{(k)}}(\hat B_k)
&=\nu_{\psi_{[x_k]_{n_k}}^*\bfg^{(k)}}(\psi_{[x_k]_{n_k}}^{-1}(B_k))
=r^{n_k}\nu_{\bfg^{(k)}}(K_{[x_k]_{n_k}}\cap B_k)\\
&={\nu_{\bfg}(K_{[x_k]_{n_k}}\cap B_k)}\big/{\nu_{\bfg}(K_{[x_k]_{n_k}})}
=Y_{n_k}^{(k)}(x_k)\ge 1-2^{-k}.
\end{align*}
From \Eq{rnk}, we can use the following proposition.
\begin{proposition}\label{prop:key}
Let $\{h_n\}_{n=1}^\infty$ be a sequence in $\cH$ and $\{w_n\}_{n=1}^\infty$ be a sequence in $W_*$ such that
$K_{\cN_3(w_n)}\cap K^\partial=\emptyset$ and $\int_{K_{w_n}}h_n\,d\mu=0$
for all $n\in\N$, and 
\begin{equation}\label{eq:N3}
\sup_n r^{|w_n|}\nu_{h_n}(K_{\cN_3(w_n)})<\infty.
\end{equation}
Then, the sequence $\{\psi_{w_n}^* h_n\}_{n=1}^\infty$ has a convergent subsequence in $\cF$.
\end{proposition}
We note that $r^{|w_n|}\nu_{h_n}(K_{w_n})=\nu_{\psi_{w_n}^* h_n}(K)=2\cE(\psi_{w_n}^* h_n)$ from \Lem{energymeas'}.

Since the proof of \Prop{key} is long, we postpone it until the next section and finish the proof of \Thm{main2} first.

By applying \Prop{key} to $\{g_i^{(k)}\}_{k=1}^\infty\subset\cH$ and $\{[x_k]_{n_k}\}_{k=1}^\infty\subset W_*$ for each $i=1,\dots,d$ successively, we can take a subsequence $\{\bfh^{(k(j))}\}_{j=1}^\infty$ of $\{\bfh^{(k)}\}_{k=1}^\infty$, converging to some $\bfh\in\cH^d$ in $\cF^d$.
By taking a further subsequence, we may assume that $2\cE(\bfh-\bfh^{(k(j))})\le 2^{-j}$ for all $j$ and 
\begin{equation}\label{eq:convergence2}
\lim_{j\to\infty}\Ph_{\bfh^{(k(j))}}(x)=\Ph_{\bfh}(x)
\text{ for $\nu_{\bfh}$-a.e.\,$x$}
\end{equation}
from \Lem{energyineq}.
Then, $\nu_{\bfh}(K)=1$ and
\begin{align*}
\sqrt{\nu_{\bfh}(K\setminus \hat B_{k(j)})}
&\le\Bigl|\sqrt{\nu_{\bfh}(K\setminus \hat B_{k(j)})}-\sqrt{\nu_{\bfh^{(k(j))}}(K\setminus \hat B_{k(j)})}\Bigr|+\sqrt{\nu_{\bfh^{(k(j))}}(K\setminus \hat B_{k(j)})}\\
&\le\sqrt{2\cE(\bfh-\bfh^{(k(j))})}+\sqrt{\nu_{\bfh^{(k(j))}}(K\setminus \hat B_{k(j)})}\notag\\
&\le 2^{-j/2}+2^{-k(j)/2}
\le 2^{-j/2}+2^{-j/2},\notag
\end{align*}
that is, $\nu_{\bfh}(K\setminus \hat B_{k(j)})\le 2^{-j+2}$.
From Borel--Cantelli's lemma, for $\nu_{\bfh}$-a.e.\,$x\in K$, $x\in \hat B_{k(j)}$ for sufficiently large $j$.
Note that $x\in \hat B_{k(j)}$ implies that $\Ph_{\bfg}(\psi_{[x_{k(j)}]_{n_{k(j)}}}(x))\in U_{k(j)}$.
From \Cor{Phi}, $\Ph_{\bfh^{(k(j))}}(x)\in U_{k(j)}$ for sufficiently large $j$ for $\nu_{\bfh}$-a.e.\,$x\in K$.
Therefore, $\Ph_{\bfh}(x)=L$ for $\nu_{\bfh}$-a.e.\,$x\in K$ from \Eq{convergence2}.
From \Lem{renormalize}, we may assume that $L$ is the identity matrix. This completes Step~1 of \Str{str}.

Take $w\in W_*$ such that $K_w\cap K^\partial=\emptyset$ and $\nu_{\bfh}(K_w)>0$.
From the regularity of $(\cE,\cF)$, there exists $f\in \cFD\cap C(K)$ such that $0\le f\le1$ on $K$ and $f=1$ on $K_w$.
From \Prop{W12}, the measure $\tilde\bfh_*(f^2\nu_{\bfh})$ on $\R^d$ is described as $\xi(x)\,dx$ with $\sqrt{\xi}\in W^{1,2}(\R^d)$.
From \Cor{Phi} and \Lem{scale}, $\psi_w^*\bfh$ plays the role of $\bfh$ in Step~2 of \Str{str}, and condition~(U')$_d$ is satisfied.

Now, if $d_\mathrm{s}<2$, then the process associated with $(\cE,\cF)$ is point recurrent and the capacity on nonempty set is positive, thus $d=1$ from \Thm{general}~(i).
if $d_\mathrm{s}>2$, we have $d\le d_\mathrm{s}$ from \Thm{general}~(ii).
When $d_\mathrm{s}=2$, \Eq{upper} holds with $d_\mathrm{s}$ replaced by any number bigger than $2$, since the larger $d_\mathrm{s}$ is, the weaker the inequality is. Thus, the Sobolev inequality~\Eq{sobolev} holds with $d_\mathrm{s}$ replaced by any number bigger than $2$, for example, $2.01$. From \Thm{general}~(ii), $d\le 2.01$. Since $d$ is a natural number, we obtain $d\le 2$.
This completes the proof of \Thm{main2} if we grant Propositions~\ref{prop:A7} and \ref{prop:key}, which are proved in the next section.
\qed\end{proof}
\end{theopargself}
\section{Proof of Propositions~\protect\ref{prop:A7}\ and \protect\ref{prop:key}}
In this section, we prove Propositions~\ref{prop:A7} and \ref{prop:key}.
We use the same notations as those in Section~4.3.
In Section~5.1, we present a description of the structure of $\cF$ (\Prop{compatible}) and a quantitative estimate for a class of harmonic functions (\Prop{key1}) as preparatory results. For the proofs, we use a characterization of $\cF$ by the Besov space, folding/unfolding maps on $K$, some geometric properties of $K$ originating from the nondiagonal property~(ND), the elliptic Harnack inequality, and so on.
Using these results, we prove in Section~5.2 a claim apparently stronger than \Prop{key} (\Prop{key'}), and \Prop{A7}.
\subsection{Preliminaries}
First, we introduce some concepts. 
We have to be careful that condition~(A7) cannot be used; in particular, \Lem{energymeas'} and \Eq{additive} are not available, while \Lem{energymeas} is valid.
We remark that an assertion stronger than \Lem{energymeas}~(i) holds from \cite[pp.~600--601]{HK06}: For any $w\in W_*$, there exists a constant $c_{5.1}\ge1$ such that 
\begin{equation}\label{eq:capequiv}
c_{5.1}^{-1}\Cp(B)\le \Cp(\psi_w(B))\le c_{5.1}\Cp(B)
\end{equation}
for every $B\subset K$.

For a nonempty subset $A$ of $W_m$ for some $m\in\N$,  
a collection $\{f_w\}_{w\in A}$ of functions in $\cF$
is called {\em compatible} if $\tilde f_v(\psi_v^{-1}(x))=\tilde f_w(\psi_w^{-1}(x))$ for q.e.\,$x\in K_v\cap K_w$ for every $v,w\in A$. 
This concept is well-defined from \Eq{capequiv}.
We define
\begin{equation*}
  \cF^A=\left\{f\in L^2(K_A,\mu|_{K_A})\;\vrule\;\text{$\psi_w^* f\in\cF $ for all $w\in A$ and $\{\psi_w^* f\}_{w\in A}$ is compatible}\right\}
\end{equation*}
and 
\begin{equation}\label{eq:cEAdef}
  \cE^A(f,g)=r^{-m}\sum_{w\in A}\cE(\psi_w^*f,\psi_w^*g)
  \quad\text{for }f,g\in\cF^A.\footnote{In \cite{HK06},  symbol $(\cE_A,\cF_A)$ was used instead. Since it is slightly misleading, we use the terminology $(\cE^A,\cF^A)$ here.}
\end{equation}
It is evident that $\{f|_{K_A}\mid f\in\cF\}\subset \cF^A$.
Also, from (A3), $\cE^A(f,g)=\cE^{A\cdot W_n}(f,g)$ for any $n\in\N$ and $f,g\in\cF^A$.
See \Defn{AA'} for the definition of $A\cdot W_n$.

For simplicity, we write $\cE^A(f)$ for $\cE^A(f|_{K_A},f|_{K_A})$ if $f\in\cF$.
Then,
\begin{align}\label{eq:energyineq}
\frac12\nu_f(K_A)&=\frac{1}{2r^m}\sum_{w\in W_m}\nu_{\psi_w^*f}(\psi_w^{-1}(K_A))
\ge \frac{1}{2r^m}\sum_{w\in A}\nu_{\psi_w^*f}(K)
=\cE^A(f),
\end{align}
where the first identity follows from \Lem{energymeas}~(ii).
It will turn out that the above inequality is replaced by the equality from \Prop{A7}, which is yet to be proved.

The following result was used in \cite[Section~5.3]{HK06} without proof.
Since the proof is not obvious, we provide the proof here.
\begin{proposition}\label{prop:compatible}
Let $m,n\in\Z_+$ and let $A$ be a nonempty subset of $W_m$. 
Then, $\cF^A=\cF^{A\cdot W_n}$. 
In particular, $\cF=\cF^{W_n}$ for every $n\in\Z_+$.
\end{proposition}
Although this assertion might be deduced directly from the powerful theorem on the uniqueness of self-similar diffusions on $K$~\cite{BBKT10}, we give a proof without using this fact, since some concepts and lemmas stated below in proving \Prop{compatible} are useful elsewhere.

For Borel subsets $B_1$ and $B_2$ of $K$ and a positive constant $\dl$, we define
\begin{align*}
E_\dl(f,B_1,B_2)&:=\dl^{-d_\mathrm{w}-d_\mathrm{H}}\iint_{\{(x,y)\in B_1\times B_2\mid |x-y|_{\R^D}<\dl\}}|f(x)-f(y)|^2\,\mu(dx)\,\mu(dy)\end{align*}
for $f\in L^2(K,\mu)$.
We write $E_\dl(f,B_1)$ for $E_\dl(f,B_1,B_1)$.
Note that $E_\dl(f,K,K)=E_\dl(f)$ (see \Eq{Edl}).
\begin{definition}\label{def:periodic}
We define a folding map $\ph\colon[0,1]^D\to[0,1/l]^D$ as follows.
Let $\hat\ph\colon \R\to\R$ be a periodic function with period $2/l$ such that $\hat\ph(t)=|t|$ for $t\in[-1/l,1/l]$.
The map $\ph$ is defined as
\[
\ph(x_1,\dots,x_D)=(\hat\ph(x_1),\dots,\hat\ph(x_D)),
\quad (x_1,\dots,x_D)\in[0,1]^D.
\]
Moreover, we define $\ph_{i}\colon K\to K_i$ for $i\in S$ as
\[
\ph_{i}(x)=\left(\ph|_{K_i}\right)^{-1}(\ph(x)),\quad
x\in K.
\]
\end{definition}
Note that $\ph_i|_{K_i}\colon K_i\to K_i$ is the identity map and $\ph_i\circ \ph_j=\ph_i$ for $i,j\in S$.

Hereafter, $a_1\lesssim a_2$ means that there exists a positive constant $c$ depending only on $(K,\mu)$ and $(\cE,\cF)$ such that $a_1\le ca_2$ holds.
\begin{lemma}\label{lem:reflection}
Let $k\in S$ and $f\in \cF$.
Define $g\in\cF^S$ as
$
  g(x)=f(\psi_k^{-1}(\ph_{k}(x)))$
 for $x\in K$.
Then, $g\in\cF$ and $\cE(g)\lesssim \cE(f)$.
\end{lemma}
\begin{proof}
Let $\dl\in(0,1/l)$.
We have
\begin{align*}
E_\dl(g)
=\sum_{i\in S}\sum_{j\in S} E_\dl(g,K_i,K_j)
=\sum_{i\in S} E_\dl(g,K_i)+\sum_{i,j\in S,\ i\ne j,\ K_i\cap K_j\ne\emptyset}E_\dl(g,K_i,K_j).
\end{align*}
In the first term of the rightmost side, we have 
\[
E_\dl(g,K_i)=E_\dl(g,K_k)\lesssim E_{l\dl}(f)\lesssim\cE(f).
\]
In the second term, we have
\begin{align*}
E_\dl(g,K_i,K_j)
&=\dl^{-d_\mathrm{w}-d_\mathrm{H}}\iint_{\{(x,y)\in K_i\times K_j\mid |x-y|_{\R^D}<\dl\}}|g(x)-g(\ph_i(y))|^2\,\mu(dx)\,\mu(dy)\\*
&\hspace{16em}\text{(since $g(\ph_i(y))=g(y)$)}\\
&\le \dl^{-d_\mathrm{w}-d_\mathrm{H}}\iint_{\{(x,z)\in K_i\times K_i\mid |x-z|_{\R^D}<\dl\}}|g(x)-g(z)|^2\,\mu(dx)\,\mu(dz)\\
&=E_\dl(g,K_i)
\lesssim\cE(f).
\end{align*}
Here, in the first inequality, we used the inequality $|x-\ph_i(y)|_{\R^D}\le|x-y|_{\R^D}$ for $x\in K_i$ and $y\in K_j$, and the identity $(\ph_i|_{K_j})_*(\mu|_{K_j})=\mu|_{K_i}$.
Therefore, $\limsup_{\dl\to0}E_\dl(g)\lesssim\cE(f)$.
This completes the proof.
\qed\end{proof}
\begin{corollary}\label{cor:reflection}
Let $k\in S$ and $f\in \cF^S$.
Define $g\in\cF^S$ as $g=f\circ\ph_k$.
Then, $g\in\cF$.
\end{corollary}
We remark that $f=g$ on $K_k$.
\begin{theopargself}
\begin{proof}[of \Cor{reflection}]
Apply \Lem{reflection} to $\psi_k^* f\in \cF$ as $f$.
\qed\end{proof}
\end{theopargself}
\begin{definition}
For $m\in\N$ and $v,w\in W_m$, we write $v\underset{m}{\leftrightsquigarrow}w$ if $\psi_v(Q_0)\cap \psi_w(Q_0)$ is a $(D-1)$-dimensional hypercube.
\end{definition}
For $i,j\in S=W^1$ with $i\underset{1}{\leftrightsquigarrow}j$, let $H_{i,j}$ be a unique $(D-1)$-dimensional hyperplane including $K_i\cap K_j$.
Then, $H_{i,j}$ splits $\R^D$ into two closed half spaces, say $G_{i,j}$ and $G_{j,i}$, which satisfy that $G_{i,j}\supset K_i$ and $G_{j,i}\supset K_j$.
\begin{lemma}\label{lem:reflection0}
Let $i,j\in S$ satisfy that $i\underset{1}{\leftrightsquigarrow}j$.
Suppose that $f\in\cF^S$ satisfies that $\tilde f=0$ q.e.\ on $K_i\cap K_j$.
Define $\hat g\in \cF^S$ as
  $\hat g(x)=f(\ph_{i}(x))\cdot \bfone_{G_{i,j}}(x)$
  for $x\in K$.
Then, $\hat g\in \cF$.
\end{lemma}
\begin{proof}
From \cite[Lemma~2.3.4]{FOT}, there exists a sequence $\{\hat f_n\}_{n=1}^\infty$ in $\cF\cap C(K)$ such that $\hat f_n\to \psi_i^*f$ in $\cF$ and $\Supp[\hat f_n]\subset K\setminus \psi_i^{-1}(K_i\cap K_j)$.
For each $n$, define 
\[
g_n(x)=\hat f_n(\psi_i^{-1}(\ph_{i}(x))),\quad
\hat g_n(x)=g_n(x)\cdot \bfone_{G_{i,j}}(x)\quad
  \text{for }x\in K.
\]
Then, from \Lem{reflection}, $g_n\in \cF$ and
$E_\dl(g_n)\lesssim \cE(\hat f_n)$ for $\dl>0$.
Here, we note that the constant involved in symbol~$\lesssim$ is independent of $n$ and $\dl$.

Let $n\in \N$ and $\dl>0$ be smaller than the Euclidean distance between $\Supp[g_n]$ and $\psi_i^{-1}(K_i\cap K_j)$.
Then,
\[
  E_\dl(\hat g_n)=E_\dl(g_n,K\cap G_{i,j})\le E_\dl(g_n)\lesssim \cE(\hat f_n).
\]
Therefore, $\limsup_{\dl\to0}E_\dl(\hat g_n)\lesssim \cE(\hat f_n)$, which implies that $\hat g_n\in \cF$ and 
\[
\limsup_{n\to\infty}\cE(\hat g_n)\lesssim\limsup_{n\to\infty}\cE(\hat f_n)=\cE(\psi_i^*f)<\infty.
\]
Since $\hat g_n\to\hat g$ in $L^2(K,\mu)$, $\hat g_n$ converges weakly in $\cF$ and the limit coincides with $\hat g$. In particular, $\hat g\in\cF$.
\qed\end{proof}
\begin{definition}\label{def:Xi}
We define maps
\[
\Xi_{i}\colon \cF^S\to \cF,
\quad i\in S
\]
and 
\[
\Xi_{i,j}\colon \{f\in\cF^S\mid \tilde f=0 \text{ q.e.~on }K_i\cap K_j\}\to \cF,
\quad i,j\in S\text{ with }i\underset{1}{\leftrightsquigarrow}j
\] 
by $\Xi_{i}(f)=g$ and $\Xi_{i,j}(f)=\hat g$, where $g$ and $\hat g$ are provided in \Cor{reflection} and \Lem{reflection0}, respectively.
\end{definition}
For $i\in S$, let $z^{(i)}=(z^{(i)}_1,\dots,z^{(i)}_D)\in \R^D$ be defined as $z^{(i)}=\psi_i(1/2,\dots,1/2)$, that is, the center of $\psi_i(Q_0)$.
\begin{definition}
For $i,j\in S$, we define a distance $\mathsf{d}(i,j)$ between $i$ and $j$ as $\mathsf{d}(i,j)=l\sum_{k=1}^D|z^{(i)}_k-z^{(j)}_k|$.
\end{definition}
Note that $\mathsf{d}(i,j)=1$ if and only if $i\underset{1}{\leftrightsquigarrow}j$.

We recall the following fact, where condition~(ND) plays the essential role.
\begin{proposition}[{cf.\ \cite[Proposition~2.5]{Ka10}}]\label{prop:NDH}
Let $C$ be a $D$-dimensional cube with side length $2/l$ that is a union of some $2^D$ elements of $\sC$.
We define $T\subset S$ as $T=\{i\in S\mid K_i\subset C\}$.
Then, for each $i,j\in T$, there exists a sequence $\{n(k)\}_{k=0}^{\mathsf{d}(i,j)}$ of elements of $T$ such that $n(0)=i$, $n(\mathsf{d}(i,j))=j$, and $n(k-1)\underset{1}{\leftrightsquigarrow}n(k)$ for $k=1,2,\dots, \mathsf{d}(i,j)$.
\end{proposition}
Now, we prove \Prop{compatible}.
\begin{theopargself}
\begin{proof}[of \Prop{compatible}]
By induction, it is sufficient to prove that $\cF=\cF^S$.
Take $f\in \cF^S$. In order to prove that $f\in \cF$, it suffices to show that $\limsup_{\dl\to0}E_\dl(f)<\infty$.
For $\dl\in(0,1/l)$, we have
\begin{align*}
E_\dl(f)
=\sum_{i\in S}\sum_{j\in S} E_\dl(f,K_i,K_j)
=\sum_{i\in S} E_\dl(f,K_i)+\sum_{i,j\in S,\ i\ne j,\ K_i\cap K_j\ne\emptyset}E_\dl(f,K_i,K_j).
\end{align*}
Since $E_\dl(f,K_i)\lesssim E_{l\dl}(\psi_i^* f)\lesssim \cE(\psi_i^*f)$, we have $\limsup_{\dl\to0}E_\dl(f,K_i)\lesssim\cE(\psi_i^* f)$.
Therefore, it is sufficient to show that $\limsup_{\dl\to0}E_\dl(f,K_i,K_j)<\infty$ for $i,j\in S$ such that $i\ne j$ and $K_i\cap K_j\ne\emptyset$.
Hereafter, we fix such $i$ and $j$.

We can take a $D$-dimensional cube $C$ with side length $2/l$ such that $C$ is a union of some $2^D$ elements of $\sC$ and $K_i\cup K_j\subset C$.
Take a subset $T$ of $S$ as in \Prop{NDH}. Then, from \Prop{NDH}, there exists a sequence $\{n(k)\}_{k=0}^N$ in $T$, where $N=\mathsf{d}(i,j)$, such that $n(0)=i$, $n(N)=j$, and $n(k-1)\underset{1}{\leftrightsquigarrow}n(k)$ for $k=1,2,\dots, N$.
We note that $N$ is equal to the number of $\a\in\{1,\dots,D\}$ such that the $\a$-th coordinates of the centers of $K_i$ and $K_j$, that is, $z^{(i)}_\a$ and $z^{(j)}_\a$, are different.
In particular, for each $k=1,\dots, N$, there exists a unique $\a(k)\in\{1,\dots,D\}$ such that $z^{(n(k-1))}_{\a(k)}\ne z^{(n(k))}_{\a(k)}$.
Then, $z^{(n(0))}_{\a(k)}=z^{(n(1))}_{\a(k)}=\dots=z^{(n(k-1))}_{\a(k)}$ and $z^{(n(k))}_{\a(k)}=z^{(n(k+1))}_{\a(k)}=\dots=z^{(n(N))}_{\a(k)}$; there is no leeway to change a fixed coordinate more than once.
This in particular implies that $\bigcup_{s=0}^{k-1} K_{n(s)}\subset G_{n(k-1),n(k)}$ and $\bigcup_{s=k}^{N} K_{n(s)}\subset G_{n(k),n(k-1)}$
(see the description before \Lem{reflection0} for the definition of $G_{\cdot,\cdot}$).

Keeping \Defn{Xi} in mind, we define $h_k\in \cF$, $k=0,1,\dots,N$, inductively by
\[
 h_0=\Xi_{i}(f)\quad\text{and}\quad
 h_k=\Xi_{n(k),n(k-1)}\biggl(f-\sum_{s=0}^{k-1}h_s\biggr)\quad
 \text{for }k=1,\dots, N.
\]
Based on the above observation, we can prove by mathematical induction that $f-\sum_{s=0}^{k}h_s=0$ on $\bigcup_{s=0}^k K_{n(s)}$ for every $k=0,1,\dots, N$.
Denoting $\sum_{s=0}^{N}h_s$ by $h\in\cF$, we have $f=h$ on $\bigcup_{s=0}^N K_{n(s)}(\supset K_i\cup K_j)$.
Therefore, 
\[
  E_\dl(f,K_i,K_j)=E_\dl(h,K_i,K_j)\le E_\dl(h)\lesssim \cE(h),
\]
which implies that $\limsup_{\dl\to 0}E_\dl(f,K_i,K_j)<\infty$.
\qed\end{proof}
\end{theopargself}
For the proof of \Prop{key'} in the next subsection, we study some properties of functions that are harmonic on subsets of $K$ and other related function spaces.
From \Defn{space} to \Lem{basicH} stated below, $m$ is a fixed natural number and $A$ is a subset of $W_m$.
\begin{definition}\label{def:space}
We define closed subspaces $\cF^0_A$ and $\cH(A)$ of $\cF$ as
\begin{align*}
  \cF^0_A&=\{f\in\cF\mid  f=0 \ \mu\text{-a.e.\ on }K_{W_m\setminus A}\},\\
  \cH(A)&=\{h\in\cF\mid \cE(h)\le \cE(h+g)\text{ for all }g\in \cF^0_A\}.
\end{align*}
\end{definition}
Note that the inclusion $\cF_A^0\subset\cFD$ does not necessarily hold if $K_A\cap K^\partial\ne\emptyset$.
The following lemma is a variant of Lemmas~\ref{lem:harmonic} and \ref{lem:basicharmonic} and its proof is omitted.
\begin{lemma}\label{lem:harmonicA}
\begin{enumerate}
\item For $h\in\cF$, $h\in \cH(A)$ if and only if $\cE(h,g)=0$ for all $g\in \cF^0_A$.
\item For any $f\in\cH(A)$ and $w\in A$, $\psi_w^* f$ belongs to $\cH$.
\end{enumerate}
\end{lemma}
 The following is proved as in \cite[Lemma~3.5]{Hi05}; we provide a proof for readers' convenience.\begin{lemma}\label{lem:poincareA}
Suppose that $A\ne W_m$.
Then, there exists some constant $c_{5.2}>0$ (depending on $A$) such that 
$
  \|f\|_{L^2(K,\mu)}^2\le c_{5.2}\cE(f)$
for all $f\in\cF^0_A$.
\end{lemma}
\begin{proof}
 Let $f\in\cF^0_A$.
  From Chebyshev's inequality and \Eq{Npoincare}, for $b>0$,
  \begin{equation}\label{eq:Chebyshev}
  \mu\left(
  \left\{\left| f- \int_K f\,d\mu\right|>b
  \right\}\right)
  \le \frac{1}{b^2}\left\|f- \int_K f\,d\mu\right\|_{L^2(K,\mu)}^2
  \le \frac{c_{4.1}}{b^2}\cE(f).
  \end{equation}
Let $a=\mu(K\setminus K_A)>0$ and $b=(2c_{4.1}\cE(f)/a)^{1/2}$. Then, the last term of \Eq{Chebyshev} is less than $a$. Since $f=0$ on $K\setminus K_A$, $\left|\int_K f\,d\mu\right|$ must be less than or equal to $b$.
  Therefore, 
  \[
    \|f\|_{L^2(K,\mu)}^2 
    = \left\|f-\int_K f\,d\mu\right\|_{L^2(K,\mu)}^2 
        +\left|\int_K f\,d\mu\right|^2 
    \le c_{4.1} \cE(f)+\frac{2c_{4.1}}{a}\cE(f).\qedhere
  \]
\end{proof}
\begin{lemma}\label{lem:Hw}
Suppose that $A\ne W_m$.
Then, for each $f\in \cF$, there exists a unique function $H_Af\in\cH(A)$ such that $H_Af=f$ on $K_{W_m\setminus A}$.
\end{lemma}
\begin{proof}
This is proved by a standard argument.
Let $\cG=\{\hat f\in\cF\mid \hat f-f\in \cF^0_{A}\}$.
Take a sequence $\{\hat f_n\}$ from $\cG$ such that $\cE(\hat f_n)$ decreases to $\inf\{\cE(\hat f)\mid \hat f\in\cG\}=:a$ as $n\to\infty$.
Then, we have
\begin{align*}
\|\hat f_n\|_{L^2(K,\mu)}
&\le \|\hat f_n-f\|_{L^2(K,\mu)}+\|f\|_{L^2(K,\mu)}\\
&\le\sqrt{c_{5.2}}\cE(\hat f_n-f)^{1/2}+\|f\|_{L^2(K,\mu)} \quad\text{(from \Lem{poincareA})}\\
&\le \sqrt{c_{5.2}}\{\cE(\hat f_n)^{1/2}+\cE(f)^{1/2}\}+\|f\|_{L^2(K,\mu)}.
\end{align*}
Therefore, $\{\hat f_n\}$ is bounded in $\cF$.
A weak limit point $\hat f_\infty$ of $\{\hat f_n\}$ in $\cF$ belongs to $\cG$ and attains the infimum of $\inf\{\cE(\hat f)\mid \hat f\in\cG\}$, i.e., $\cE(\hat f_\infty)=a$.
Thus, $\hat f_\infty\in\cH(A)$ and we can take $\hat f_\infty$ as $H_A f$.
Uniqueness follows from the strict convexity of $\cE(\cdot)$.
More precisely speaking, if another $f'$ attains the infimum, then $\hat f_\infty-f'\in \cF^0_A$ and \Lem{poincareA} implies
\[
c_{5.2}^{-1}\|{\hat f_\infty-f'}\|_{L^2(K,\mu)}^2
\le\cE({\hat f_\infty-f'})
=2\cE(\hat f_\infty)+2\cE(f')-4\cE\left((\hat f_\infty+f')/2\right)\le0.
\]
Therefore, $f'=\hat f_\infty$.
\qed\end{proof}
From this lemma, we can define a bounded linear map $H_A\colon \cF\ni f\mapsto H_Af\in \cF$.
The following lemma is also proved in a standard manner.
\begin{lemma}\label{lem:basicH}
Suppose that $A\ne W_m$.
Let $f,f_1,f_2\in\cF$.
\begin{enumerate}
\item If $f_1=f_2$ on $K_A$, then $H_A f_1=H_A f_2$ on $K_A$.
\item It holds that
\[
\muessinf_{x\in K_A}f(x)\le \muessinf_{x\in K_A} H_A f(x)\le \muesssup_{x\in K_A} H_A f(x)\le \muesssup_{x\in K_A} f(x).
\]
\end{enumerate}
\end{lemma}
\begin{proof}
(i) Since $f_1-f_2=0$ on $K_A$, $\cE^A(f_1-f_2)=0$.
Therefore, $f_1-f_2$ is the minimizer of $\inf\{\cE(\hat f)\mid \hat f-(f_1-f_2)\in \cF^0_A\}$. This implies that $H_A(f_1-f_2)=f_1-f_2$.
From the linearity of $H_A$, $H_A f_1-H_A f_2=0$ on $K_A$.

(ii) Suppose that $f\le b$ $\mu$-a.e.\ on $K_A$ for $b\in\R$.
Let $\hat f=f\wg b\in\cF$.
Since $f=\hat f$ on $K_A$, $H_Af=H_A \hat f$ on $K_A$ from (i).
Since $H_A\hat f-\hat f\in\cF^0_A$, $b-\hat f\in \cF$, and $b-\hat f\ge0$, we have
$(H_A\hat f)\wg b-\hat f=(H_A \hat f-\hat f)\wg (b-\hat f)\in \cF^0_A$.
Moreover, we have $\cE((H_A \hat f)\wg b)\le \cE(H_A\hat f)$ by the Markov property of $(\cE,\cF)$.
Thus, $(H_A \hat f)\wg b=H_A \hat f$, which implies that $H_A\hat f\le b$.
Therefore, $H_A f=H_A \hat f\le b$ on $K_A$.
This implies the last inequality.
By considering $-f$ in place of $f$, we obtain the first inequality.
The second inequality is evident.
 \qed\end{proof}
Hereafter, in most cases, we use the map $H_A$ for $A=\cN_3(w)$ with $w\in W_*$. (See \Defn{cell} for the definition of $\cN_n(w)$.)
\begin{definition}\label{def:surface}
For $i\in \{1,\dots,D\}$ and $j\in\{0,1\}$, we define
\[
K^\partial_{i,j}=\{x=(x_1,\dots,x_D)\in K\mid x_i=j\}.
\]
A subset $\cFH$ of $\cF$ including $\cFD$ is defined as
\[
\cFH=\bigl\{f\in\cF\bigm| \text{$\tilde f=0$ q.e.\ on $K^\partial_{i,j}$ for some $i\in\{1,\dots,D\}$ and some $j\in\{0,1\}$}\bigr\}.
\]
\end{definition}
We note that $\bigcup_{i=1}^D \bigcup_{j=0}^1 K^\partial_{i,j}=K^\partial$.
\begin{lemma}\label{lem:poincareH}
There exists some constant $c_{5.3}>0$ such that 
$
  \|f\|_{L^2(K,\mu)}^2\le c_{5.3}\cE(f)$
  for all $f\in\cFH$.
\end{lemma}
\begin{proof}
From \Eq{reflection}, it suffices to consider the case when $\tilde f=0$ q.e.\ on $K^\partial_{1,1}$.
Let us recall the folding map $\ph$ in \Defn{periodic}.
Let $K'=\{x=(x_1,\dots,x_D)\in K\mid x_1\le 1/l\}$.
From \Lem{reflection0}, the function $g$ defined as $g(x)=f(l\cdot\ph(x))\cdot \bfone_{K'}(x)$ belongs to $\cF$.
Then, it holds that
$\|g\|_{L^2(K,\mu)}^2\le c_{5.2}\cE(g)$
 from \Lem{poincareA}.
Since $\|g\|_{L^2(K,\mu)}=c_{5.4}\|f\|_{L^2(K,\mu)}$ and $\cE(g)=c_{5.5}\cE(f)$ for some positive constants $c_{5.4}$ and $c_{5.5}$ that are independent of $f$, we complete the proof.
\qed\end{proof}
\begin{definition}\label{def:shape}
Let $A\subset W_m$ and $A'\subset W_{m'}$ for $m,m'\in\Z_+$.
We say that $K_A$ and $K_{A'}$ have the same shape and write $K_A\sim K_{A'}$ if there exists a similitude $\xi(x)=l^{m-m'}x+b$ with some $b\in (l^{-m'}\Z)^D$ such that $\xi(K_A)=K_{A'}$ and $\xi(K_A\cap K^\partial)=K_{A'}\cap K^\partial$. 
\end{definition}
It is evident that $\sim$ is an equivalence relation on the set $\{K_A\mid A\subset W_m\text{ for some } m\in\Z_+\}$.

The following Lemmas~\ref{lem:7'}--\ref{lem:elliptic} are used only to prove \Prop{key1} stated below.
\begin{lemma}\label{lem:7'}
There exists a positive constant $c_{5.6}$ such that for any $w\in W_*$ and $f\in \cF^0_{\cN_3(w)}\cap \cFD$,
\[
  \|f\|_{L^2(K,\mu)}^2=\int_{K_{\cN_3(w)}}f^2\,d\mu
  \le c_{5.6}(r/M)^{|w|}\cE^{\cN_3(w)}(f)
  = c_{5.6}(r/M)^{|w|}\cE(f).
\]
\end{lemma}
\begin{proof}
It is sufficient to prove the inequality in the equation described above.
Let $\hat w\in W_*$ and $\hat f \in \cF^0_{\cN_3(\hat w)}\cap\cFD$. Then, from \Lem{poincareH}, 
\begin{equation}\label{eq:c73}
\int_{K_{\cN_3(\hat w)}}\hat f^2\,d\mu
\le c_{5.3}\cE(\hat f)
=c_{5.3}\cE^{\cN_3(\hat w)}(\hat f).
\end{equation}
Next, let $w\in W_*$, $f\in \cF^0_{\cN_3(w)}\cap\cFD$ and suppose that $K_{\cN_3(\hat w)}\sim K_{\cN_3(w)}$.
We take a similitude $\xi$ as in \Defn{shape} such that $\xi(K_{\cN_3(\hat w)})=K_{\cN_3(w)}$ and $\xi(K_{\cN_3(\hat w)}\cap K^\partial)=K_{\cN_3(w)}\cap K^\partial$, and define 
\[
  \hat f(x)=\begin{cases}
  f(\xi(x))&\text{if }x\in K_{\cN_3(\hat w)},\\
  0&\text{otherwise.}
  \end{cases}
\]
Then, $\hat f\in\cF^0_{\cN_3(\hat w)}\cap \cFD$ from \Prop{compatible} and
\begin{align*}
\int_{K_{\cN_3(w)}} f^2\,d\mu
&=M^{|\hat w|-|w|}\int_{K_{\cN_3(\hat w)}} \hat f^2\,d\mu\\
&\le {c_{5.3}}M^{|\hat w|-|w|}\cE^{\cN_3(\hat w)}(\hat f)\qquad\text{(from \Eq{c73})}\\
&= {c_{5.3}}M^{|\hat w|-|w|}r^{|w|-|\hat w|}\cE^{\cN_3(w)}(f)\quad\text{(from \Eq{cEAdef})}\\
&=c_{5.3}(M/r)^{|\hat w|}\cdot (r/M)^{|w|}\cE^{\cN_3(w)}(f).
\end{align*}
Since the number of the equivalent classes of $\{K_{\cN_3(w)}\mid w\in W_*\}$ with respect to $\sim$ is finite, we obtain the assertion.
\qed\end{proof}
\begin{lemma}\label{lem:Hwf}
There exists a positive constant $c_{5.7}$ such that for any $f\in \cF$ and $w\in W_*$ with $K_{\cN_3(w)}\cap K^\partial=\emptyset$, it holds that
\begin{equation}\label{eq:Hwf2}
\cE^{\cN_3(w)}(H_{\cN_3(w)} f)\le \cE^{\cN_3(w)}(f)
\end{equation}
and
\begin{align}\label{eq:Hwf1}
\biggl(\int_{K_{\cN_3(w)}}(H_{\cN_3(w)} f)^2\,d\mu\biggr)^{1/2}
\le c_{5.7}\bigl((r/M)^{|w|}\cE^{\cN_3(w)}(f)\bigr)^{1/2}
+\biggl(\int_{K_{\cN_3(w)}}f^2\,d\mu\biggr)^{1/2}.
\end{align}
\end{lemma}
\begin{proof}
Equation~\Eq{Hwf2} is evident. We prove \Eq{Hwf1}.
Since $H_{\cN_3(w)} f-f\in \cF^0_{\cN_3(w)}\cap \cFD$, from \Lem{7'},
\begin{align*}
&\biggl(\int_{K_{\cN_3(w)}}(H_{\cN_3(w)} f)^2\,d\mu\biggr)^{1/2}
-\biggl(\int_{K_{\cN_3(w)}}f^2\,d\mu\biggr)^{1/2}\\
&\le\biggl(\int_{K_{\cN_3(w)}}(H_{\cN_3(w)} f-f)^2\,d\mu\biggr)^{1/2}\\
&\le \bigl(c_{5.6}(r/M)^{|w|}\cE^{\cN_3(w)}(H_{\cN_3(w)}f-f)\bigr)^{1/2}\\
&\le 2\bigl(c_{5.6}(r/M)^{|w|}\cE^{\cN_3(w)}(f)\bigr)^{1/2}.
\end{align*}
Here, we used \Eq{Hwf2} in the last inequality.
\qed\end{proof}
\begin{lemma}\label{lem:poincareD}
There exists a positive constant $c_{5.8}$ such that for any $f\in \cF$ and $w\in W_*$ with $K_{\cN_3(w)}\cap K^\partial=\emptyset$,\begin{equation}\label{eq:poincareD3}
  \int_{K_{\cN_3(w)}}\left|f(x)-\mint_{K_w}f\,d\mu\right|^2\,\mu(dx)
  \le c_{5.8}(r/M)^{|w|}\cE^{\cN_3(w)}(f).
\end{equation}
\end{lemma}
\begin{proof}
Let $m=|w|$.
We write $\cN_3(w)=\{v_1,\dots,v_s\}$. Here, $s$ is the cardinality of $\cN_3(w)$, which does not exceed $7^D$.
From the assumption of the nondiagonality of $K$, we can renumber the indices such that the following hold:
\begin{itemize}
\item $v_1=w$;
\item for any $i\ge2$, there exists $j<i$ such that $v_i\underset{m}{\leftrightsquigarrow}v_j$.
\end{itemize}
First, we prove by mathematical induction that 
\begin{equation}\label{eq:induction}
  \int_{K_{v_i}}\left|f(x)-\mint_{K_w}f\,d\mu\right|^2\,\mu(dx)
  \le c_i(r/M)^m\cE^{\cN_3(w)}(f)
\end{equation}
for $i=1,\dots, s$ with $c_i=(\sqrt{c_{4.1}}+2(i-1)\sqrt{c_{5.3}})^2$.
When $i=1$, we have
\begin{align*}
\text{LHS of \Eq{induction}}
&=M^{-m}\int_K\left|\psi_w^* f(x)-\int_K \psi_w^* f\,d\mu\right|^2\,\mu(dx)\\
&\le c_{4.1} M^{-m}\cE(\psi_w^* f)={c_{4.1}} M^{-m}r^m \cE^{\{w\}}(f)
\quad\text{(from \Eq{Npoincare} and \Eq{cEAdef})}\\
&\le {c_{4.1}}(r/M)^m\cE^{\cN_3(w)}(f).
\end{align*}
Therefore, \Eq{induction} holds for $i=1$.
Supposing \Eq{induction} holds for $i=1,\dots,k$ with $k<s$, we prove \Eq{induction} for $i=k+1$.
Take $j$ such that $j\le k$ and $v_{k+1}\underset{m}{\leftrightsquigarrow}v_j$.
Let $\Gm\colon\R^D\to\R^D$ be the reflection map with respect to the $(D-1)$-dimensional hyperplane containing $K_{v_{k+1}}\cap K_{v_j}$.
Define a function $\hat f$ on $K_{v_{k+1}}$ as $\hat f(x)=f(\Gm(x))$.
Then,
\begin{align*}
&\biggl(\int_{K_{v_{k+1}}}\biggl|f(x)-\mint_{K_w}f\,d\mu\biggr|^2\,\mu(dx)\biggr)^{1/2}\\
&\le \biggl(\int_{K_{v_{k+1}}}\biggl|\hat f(x)-\mint_{K_w}f\,d\mu\biggr|^2\,\mu(dx)\biggr)^{1/2}
+\biggl(\int_{K_{v_{k+1}}}\left|f(x)-\hat f(x)\right|^2\,\mu(dx)\biggr)^{1/2}\\
&=\biggl(\int_{K_{v_{j}}}\biggl|f(x)-\mint_{K_w}f\,d\mu\biggr|^2\,\mu(dx)\biggr)^{1/2}+M^{-m/2}\bigl\|\psi_{v_{k+1}}^*(f-\hat f)\bigr\|_{L^2(K,\mu)},
\end{align*}
since $(\Gm|_{K_{v_{k+1}}})_*(\mu|_{K_{v_{k+1}}})=\mu|_{K_{v_j}}$.
The first term is dominated by $\bigl(c_j(r/M)^m\cE^{\cN_3(w)}(f)\bigr)^{1/2}$ from the induction hypothesis.
Since $\psi_{v_{k+1}}^* (f-\hat f)$ belongs to $\cFH$, from \Lem{poincareH}, the second term is dominated by 
\begin{align*}
 M^{-m/2}\sqrt{c_{5.3}\cE(\psi_{v_{k+1}}^*(f-\hat f))}
&\le \sqrt{c_{5.3}}M^{-m/2}\bigl(\cE(\psi_{v_{k+1}}^* f)^{1/2}+\cE(\psi_{v_{k+1}}^* \hat f)^{1/2}\bigr)\\
&= \sqrt{c_{5.3}}M^{-m/2}\bigl(\cE(\psi_{v_{k+1}}^* f)^{1/2}+\cE(\psi_{v_{j}}^* f)^{1/2}\bigr)\\
&= \sqrt{c_{5.3}}M^{-m/2} \bigl(r^{m/2}\cE^{\{v_{k+1}\}}(f)^{1/2}+r^{m/2}\cE^{\{v_j\}}(f)^{1/2}\bigr)\\
&\le 2\sqrt{c_{5.3}}(r/M)^{m/2}\cE^{\cN_3(w)}(f)^{1/2}.
\end{align*}
Therefore,
\begin{align*}
\int_{K_{v_{k+1}}}\left|f(x)-\mint_{K_w}f\,d\mu\right|^2\,\mu(dx)
&\le(\sqrt{c_j}+2\sqrt{c_{5.3}})^2 (r/M)^{m}\cE^{\cN_3(w)}(f)\\
&=c_{j+1}(r/M)^{m}\cE^{\cN_3(w)}(f)
\le c_{k+1}(r/M)^{m}\cE^{\cN_3(w)}(f);
\end{align*}
thus \Eq{induction} holds for $i=k+1$.

Now, by summing up \Eq{induction} for $i=1,\dots,s$, we obtain \Eq{poincareD3} with 
$c_{5.8}=\sum_{i=1}^{7^D}\bigl(\sqrt{c_{4.1}}+2(i-1)\sqrt{c_{5.3}}\bigr)^2$.
\qed\end{proof}
\begin{lemma}\label{lem:elliptic}
There exists a constant $c_{5.9}>0$ such that for any $w\in W_*$ with $K_{\cN_3(w)}\cap K^\partial=\emptyset$ and $h\in \cH(\cN_3(w))$ with $h\ge0$ $\mu$-a.e., the following inequalities hold:
\begin{align*}
\muesssup_{\makebox[0pt]{\scriptsize$x\!\in \!K_{\cN_2(w)}$}}h(x)
&\le c_{5.9}\,\muessinf_{\makebox[0pt]{\scriptsize$x\!\in \!K_{\cN_2(w)}$}}\,h(x)\\
&\le c_{5.9}\biggl(\mint_{K_{\cN_2(w)}}h^2\,d\mu\biggr)^{1/2}
\le c_{5.9}\biggl(M^{|w|}\int_{K_{\cN_2(w)}}h^2\,d\mu\biggr)^{1/2}.
\end{align*}
\end{lemma}
\begin{proof}
The first inequality follows from the elliptic Harnack inequality; this is implied by the parabolic Harnack inequality, which is equivalent to \Eq{Aronson} in our context. See, e.g., \cite{BBK06}, \cite{BB99}, \cite[Theorem~4.6.5]{FOT}, and \cite[Proposition~2.9]{HK06} for further details.
The remaining inequalities are evident.
\qed\end{proof}
\begin{proposition}\label{prop:key1}
There exists a constant $c_{5.10}>0$ such that for any $w\in W_*$ with $K_{\cN_3(w)}\cap K^\partial=\emptyset$ and $h\in \cH(\cN_3(w))$ with $\int_{K_w}h\,d\mu=0$,
\[
\muesssup_{\makebox[0pt]{\scriptsize$x\!\in \!K_{\cN_2(w)}$}}\,\left|h(x)\right|\le c_{5.10}\bigl(r^{|w|}\cE^{\cN_3(w)}(h)\bigr)^{1/2}.
\]
\end{proposition}
\begin{proof}
Let $w\in W_*$ and $h\in \cH(\cN_3(w))$ as stated above.
We define $h_1=h\vee0$ and $h_2=(-h)\vee0$.
Since $h=H_{\cN_3(w)} h_1-H_{\cN_3(w)} h_2$ and $H_{\cN_3(w)} h_j\ge0$ on $K$ for $j=1,2$ from \Lem{basicH}, we have
\begin{align*}
&\muesssup_{\makebox[0pt]{\scriptsize$x\!\in \!K_{\cN_2(w)}$}}\,|h(x)|
\le\muesssup_{\makebox[0pt]{\scriptsize$x\!\in \!K_{\cN_2(w)}$}}\,(H_{\cN_3(w)} h_1)(x)
+\muesssup_{\makebox[0pt]{\scriptsize$x\!\in \!K_{\cN_2(w)}$}}\,(H_{\cN_3(w)} h_2)(x)\\
&\le \sum_{j=1}^2 c_{5.9}\biggl(M^{|w|}\int_{K_{\cN_2(w)}}(H_{\cN_3(w)} h_j)^2\,d\mu\biggr)^{1/2}
\quad\text{(from \Lem{elliptic})}\\
&\le \sum_{j=1}^2 c_{5.11}\biggl\{\bigl(r^{|w|}\cE^{\cN_3(w)}(h_j)\bigr)^{1/2}+\biggl(M^{|w|}\int_{K_{\cN_3(w)}}h_j^2\,d\mu\biggr)^{1/2}\biggr\}\quad\text{(from \Lem{Hwf})}\\
&\le 2c_{5.11}\biggl\{\bigl(r^{|w|}\cE^{\cN_3(w)}(h)\bigr)^{1/2}+\biggl(M^{|w|}\int_{K_{\cN_3(w)}}h^2\,d\mu\biggr)^{1/2}\biggr\}\\
&\le 2c_{5.11}\Bigl\{\bigl(r^{|w|}\cE^{\cN_3(w)}(h)\bigr)^{1/2}+\bigl(c_{5.8}r^{|w|}\cE^{\cN_3(w)}(h)\bigr)^{1/2}\Bigr\}.
\end{align*}
Here, the last inequality follows from the assumption $\int_{K_w}h\,d\mu=0$ and \Lem{poincareD}.
This completes the proof.
\qed\end{proof}
\subsection{Proof of Propositions~\ref{prop:A7}\ and \ref{prop:key}}
The ideas of the proof are based on \cite{Hi05,HK06}.
First, we prove \Prop{key}.
By taking \Eq{energyineq} into consideration, it is sufficient to prove the following.
\begin{proposition}\label{prop:key'}
Under the same assumptions as those of \Prop{key}, with \Eq{N3} replaced by
\begin{equation}\label{eq:N3'}
 \sup_n r^{|w_n|}\cE^{\cN_3(w_n)}(h_n)<\infty,
\end{equation}
the same conclusion of \Prop{key} holds.
\end{proposition}
\begin{proof}
Let us recall the concept of ``same shape'' in \Defn{shape}.
Since there are only finite kinds of shapes of $K_{\cN_3(w)}$ $(w\in W_*)$, that is, the cardinality of $\{K_{\cN_3(w)}\mid w\in W_*\}/\!\sim$ is finite, we may assume that all $K_{\cN_3(w_n)}$, $n\in\N$, are of the same shape by taking a suitable subsequence.

We let $u=w_1$.
Take $g\in \cF^0_{\cN_2(u)}\cap C(K)$ such that $0\le g\le 1$ on $K$ and $g=1$ on $K_{\cN_1(u)}$.
Let $c_{5.12}=\max\{\cE(\psi_v^* g)^{1/2}\mid v\in \cN_2(u)\}$.

Let $n\in \N$, and take a similitude $\xi_n$ on $\R^D$ as in \Defn{shape} such that $\xi_n(K_{\cN_3(u)})=K_{\cN_3(w_n)}$.
Set
\[
  f_n(x)=\begin{cases}
  g(x)h_n(\xi_n(x))& \text{if }x\in K_{\cN_3(u)},\\
  0&\text{otherwise.}
  \end{cases}
\]
Then, $f_n\in \cF^0_{\cN_2(u)}$ since $h_n$ is bounded on $K_{\cN_2(w_n)}$ from \Prop{key1}.
We have
\begin{align*}
\cE(f_n)^{1/2}
&=\biggl(\sum_{v\in \cN_2(u)} r^{-|u|}\cE(\psi_v^* f_n)\biggr)^{1/2}
\le\sum_{v\in \cN_2(u)} r^{-|u|/2}\cE(\psi_v^* f_n)^{1/2}\\
&\le r^{-|u|/2}\sum_{v\in \cN_2(u)} \bigl\{\cE(\psi_v^* g)^{1/2}\|\psi_v^*(h_n\circ \xi_n)\|_{L^\infty(K,\mu)}\\
&\hspace{4em}+\cE(\psi_v^*(h_n\circ \xi_n))^{1/2}\|\psi_v^* g\|_{L^\infty(K,\mu)}\bigr\}\quad\text{(from \Prop{cEfg})}\\
&\le r^{-|u|/2}\sum_{v\in \cN_2(u)} \Bigl\{\cE(\psi_v^* g)^{1/2}\muesssup_{\makebox[0pt]{\scriptsize $x\!\in\! K_{\cN_2(w_n)}$}}\,|h_n(x)|+\cE(\psi_v^*(h_n\circ \xi_n))^{1/2}\Bigr\}\\
&\le r^{-|u|/2}\sum_{v\in \cN_2(u)}\Bigl\{c_{5.12}c_{5.10}\bigl(r^{|w_n|}\cE^{\cN_3(w_n)}(h_n)\bigr)^{1/2}
+\bigl(r^{|w_n|}\cE^{\{v'\}}(h_n)\bigr)^{1/2}\Bigr\}\\
&\qquad\left(\;\parbox{0.7\textwidth}{from \Prop{key1}, and $v'$ denotes a unique element of $\cN_2(w_n)$ such that $K_{v'}=\xi_n(K_v)$}\;\right)\\
&\le c_{5.13}r^{-|u|/2}\bigl(r^{|w_n|}\cE^{\cN_3(w_n)}(h_n)\bigr)^{1/2}
\end{align*}
and
\begin{align*}
\|f_n\|_{L^\infty(K,\mu)}^2
&\le\muesssup_{\makebox[0pt]{\scriptsize $x\!\in\! K_{\cN_2(w_n)}$}}\,|h_n(x)|^2
\le c_{5.10}^2 r^{|w_n|}\cE^{\cN_3(w_n)}(h_n)
\quad\text{(from \Prop{key1})}.
\end{align*}
Therefore, $\{f_n\}_{n=1}^\infty$ is bounded both in $\cF$ and in $L^\infty(K,\mu)$ under assumption~\Eq{N3'}.
We can take a suitable subsequence of $\{f_n\}_{n=1}^\infty$, denoted by the same notation, converging to some $f_\infty$ weakly in $\cF$.
It is evident that $f_\infty\in\cF^0_{\cN_2(u)}\cap L^\infty(K,\mu)$.
Since $f_n=h_n\circ\xi_n$ on $K_{\cN_1(u)}$, it holds that $f_n\in\cH(\cN_1(u))$ for all $n$. This implies that $f_\infty\in\cH(\cN_1(u))$.

We define $\hat f_n=f_n-f_\infty$ for $n\in\N$.
Then, $\hat f_n\in \cH(\cN_1(u))\cap L^\infty(K,\mu)$ and $\hat f_n\to 0$ weakly in $\cF$.
Since $\cF$ is compactly imbedded in $L^2(K,\mu)$, which is equivalent to the statement that the resolvent operators are compact ones on $L^2(K,\mu)$, $\hat f_n\to0$ strongly in $L^2(K,\mu)$.
\Prop{cEfg} implies that
\[
\cE(f_n^2)^{1/2}\le 2\cE(f_n)^{1/2}\|f_n\|_{L^\infty(K,\mu)}
\]
and
\[
\cE(f_n f_\infty)^{1/2}\le \cE(f_n)^{1/2}\|f_\infty\|_{L^\infty(K,\mu)}+\cE(f_\infty)^{1/2}\|f_n\|_{L^\infty(K,\mu)},
\]
which are both bounded in $n$.
Then, $\{\hat f_n^2\}_{n=1}^\infty$ is also bounded in $\cF$ since 
\[
\cE(\hat f_n^2)^{1/2}\le \cE(f_n^2)^{1/2}+2\cE(f_n f_\infty)^{1/2}+\cE(f_\infty^2)^{1/2}.
\]
By taking a subsequence if necessary, we may assume that $\{\hat f_n^2\}_{n=1}^\infty$ converges weakly in $\cF$.
This implies that $\{\hat f_n^2\}_{n=1}^\infty$ converges in $L^2(K,\mu)$ from the same reason as described above. 
Then, the limit function has to be $0$.

Now, we take $\hat g\in \cF^0_{\cN_1(u)}\cap C(K)$ such that $0\le \hat g\le 1$ on $K$ and $\hat g=1$ on $K_u$.
Then, since $\hat f_n\hat g\in \cF^0_{\cN_1(u)}$,
\[
0=2\cE(\hat f_n,\hat f_n\hat g)=\cE(\hat f_n^2,\hat g)+\int_K \hat g\,d\nu_{\hat f_n},
\]
where the second equality follows from the characterization of the energy measure $\nu_{\hat f_n}$.
From the above argument, $\cE(\hat f_n^2,\hat g)\to0$ as $n\to\infty$.
We also have
\begin{align*}
\int_K \hat g\,d\nu_{\hat f_n}
&=\sum_{v\in \cN_1(u)}r^{-|u|}\int_K \psi_v^*\hat g\,d\nu_{\psi_v^*\hat f_n}\qquad\text{(from \Lem{energymeas}~(ii))}\\
&\ge r^{-|u|}\int_K \psi_u^*\hat g\,d\nu_{\psi_u^*\hat f_n}
=r^{-|u|}\nu_{\psi_u^*\hat f_n}(K)
=2r^{-|u|}\cE(\psi_u^*\hat f_n).
\end{align*}
Combining these relations, we obtain that $\limsup_{n\to\infty}\cE(\psi_u^*\hat f_n)\le0$, in other words, $\lim_{n\to\infty}\cE(\psi_u^*\hat f_n)=0$.
Therefore, $\psi_u^*\hat f_n \to 0$ in $\cF$ since $\hat f_n\to 0$ in $L^2(K,\mu)$.
This implies that $\psi_{w_n}^* h_n=\psi_u^* f_n\to \psi_u^* f_\infty$ in $\cF$ as $n\to\infty$, which completes the proof.
\qed\end{proof}

Next, we proceed to prove \Prop{A7}.
Let $I\subset S$ be defined by 
\[
  I=\left\{i\in S\;\vrule\; K_i\subset \{(x_1,\dots,x_D)\in \R^D\mid x_D\le 1/l\}\right\}.
\]
For $n\in\N$, we denote the direct product of $n$ copies of $I$ by $I_n$, which is regarded as a subset of $W_n$.
Note that $K_{I_1}\supset K_{I_2}\supset K_{I_3}\supset\cdots$ and $\bigcap_{n=1}^\infty K_{I_n}=K^\partial_{D,0}$ (see \Defn{surface}).
We define
\[
\cK(w;a)=\biggl\{f\in\cH(\cN_3(w))\biggm| \int_{K_{w}}f\,d\mu=0,\ r^{|w|}\cE^{\cN_3(w)}(f)\le a\biggr\}
\]
for $w\in \bigcup_{n=2}^\infty I_{n}$ and $a>0$, and 
\[
\cK=\mbox{the closure of }\biggl\{\psi_w^* f\biggm|w\in \bigcup_{n=2}^\infty I_{n},\ f\in\cK(w;1)\biggr\} \mbox{ in }\cF.
\]
\begin{remark}
Since $l\ge3$, for any $w\in \bigcup_{n=2}^\infty I_{n}$, it holds that
$
  K_{\cN_3(w)}\cap K^\partial_{D,1}=\emptyset$.
Moreover, for each $i=1,\dots,D-1$, there exists $j\in\{0,1\}$ such that
$
  K_{\cN_3(w)}\cap K^\partial_{i,j}=\emptyset$.
\end{remark}

\begin{lemma}\label{lem:cpt}
The set $\cK$ is a compact subset in $\cF$.
\end{lemma}
\begin{proof}
  We fix $u\in W_2$ such that $K_{\cN_1(u)}\cap K^\partial=\emptyset$.
  As in \Defn{periodic}, we define a folding map $\ph^{(2)}\colon [0,1]^D\to[0,1/l^{2}]^D$ as
\[
\ph^{(2)}(x_1,\dots,x_D)=(\hat\ph^{(2)}(x_1),\dots,\hat\ph^{(2)}(x_D)),
\quad (x_1,\dots,x_D)\in[0,1]^D,
\]
where $\hat\ph^{(2)}\colon \R\to\R$ is a periodic function with period $2/l^{2}$ such that $\hat\ph^{(2)}(t)=|t|$ for $t\in[-1/l^{2},1/l^{2}]$.
Let $\ph_u\colon K\to K_u\subset K$ be defined as
\[
\ph_u(x)=\bigl(\ph^{(2)}\big|_{K_u}\bigr)^{-1}(\ph^{(2)}(x)),\quad
x\in K.
\]
This is the folding map based on $K_u$.
For $f\in\cF$,
define $f_u(x)=f(\psi_u^{-1}(\ph_u(x)))$ for $x\in K$.
Then, $f_u\in \cF^{W_2}=\cF$ from \Prop{compatible}.
Now, let $w\in I_n$ with $n\ge 2$ and $f\in \cK(w;1)$.
Let $\cN_{3,u}(w)$ denote a subset of $W_{n+2}$ such that $K_{\cN_{3,u}(w)}$ is the connected component of $\ph_u^{-1}(\psi_u(K_{\cN_3(w)}))$ that includes $K_{u\cdot w}$ (see \Fig{cpt}).
\begin{figure}[t]
\begin{minipage}{0.5\hsize}
\includegraphics[clip]{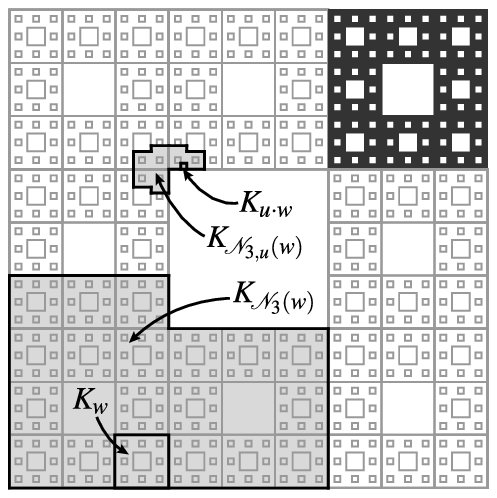}
\caption{Case of $w\in I_2$}\label{fig:cpt}
\end{minipage}
\begin{minipage}{0.48\hsize}
\includegraphics[clip]{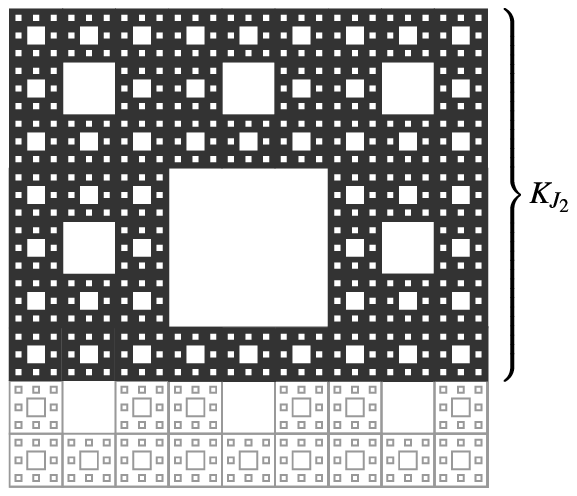}
\caption{Illustration of $K_{J_m}$ for $m=2$}\label{fig:Kjm}
\end{minipage}
\end{figure}
Set $K_{\cN_{3,u}(w)}$ is described as a union of (at most $2^D$) sets that are isometric to $\psi_u(K_{\cN_3(w)})$.
Moreover, $K_{\cN_{3,u}(w)}\supset K_{\cN_3(u\cdot w)}$, $K_{\cN_{3,u}(w)}\cap K^\partial=\emptyset$ and $f_u\in\cH(\cN_{3,u}(w))$.
In particular, $K_{\cN_3(u\cdot w)}\cap K^\partial=\emptyset$ and $f_u\in\cH(\cN_3(u\cdot w))$.
We also have $\psi_{u\cdot w}^*f_u=\psi_w^* f$, $\int_{K_{u\cdot w}} f_u\,d\mu=M^{-2}\int_{K_w}f\,d\mu=0$, and
\begin{align*}
r^{|u\cdot w|}\cE^{\cN_3(u\cdot w)}(f_u)
&=\sum_{v'\in\cN_3(u\cdot w)}\cE(\psi_{v'}^*f_u)
\le 2^D\sum_{v\in\cN_3(w)}\cE(\psi_{v}^* f)\\
&= 2^D r^{|w|}\cE^{\cN_3(w)}(f)\le 2^D.
\end{align*}
In other words, $f\in \cK(w;1)$ implies $f_u\in\cK(u\cdot w;2^D)$.
Therefore, we have
\begin{align}\label{eq:inc}
\biggl\{\psi_w^* f\biggm|w\in \bigcup_{n\ge2}I_{n},\ f\in\cK(w;1)\biggr\}
\subset 
\biggl\{\psi_{u\cdot w}^*g\biggm|w\in \bigcup_{n\ge2}I_{n},\ g\in\cK(u\cdot w;2^D)\biggr\}.
\end{align}
From \Prop{key'}, the right-hand side is relatively compact in $\cF$.
This completes the proof.
(We note that \Prop{key'} cannot be applied directly to the left-hand side of \Eq{inc}, since $K_{\cN_3(w)}\cap K^\partial=\emptyset$ does not necessarily hold.)
\qed\end{proof}
The following claim is stated in \cite{HK06} without an explicit proof.
We provide the proof for completeness.
\begin{lemma}\label{lem:B4}
Let $f\in\cF$. 
If $\cE^{W_m\setminus I_m}(f)=0$ for all $m\in\N$, then $f$ is constant $\mu$-a.e.
\end{lemma}
\begin{proof}
Let $S'\subset S$ be defined by 
$
  S'=\bigl\{i\in S\bigm| K_i\subset \{(x_1,\dots,x_D)\in \R^D\mid x_D\le 1-1/l\}\bigr\}
$.
For $m\ge2$, let $J_m:=W_{m}\setminus (I_{m-1}\cdot S')\subset W_m\setminus I_m$ (see \Fig{Kjm}).
Then, $J_m$ is connected in the following sense: For any $v,w\in J_m$, there exists a sequence $w_0,w_1,\dots,w_k$ in $J_m$ for some $k$ such that $w_0=v$, $w_k=w$, and $w_j\underset{m}{\leftrightsquigarrow}w_{j+1}$ for all $j=0,1,\dots, k-1$.
This is confirmed by the assumptions of connectedness, nondiagonality, and borders inclusion on $K$ (see Section~4.3).
Since $f$ is constant on $K_w$ for each $w\in J_m$, the connectedness of $J_m$ implies that $f$ is constant on $K_{J_m}$.
Since $K\setminus \bigcup_{m=2}^\infty K_{J_m}=K^\partial_{D,0}$ and $\mu(K^\partial_{D,0})=0$, we conclude that $f$ is constant $\mu$-a.e.
\qed\end{proof}
The following proposition states that the energy measures of a class of functions do not concentrate near the boundary uniformly in some sense, which is the key proposition to prove \Prop{A7}.
\begin{proposition}\label{prop:keyp}
There exist $c_0\in(0,1)$ and $m\in\N$ such that for all 
$n\in\Z_+$ and $h\in \cH(I_n)$, 
\[
\cE^{I_{n+m}}(h)\le c_0\cE^{I_{n}}(h).
\]
\end{proposition}
\begin{proof}
We define $C:=\sup_{n\in\Z_+}\max_{w\in I_{n+2}}\# \cN_3(w)\le 7^D$.
Let $n\in\Z_+$ and $w\in I_{n+2}$. 

Let $\dl=1/(4C^2)$. We define
$\cK^\dl=\{f\in\cK\mid \cE(f)\ge\dl\}$.
From \Lem{B4}, for each $f\in\cK^\dl$, there exist $m(f)\in\N$ 
and $a(f)\in(0,1)$ such that $\cE^{I_m}(f)<a(f)\cE(f)$ for all $m\ge m(f)$. 
By continuity, $\cE^{I_m}(g)<a(f)\cE(g)$ for all $m\ge m(f)$ for any $g$ in some neighborhood of $f$ in
$\cF$. Since $\cK^\dl$ is compact in $\cF$ from \Lem{cpt}, there exist $m'\in\N$ and $a'\in(0,1)$ 
such that $\cE^{I_{m'}}(f)<a'\cE(f)$ for every $f\in\cK^\dl$.
Then, 
\begin{equation}\label{eq:dom1}
\cE(f)< a \cE^{W_{m'}\setminus I_{m'}}(f)
\quad\text{for all }f\in\cK^\dl
\end{equation}
with $a=(1-a')^{-1}>1$.

Now, consider $n$ and $h$ in the claim of the proposition.
We note that $h\in\cH(\cN_3(w))$ for any $w\in I_{n+2}$ since $\cN_3(w)\subset I_n\cdot W_2$.
We construct an oriented graph as follows: The vertex set is $I_{n+2}$ and the set $E$ of oriented edges is defined as
\[
E=\left\{(v,w)\in I_{n+2}\times I_{n+2}\;\vrule\;
v\in \cN_3(w),\, \cE^{\{w\}}(h)>0, \mbox{ and }
\cE^{\{w\}}(h)\ge 2C \cE^{\{v\}}(h)\right\}.
\]
This graph does not have any loops.
Let $Y$ be the set of all elements $w$ in $I_{n+2}$ such that $\cE^{\{w\}}(h)>0$ and $w$ is not a source of any edges.
For $w\in Y$, we define 
\begin{align*}
N_0(w)=\{w\},\quad
N_k(w)=\biggl\{v\in I_{n+2}\setminus \bigcup_{j=0}^{k-1}
N_j(w)\biggm| (v,u)\in E \mbox{ for some }u\in N_{k-1}(w)\biggr\}
\end{align*}
for $k=1,2,3,\dots$ inductively, and $N(w)=\bigcup_{k\ge0}N_k(w)$.
It is evident that 
\begin{equation}\label{eq:graph}
I_{n+2}=\bigcup_{w\in Y}N(w)\cup \bigl\{w\in I_{n+2}\bigm|\cE^{\{w\}}(h)=0\bigr\}
\end{equation}
and that for all $k\in\Z_+$,
$
\# N_k(w)\le C^k$ and $\cE^{\{v\}}(h)\le 
(2C)^{-k}\cE^{\{w\}}(h)$ for $v\in N_k(w)$.
Then, for each $w\in Y$,
\begin{equation}\label{eq:dom2}
  \cE^{N(w)}(h)
  = \sum_{k=0}^\infty\sum_{v\in N_k(w)}\cE^{\{v\}}(h)
  \le \sum_{k=0}^\infty C^k (2C)^{-k}\cE^{\{w\}}(h)
  = 2\cE^{\{w\}}(h).
\end{equation}
Suppose $w\in Y$ and $\cE^{\{w\}}(h)\ge \dl\cE^{\cN_3(w)}(h)$. Then, since 
\[
\psi_w^*\biggl(\biggl(h-\mint_{K_{w}}h\,d\mu\biggr)\biggr/\sqrt{ r^{n+2}\cE^{\cN_3(w)}(h)}\biggr)\in\cK^\dl,
\]
(\ref{eq:dom1}) implies that $\cE(\psi_w^* h)< a \cE^{W_{m'}\setminus I_{m'}}(\psi_w^* h)$, that is,
$
  \cE^{\{w\}}(h)< a\cE^{w\cdot (W_{m'}\setminus I_{m'})}(h)$.
(See \Fig{Kw}.)
\begin{figure}[t]
\includegraphics[clip]{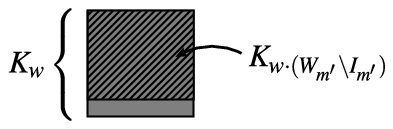}
\caption{Case of $w\in Y$ and $\cE^{\{w\}}(h)\ge \dl\cE^{\cN_3(w)}(h)$}\label{fig:Kw}
\end{figure}

Next, suppose $w\in Y$ and $\cE^{\{w\}}(h)< \dl\cE^{\cN_3(w)}(h)$.
Since $w$ is not a source of any edges, $\cE^{\{v\}}(h)<2C\cE^{\{w\}}(h)$ for every $v\in \cN_3(w)\cap I_{n+2}$.
Then,
\[
  \cE^{\cN_3(w)\cap I_{n+2}}(h)
  < C\cdot 2C\cE^{\{w\}}(h)
  < 2C^2\dl \cE^{\cN_3(w)}(h)
  = (1/2) \cE^{\cN_3(w)}(h),
\]
which implies that 
$\cE^{\cN_3(w)\cap I_{n+2}}(h)
< \cE^{\cN_3(w)\cap ((I_n\cdot W_{2})\setminus I_{n+2})}(h)$ since $\cN_3(w)\subset I_n\cdot W_2$.
In particular,
$
  \cE^{\{w\}}(h)
  <\cE^{\cN_3(w)\cap ((I_n\cdot W_{2})\setminus I_{n+2})}(h)$.
(See \Fig{Kw2}.)
\begin{figure}[t]
\includegraphics[clip]{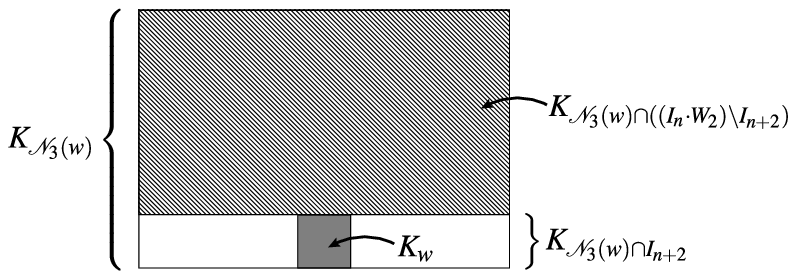}
\caption{Case of $w\in Y$ and $\cE^w(h)< \dl\cE^{\cN_3(w)}(h)$}\label{fig:Kw2}
\end{figure}

Therefore, in any case, for $w\in Y$, we have  
\begin{align}\label{eq:dom3}
  \cE^{\{w\}}(h)
  &< a\cE^{w\cdot (W_{m'}\setminus I_{m'})}(h)\vee\cE^{\cN_3(w)\cap ((I_n\cdot W_{2})\setminus I_{n+2})}(h)\notag\\
  &\le a\cE^{(w\cdot(W_{m'}\setminus I_{m'}))\cup((\cN_3(w)\cap ((I_n\cdot W_{2})\setminus I_{n+2}))\cdot W_{m'})}(h)\notag \\
  &\le a\cE^{(\cN_3(w)\cdot W_{m'})\cap((I_n\cdot W_{m})\setminus I_{n+m})}(h),
\end{align}
where $m=m'+2$ (see \Fig{Kw3}); further, it should be noted that
$I_{n+m}\subset I_{n+2}\cdot W_{m'}$.
\begin{figure}[t]
\includegraphics[clip]{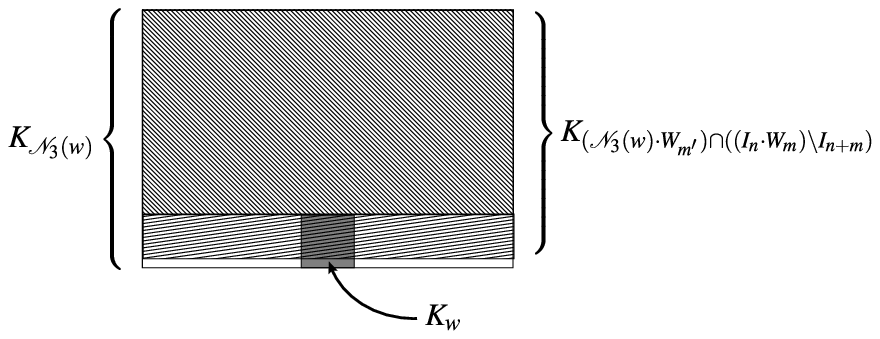}
\caption{Illustration of $K_w$ and $K_{(\cN_3(w)\cdot W_{m'})\cap((I_n\cdot W_{m})\setminus I_{n+m})}$}\label{fig:Kw3}
\end{figure}

Then, we have
\begin{align*}
\cE^{I_{n+m}}(h)
&\le \cE^{I_{n+2}}(h)
\le \sum_{w\in Y} \cE^{N(w)}(h)
\qquad\text{(from \Eq{graph})}\\
&\le 2a\sum_{w\in Y}\cE^{(\cN_3(w)\cdot W_{m'})\cap((I_n\cdot W_{m})\setminus I_{n+m})}(h)
\quad\text{(from \Eq{dom2} and \Eq{dom3})}\\
&\le 2a C'\cE^{(I_n\cdot W_{m})\setminus I_{n+m}}(h)
=2aC'\left(\cE^{I_n}(h)-\cE^{I_{n+m}}(h)\right),
\end{align*}
where 
$
C':=\sup_{n\in\Z_+}\max_{v\in W_{n+2}}\#\{w\in I_{n+2}\mid v\in\cN_{3}(w)\}\le 7^D$.
Hence, the claim of the proposition holds by setting $c_0=2a C'/(1+2a C')$.
\qed\end{proof} 
\begin{proposition}\label{prop:A7'}
For any $f\in\cF$, $\nu_f(K^\partial_{D,0})=0$.
\end{proposition}
\begin{proof}
Take $c_0$ and $m$ as in \Prop{keyp}.
For each $n\in\N$, define $h_n=H_{I_n}(f)\in \cH(I_n)$ (cf.\ \Lem{Hw}).
From \Lem{energymeas}~(ii) and \Prop{keyp}, for $j\in\N$,
\begin{align*}
\frac12\nu_{h_n}(K^\partial_{D,0})
&\le \frac12\sum_{w\in I_{n+jm}}r^{-(n+jm)}\nu_{\psi_w^* h_n}(K)\\
&\quad\qquad\text{(since $\psi_w^{-1}(K^\partial_{D,0})=\emptyset$ for $w\in W_{n+jm}\setminus I_{n+jm}$)}\\
&=\cE^{I_{n+jm}}(h_n)
\le c_0^j\cE^{I_n}(h_n)
\le c_0^j\cE^{I_n}(f)
\le c_0^j\cE(f)
\to 0
\quad \text{as }j\to \infty.
\end{align*}
Therefore, 
\begin{equation}\label{eq:zero}
\nu_{h_n}(K^\partial_{D,0})=0.
\end{equation}
Since $\cE(h_n)\le \cE(f)$ and $h_n=f$ on $K_{W_n\setminus I_n}$ for each $n$, $\{h_n\}_{n=1}^\infty$ is bounded in $\cF$ in view of \Lem{poincareA}.
Moreover, since $\mu(K^\partial_{D,0})=0$ and $\bigcup_{n=1}^\infty K_{W_n\setminus I_n}=K\setminus K^\partial_{D,0}$, $h_n(x)$ converges to $f(x)$ for $\mu$-a.e.\,$x\in K$.
Therefore, $h_n$ converges weakly to $f$ in $\cF$.
In particular, the Ces\`aro mean of a certain subsequence of $\{h_n\}_{n=1}^\infty$ converges to $f$ in $\cF$.
By combining \Eq{triangular}, \Eq{energy}, and \Eq{zero}, we obtain that $\nu_f(K^\partial_{D,0})=0$.
\qed\end{proof}
\begin{theopargself}
\begin{proof}[of \Prop{A7}]
From Propositions~\ref{prop:em} and \ref{prop:A7'}, and \Lem{symmetry}, we conclude that $\nu(K^\partial)=0$.
\qed\end{proof}
\end{theopargself}

\begin{acknowledgements}
The author expresses his gratitude to M.~T.~Barlow, S.~Watanabe, and M.~Yor for their valuable comments on the concept of martingale dimensions.
The author also thanks N.~Kajino for discussion on \Prop{compatible} and the anonymous referees for their careful reading and constructive suggestions.
\end{acknowledgements}



\end{document}